\documentclass[11pt]{amsart}

\usepackage[left=3cm,top=2.5cm,bottom=2.5cm,right=3cm]{geometry}
\usepackage{times}
\usepackage{amssymb}
\usepackage{graphicx,xspace}
\usepackage{epsfig}
\usepackage{enumitem}
\usepackage{xcolor}
\usepackage{tikz}
\usetikzlibrary{decorations.markings}
\usepackage[T1]{fontenc}
\usepackage[utf8]{inputenc}
\usepackage{bm}
\usepackage{hyperref}
\usepackage[capitalise]{cleveref}
\usepackage{todonotes}

\theoremstyle{plain}

\newtheorem{lemma}{Lemma}[section]
\newtheorem{theorem}[lemma]{Theorem}

\newtheorem{proposition}[lemma]{Proposition}

\newtheorem{definition}[lemma]{Definition}

\theoremstyle{remark}
\newtheorem{remark}[lemma]{Remark}

\newtheorem{example}[lemma]{Example}

\newcommand{\RR}{\mathbb{R}}

\newcommand{\N}{\mathcal{N}}

\newcommand{\LLL}{\mathcal{L}}

\newcommand{\MM}{\mathbf{M}}

\newcommand{\NN}{\mathbb{N}}

\def\P{P}

\newcommand{\esper}{\mathbb{E}}

\def\O{\mathcal{O}}
\def\a{\alpha}
\def\ex{\text{ex}}
\def\la{\lambda}
\def\si{\sigma}
\def\Dep{L}
\def\WDep{\tilde{L}}
\def\WDepOne{\tilde{L} \langle 1 \rangle}
\def\WK{\tilde{K}}
\def\eg{{\it e.g.}}
\def\ie{{\it i.e. }}
\def\eps{\varepsilon}
\def\PPP{\mathcal{P}}
\def\III{\mathcal{I}}
\def\bC{\bm{C}}
\def\bD{\bm{D}}
\def\ka{\kappa}
\def\cN{\overline{N}}
\def\R{R}
\def\RPP{H}
\def\uu{\bm{u}}
\def\vv{\bm{v}}
\def\ww{\bm{w}}

\def\tPhi{\widetilde{\Phi}}
\def\bs{\bm{s}}
\def\bt{\bm{t}}
\def\bu{\bm{u}}
\def\bv{\bm{v}}
\def\One{\bm{1}}

\newcommand{\MWST}[1]{\mathcal{M} \big( #1 \big)}

\newcommand{\ev}{\text{ev}}
\newcommand{\odd}{\text{odd}}

\DeclareMathOperator{\pairs}{pairs}
\DeclareMathOperator{\Cr}{Cr}
\DeclareMathOperator{\CC}{CC}
\DeclareMathOperator{\Var}{Var}
\DeclareMathOperator{\Cov}{Cov}

\DeclareMathOperator{\irr}{irr}

\DeclareMathOperator{\Aut}{Aut}

\DeclareMathOperator{\Mset}{\sc MSet}

\usepackage{mathrsfs}

\usepackage{navigator}
\embeddedfile{Variances}{Variances.sage}

\author[V.~Féray]{Valentin Féray}

\address{Institut für Mathematik, Universität Zürich, Winterthurerstrasse 190, 8057 Zürich, Switzerland}
\email{valentin.feray@math.uzh.ch}

\thanks{
VF is partially supported by SNF grant nb 149461 ``Dual combinatorics of Jack polynomials''.}



\title
{Weighted dependency graphs}

\begin{document}

\maketitle

\begin{abstract}
The theory of dependency graphs is a powerful toolbox to prove asymptotic normality
of sums of random variables.
In this article, we introduce a more general notion of weighted dependency graphs
and give normality criteria in this context.
We also provide generic tools to prove that some weighted graph
is a weighted dependency graph for a given family of random variables.

To illustrate the power of the theory,
we give applications to the following objects:
uniform random pair partitions, the random graph model $G(n,M)$, uniform random permutations,
the symmetric simple exclusion process and multilinear statistics on Markov chains.
The application to random permutations gives a bivariate extension
of a functional central limit theorem of Janson and Barbour.
On Markov chains, we answer positively an open question of Bourdon and Vallée
on the asymptotic normality of subword counts in random texts generated by a Markovian source.
\end{abstract}

%

\tableofcontents

\section{Introduction}
\subsection{Background: dependency graphs}
The central limit theorem is one of the most famous results in probability theory :
it states that suitably renormalized sums of {\em independent identically distributed}
 random variables with finite variance
converge towards a standard Gaussian variable.

It is rather easy to relax the {\em identically distributed} assumption.
The Lindeberg criterion, see {\em e.g.} \cite[Chapter 27]{BillingsleyProbMeasure}, gives a sufficient
(and almost necessary) criterion 
for a sum of {\em independent} random variables to converge towards a Gaussian law
(after suitable renormalization).

Relaxing independence is more delicate and there is no universal theory to do it.
One of the ways, among many others,
is given by the theory of dependency graphs.
A dependency graph encodes the dependency structure 
in a family of random variables: roughly we take a vertex for each variable in the family 
and connect dependent random variables by edges.
The idea is that, if the degrees in a sequence of dependency graphs do not grow too fast,
then the corresponding variables behave as if independent
and the sum of the corresponding variables is asymptotically normal.
Precise normality criteria using dependency graphs have been given
by Petrovskaya/Leontovich, Janson, Baldi/Rinott and Mikhailov 
\cite{PetrovskayaLeontovich:Dep_Graph,JansonDependencyGraphs,Baldi_Rinott:DepGraphs_Stein,MikhailovDependencyGraphs}.

These results are powerful black boxes to prove asymptotic normality of sums 
of {\em partially dependent} variables
and can be applied in many different contexts.
The original motivation of Petrovskaya and Leontovich
comes from the mathematical modelization of cell populations \cite{PetrovskayaLeontovich:Dep_Graph}.
On the other hand, Janson was interested in random graph theory:
dependency graphs are used to prove central limit theorems
for some statistics, such as subgraph counts, in $G(n,p)$
 \cite{Baldi_Rinott:DepGraphs_Stein,JansonDependencyGraphs,JansonRandomGraphs};
 see also \cite{Penrose:Geometric_RG} for applications
 to {\em geometric} random graphs.
The theory has then found a field of application in geometric probability,
where central limit theorems have been proven for various statistics
on random point configurations:
the lengths of the nearest-neighbour graph, of the Delaunay triangulation and
of the Voronoi diagram of these random points \cite{DepGraph_GeomProba,DepGraph_GeomProba2}, 
or the area of their convex hull \cite{BV07}.
More recently it has been used to prove asymptotic normality of pattern counts
in random permutations \cite{BonaMonotonePatterns,JansonHitczenko:DepGraph_Patterns}.
Dependency graphs also generalize the notion of $m$-dependence 
\cite{HoeffdingRobbins:M_Dep,Berk:Unbounded_M_Dep},
widely used in statistics \cite{DasGuptaBookStat}.
All these examples illustrate the importance of the theory of dependency graphs.

\subsection{Overview of our results}
The goal of this article is to introduce a notion of weighted dependency graphs;
see \cref{Def:Main}.
As with usual dependency graphs,
we want to prove the asymptotic normality of sums of random variables
$X_n = \sum_{i=1}^{N_n} Y_{n,i}$.
Again, we take a vertex for each variable $Y_{n,i}$ in the family 
and connect dependent random variables with an edge.
The difference is that edges may now carry weights in $(0,1)$.
If two variables are almost independent (in a sense that will be made precise),
the weight of the corresponding edge is small.
Our main result, \cref{ThmMain}, is 
a normality criterion for {\em weighted dependency graphs}:
roughly, instead of involving degrees as Janson's or Baldi/Rinott's criteria,
we can now use the {\em weighted degree}, which is in general smaller.

This of course needs to quantify in some sense the ``dependency'' between random variables.
This is done using the notion of {\em joint cumulants}, 
and maximum spanning trees of weighted graphs
(which is a classical topic in algorithmics literature; see \cref{sec:SpanningTree}).
\medskip

As explained in \cref{SectComparisonWeightedUsual},
our normality criterion contains Janson's criterion and 
natural applications of Mikhailov's criterion.
Unfortunately, we are not able to deal with variables $Y_{n,i}$
with only few finite moments,
as in the result of Baldi and Rinott.

On the other hand, and most importantly,
the possibility of having small weights on edges
extends significantly the range of application of the theory.
Indeed, in this article we provide several examples where weighted dependency graphs
are used to prove asymptotic normality of sums of {\em pairwise dependent} random variables
(for such families, the only usual dependency graph is the complete graph,
and the standard theory of dependency graphs is useless).
Examples given in the article involve pair partitions,
the random graph model $G(n,M)$, permutations, 
 statistical mechanics
and finally Markov chains.

Except for variance estimates in some examples, our normality criterion is
easy to apply.
Proving that a given graph is a weighted dependency graph might be difficult
a priori, but we provide general statements that reduce it in several cases
to an elementary moment computation (see detail in \cref{ssec:Proving_WDG}).
Therefore the present article gives simple proofs of central limit theorems
on a large variety of objects,
that are hard or non-accessible via other methods.

Before describing specifically the results obtained on each of these objects,
let us mention that weighted dependency graphs can also be used to prove multivariate asymptotic normality
and functional central limit theorems;
rather than giving a cumbersome general theorem for that,
we refer the reader to examples in
\cref{SubsecSingleIndexPermStat,SubsecDoubleIndexPermStat,SubsecSSEPParticles}.

\label{SectIntroApplications}
\subsubsection{Random pair partitions}
Our first example deals with uniform random pair partitions
of a $2n$ element set.
This model is the starting point of the configuration model
in random graph theory (see {\em e.g.} \cite[Chapter 9]{JansonRandomGraphs})
and has also recently appeared in theoretical physics \cite{CourtielYeats:Random_Chord_Diagrams}.

We consider the number of {\em crossings} in such a random pair partitions.
This is a natural statistics in the combinatorics literature, see {\em e.g.} \cite{CrossingsNestings}.
A central limit theorem for this statistics has been given by Flajolet and Noy
\cite{FlajoletNoy:CLT_Crossings}.
We give an alternate proof of this result (see \cref{ThmAsympNormPP}) that does not rely on the explicit
formula for the generating function.
Our method can be extended to give a central limit theorem for the number
of $k$-crossings, for which no explicit generating function is available,
but, for simplicity of notation in this first example,
we only treat the case of crossings.


\subsubsection{Random graphs}
The second example deals with the random graph model $G(n,M)$,
that is a uniform random graph among all graphs with vertex set
$\{1,\cdots,n\}$ and $M$ edges.
This is the model considered by Erd\H{o}s and Rényi in their
seminal paper of 1959 \cite{ErdosRenyi:RandomGraphs}.

Since the number of edges is prescribed, the presence
of distinct edges are not independent events, unlike in $G(n,p)$.
Therefore the usual theory of dependency graph cannot be used,
but weighted dependency graphs work fine on this model.

To illustrate this, we consider the subgraph count statistics;
{\em i.e.} we fix a finite graph $H$ and look at the number of copies of $H$
in the random graph $G(n,M)$.
We prove a central limit theorem
for these statistics, when $n$ and $M$ go together to infinity in a suitable way
(\cref{Thm:CLT_SubgraphCounts}).

This central limit theorem is a weaker version of a theorem of Janson 
\cite[Theorem 19]{JansonOrthogonalDecomposition}
(who gets the same result with slightly weaker hypotheses).
We nevertheless think that the proof given here is interesting,
since it parallels completely the proof with usual dependency graphs
that can be done for the companion model $G(n,p)$:
we refer to \cite[Chapter 6]{JansonRandomGraphs} for the application
of dependency graphs to central limit theorem for subgraph counts in $G(n,p)$.
In comparison, Janson's approach involves martingales in the continuous time model $G(n,t)$
and a stopping time argument.

\subsubsection{Random permutations}
The study of uniform random permutations is a wide subject in probability theory
and, as for random graphs, it would be hopeless to try and do a comprehensive presentation of it.
Relevant to this paper,
Hoeffding \cite{HoeffdingCombinatorialCLT}
has given a central limit theorem
for what can be called {\em simply indexed permutation statistics}.
The latter is a statistic of the form
\[X_n=\sum_{i=1}^n a^{(n)}(i,\pi(i)),\]
where $\pi$ is a uniform random permutation of size $n$ 
and $a^{(n)}$ a sequence of real matrices with appropriate conditions.

Hoeffding's result has been extended and refined in many directions, including the following ones.
\begin{itemize}
  \item In \cite{Bolthausen_CLT_Permutations}, Bolthausen used Stein's method
    to give an upper bound for the speed of convergence 
    in Hoeffding's central limit theorem.
  \item This work has then been extended to doubly indexed permutation statistics
    (called DIPS for short) 
    by Zhao, Bai, Chao and Liang \cite{CLT_DIPS}.
    Barbour and Chen \cite{BarbourChen:DIPS}
    have then given new bounds on the speed of convergence,
    that are sharper in many situations.
DIPS have been used in various contexts in statistics;
we refer the reader to \cite{CLT_DIPS, BarbourChen:DIPS} and references therein
for background on these objects.
  \item In another direction, Barbour and Janson
    have established a {\em functional central limit theorem}
    for single indexed permutation statistics \cite{BarbourJansonFunctionalCLT}.
\end{itemize}
Using weighted dependency graphs, we provide
a {\em functional central limit theorem} for doubly indexed permutation statistics;
see \cref{ThmFunctionalCLTDIPS}.
This can be seen as an extension of Barbour and Janson's theorem
or a functional version of Zhao, Bai, Chao and Liang's result
(note however that, in the simply indexed case,
our hypotheses are slightly stronger than the ones of Barbour and Janson
and that we cannot provide a speed of convergence).
There is {\em a priori} no obstruction in obtaining an
extension for {\em $k$-indexed} permutation statistics,
except maybe that the general statement and 
the computation of covariance limits in specific examples
may become quickly cumbersome.

\subsubsection{Stationary configuration of SSEP}
The symmetric simple exclusion process (SSEP) is
a classical model of statistical physics 
that represents a system outside equilibrium.
Its success in the physics literature 
is mainly due to the fact that 
it is tractable mathematically
and displays phase transition phenomena.
We refer the reader to \cite{derrida2007SurveyASEP}
for a survey of results on SSEP and related models
from a mathematical physics viewpoint.

The description of the invariant measure, or {\em steady state},
of SSEP (and more generally the asymmetric version {\em ASEP}), 
has also attracted the interest of the combinatorics community
in the recent years.
This question is indeed connected to the hierarchy of orthogonal polynomials
and has led to the study of new combinatorial objects,
such as permutation tableaux and staircase tableaux 
\cite{CorteelWilliamsASEPGeneral,CorteelStanleyStantonWilliams:ASEP_AskeyWilson}.

In this paper we prove that indicator random variables,
which indicate the presence of particles at given locations in the steady state,
have a natural weighted dependency graph structure.
As an application we give a functional central limit theorem 
for the particle distribution function in the steady state, \cref{ThmFCLTParticles}.
An analogue result for the density function, which is roughly the derivative
of the particle distribution function has been given
by Derrida, Enaud, Landim and Olla \cite{DerridaEtAl:CLT_Density_ASEP}.
Their result holds in the more general setting of ASEP
and it would be interesting to generalize our approach to ASEP as well.

\subsubsection{Markov chains}
Our last application deals with the number of occurrences
of a given subword in a text generated by a Markov source.
More precisely, let $(M_k)_{k \ge 0}$ be an aperiodic irreducible
Markov chain on a finite state space $S$.
Assume that $M_0$ is distributed according
to the stationary distribution $\pi$ of the chain
and denote $w_n=(M_0,M_1,\ldots,M_n)$.
We are interested in the number of times $X_n$ that a given word $v=s_1 \cdots s_m$
occurs as a subword of $w_n$,         
possibly adding some additional constraints,
such as adjacency of some letters of $v$ in $w_n$.

This problem, motivated by intrusion detection in computer science
and identifying meaningful bits of DNA in molecular biology,
has attracted the attention of the analysis of algorithm community
in the nineties; we refer the reader to \cite{FlajoletValleePatterns}
for detailed motivations and references on the subject.

A central limit theorem for $X_n$ was obtained in some particular cases:
\begin{itemize}
  \item when we are only counting consecutive occurrences of $v$, {\em i.e.}
    the number of {\em factors} of $w_n$ that are equal to $v$
    (see Régnier and Szpankowski \cite{Regnier_Spankowski:CLTPattern},
    or Bourdon and Vallée \cite{BourdonValleePatternsGaussian}
    for an extension to {\em probabilistic dynamical} sources);
  \item or when the letters $M_1,M_2,\ldots,M_n$ of $w_n$ are independent
    (see Flajolet, Szpankowski and Vallée \cite{FlajoletValleePatterns}).
  \item Another related result is a central limit theorem 
    by Nicodème, Salvy and Flajolet \cite{NicodemeSalvyFlajoletCLTPositionOccurences}
    for the number of {\em occurrence positions},
    i.e. positions where an occurrence of the pattern terminates.    
    This statistics is quite different from the number of occurrences itself,
    since the number of occurrence positions is always bounded by the length of the word.
\end{itemize}
Despite all these results,
the number of occurrences in the general subword case with a Markov source was left open
by these authors; see \cite[Section 4.4]{Bourdon_Vallee:Pattern_Matching}.
Using weighted dependency graphs, 
we are able to fill this gap; see \cref{Thm:TCL_Patterns_Markov_Sources}.

Note that there is a rich literature on central limit theorems
for {\em linear statistics} on Markov chains $(M_n)_{n \ge 0}$,
that is statistics of the form $S_N^f:=\sum_{i=0}^N f(M_n)$
for a function $f$ on the state space.
We refer the reader to \cite{Jones:CLT_Markov} and references therein
for numerous results in this direction, in particular on infinite state spaces.
In \cite{LivreOrange:Cumulants}, the authors study through cumulants
linear statistics on mixing sequences (including Markov chains; Chapter 4)
and {\em multilinear} statistics on independent identically distributed random variables (Chapter 5).
It seems however that there is a lack of tools to study multilinear statistics on Markov chains
such as the above considered subword count statistics.
The theory of weighted dependency graphs introduced here is such a tool.

\subsubsection{Homogeneity versus spatial structure}
It is worth noticing that the previous examples have various structures.
The first three are homogeneous in the sense that there is a transitive automorphism
group acting on the model.
This is reflected in the corresponding weighted dependency graphs that have all equal weights.

In comparison, the last two examples have a linear structure:
particles in SSEP are living on a line and a Markov chain is canonically indexed by $\NN$.
For Markov chains, this is reflected in the corresponding weighted dependency graph,
since the weights decrease exponentially with the distance.
On the contrary, SSEP has a homogeneous weighted dependency graph (all weights are equal to $1/n$),
which comes as a surprise for the author and indicates a quite different dependency structure
from the Markov chain setting.

The possibility to cover models with various dependency structures
is, in the author's opinion, a nice feature of weighted dependency graphs.

\subsection{Finding weighted dependency graphs}
\label{ssec:Proving_WDG}
The proof of our normality criterion (\cref{ThmMain})
is quite elementary and easy.
Therefore, one could argue that the difficulty of proving a central limit
theorem has only been shifted to the difficulty of finding 
an appropriate weighted dependency graph.
Indeed, proving that a given weighted graph $\WDep$ is                           
a weighted dependency graph for a given family of random variables                    
$\{Y_\alpha, \alpha \in A\}$ 
consists in establishing bounds on all joint cumulants 
$\kappa(Y_\alpha; \alpha \in B)$, where $B$ is a multiset of elements of $A$.
We refer to this problem as {\em proving the correctness} of the weighted dependency graph $\WDep$.
Attacking it head-on is rather challenging.
(The definition of joint cumulants is given in \cref{EqDefCumulant};                 
the precise bound that should be proved can be found in \cref{EqFundamental},         
but is not relevant for the discussion here.)

To avoid this difficulty, we give in \cref{SectTools} three general results that help
proving the correctness of a weighted dependency graph.
These results make the application of our normality criterion much easier in general,
and almost immediate in some cases.

Before describing these three tools,
let us observe that 
proving the correctness of a {\em usual} dependency graph $\Dep$ is usually straightforward;
it is most of the time
an immediate consequence of the definition of the model we are working on.
Therefore the existing literature does not provide any tool for that.

\begin{enumerate}
  \item Our first tool (\cref{PropAlternate}) is an equivalence of the definition
    with a slightly different set of inequalities involving cumulants
    of product of random variables.
    When the random variables $Y_\alpha$ are Bernoulli random variables,
    we can then use the trivial fact $Y_\alpha^m=Y_\alpha$
    to reduce (most of the time significantly)
    the number of inequalities to establish.
  \item The second tool (\cref{PropEqSCQF}) shows the equivalence of bounds
    on cumulants and bounds on an auxiliary quantity defined as
    \[ P_r=\prod_{\delta \subseteq [r]} \esper\left[ \prod_{i\in \delta} Y_{\alpha_i} \right]^{(-1)^{|\delta|}}.\]
    At first sight, one might think that this new expression is not simpler to bound than cumulants,
    but its advantage is that it is {\em multiplicative}:
    if moments $\esper\big[ \prod_{i\in \delta} Y_{\alpha_i} \big]$ have a natural factorization,
    then $P_r$ factorizes accordingly and we can bound each factor separately.
  \item The third tool (\cref{PropProducts}) is a stability property
    of weighted dependency graphs by {\em products}.
    Namely, if we prove that some basic variables admit a weighted dependency graph,
    we obtain for free a weighted dependency graphs for monomials in these basic variables.
    A typical example of application is the following:
    in the random graph setting, we prove that the indicator variables
    corresponding to presence of edges have a weighted dependency graph
    and we automatically obtain a similar result for presence of triangles
    or of copies of any given fixed graph.
\end{enumerate}
Items 1 and 3 are both used in all applications described in \cref{SectIntroApplications}
and reduces the proof of the correctness of the relevant weighted dependency graph
to bounding specific simple cumulants.
For random pair partitions, random permutations and random graphs,
this bound directly follows from an easy computation of joint moments 
and item 2 above.
In summary, the proof of correctness of the weighted dependency graph
is rather immediate in these cases.

For SSEP, we also make use of an induction relation for joint cumulants
obtained by Derrida, Lebowitz and Speer \cite{DerridaLongRangeCorrelation}
(joint cumulants are called {\em truncated correlation functions} in this context).
The Markov chain setting uses linear algebra considerations
and a recent expression of joint cumulants in terms of the so-called {\em boolean cumulants},
due to Arizmendi, Hasebe, Lehner and Vargas \cite{ArizmendiHasebeLehnerVargas2014}
(see also \cite[Lemma 1.1]{LivreOrange:Cumulants}).
Boolean cumulants have been introduced in {\em non-commutative probability theory}
\cite{SpeicherWoroudi:Boolean_Cumulants,Lehner2004cumulants}
and their appearance here is rather intriguing.
\medskip

To conclude this section, let us mention that in each case,
the proof of correctness of the weighted dependency graph relies
on some expression for the joint moments of the variables $Y_\alpha$.
This expression might be of various forms:
explicit expressions in the first three cases,
an induction relation in the case of SSEP or
a matrix expression for Markov chains,
but we need such an expression.
In other words, weighted dependency graphs can be used
to study what could be called {\em locally integrable systems},
that is systems in which the joint moments
of the basic variables $Y_\alpha$ can be computed.
Such systems are not necessarily integrable in the sense
that there is no tractable expression for the generating function
or the moments of $X=\sum_{\alpha \in A} Y_\alpha$,
so that classical asymptotic methods can {\em a priori} not be used.
In particular, in all the examples above, it seems hopeless to analyse
the moments $\esper[X^r]$ by expanding them directly in terms
of joint moments.

\subsection{Usual dependency graphs: behind the central limit theorem.}
\label{SectBeyondCLT}
We have focused so far on the question of asymptotic normality.
However, usual dependency graphs can be used to establish
other kinds of results.
The first family of such results consists in
refinements of central limit theorems.
\begin{itemize}
    \item In their original paper \cite{Baldi_Rinott:DepGraphs_Stein}
     Baldi and Rinott have combined
    dependency graphs with Stein's method.
    In addition to providing a central limit theorem,
    this approach yields precise estimates
     for the Kolmogorov distance between a renormalized version of $X_n$ and the Gaussian distribution.
     For more general and in some cases sharper bounds,
     we also refer the reader to \cite{ChenShao:DepGraph_Stein}.
      An alternate approach to Stein's method,
      based on {\em mod-Gaussian convergence} and Fourier analysis,
      can also be used to establish sharp bounds in Kolmogorov distance
      in the context of dependency graphs,
      see \cite{ModGaussian2}.
      \item Another direction, addressed in \cite{DE12,ModGaussian1},
       is the validity domain of the central limit theorem.
 \end{itemize}
 
 The Gaussian law is not the only limit law that is accessible with the dependency graph approach.
 Convergence to Poisson distribution can also be proved this way,
 as demonstrated in \cite{ArratiaGoldsteinGordon:Stein_Poisson};
 again, this result has found applications, e.g., in the theory
 of random geometric graphs \cite{Penrose:Geometric_RG}.

 We now leave convergence in distribution to discuss probabilities of rare events:
 \begin{itemize}
    \item 
 In \cite{Jan04}, S. Janson has established some large deviation upper bound 
 involving the fractional chromatic number of the dependency graph.
\item
 Another important, historically first use of dependency graphs
 is the {\em Lovász local lemma} \cite{LLL,Spencer:LLL}.
 The goal here is to find a lower bound for the probability that $X_n=0$
 when $Y_{n,i}$ are indicator random variables,
 that is the probability that none of the $Y_{n,i}$ is equal to $1$.
 This inequality has found a large range of application to prove by probabilistic arguments
 the existence of an object (often a graph) with given properties:
 this is known as the {\em probabilistic method}, see \cite[Chapter 5]{AlonSpencer}.
 \end{itemize}

\subsection{Future work}
We believe that weighted dependency graphs may be useful 
in a number of different models and that they are worth being studied further.
An application of weighted dependency
graphs to the $d$-dimensional Ising model is given in a joint paper with Dousse \cite{WDG_Ising}.
In a work in progress,
we also use them to study statistics in uniform set-partitions
and obtain a far-reaching generalization of a result of
Chern, Diaconis, Kane and Rhoades \cite{CLT_SetPartitionsStatistics}.

Proving the correctness of these weighted dependency graphs
again use the tools from \cref{SectTools} of this paper.
In the case of Ising model, we also need the theory of cluster expansions.
\medskip

Another source of examples of
weighted dependency graphs
is given by determinantal point processes (see, {\em e.g.}, \cite[Chapter 4]{DPP}):
indeed, for such processes, it has been observed by Soshnikov 
that cumulants have rather nice expressions \cite[Lemma 1]{Soshnikov:CLT_DPP}.
This fits in the framework of weighted dependency graphs
and the stability by taking monomials in the initial variables
may enable to study multilinear statistics on such models.
This is a direction that we plan to investigate in future work.
\medskip

The results of the present article also invite to consider the following models.
\begin{itemize}
  \item Uniform $d$-regular graphs: the weighted dependency graph 
    for pair partitions presented in \cref{SectMatchings}
    gives bounds on joint cumulants in the configuration model.
    It would be interesting to have similar bounds
    for uniform $d$-regular graphs, especially when $d$ tends to infinity,
    in which case the graph given by the configuration model
    is simple with probability tending to $0$. 
    The fact that joint moments of presence of edges
    have no simple expression for $d$-regular graphs
    is an important source of difficulty here.
  \item The asymmetric version of SSEP, called ASEP:
    finding a weighted dependency graph for this statistical mechanics model
    is closely related to the conjecture made in \cite{DerridaLongRangeCorrelation},
    on the scaling limit of the truncated correlation functions.
  \item Markov chains on infinite state spaces:
    as mentioned earlier, there is an important body of literature on CLT for 
    linear statistics $\sum_{n=0}^N f(M_n)$ on such models, see \cite{Jones:CLT_Markov}.
    Does \cref{PropWDGInMarkov}, which gives a weighted dependency graph for Markov chain
    on a finite state space, generalize under some of these criteria?
    This would potentially give access to CLT for multilinear statistics on these models\ldots
\end{itemize}
\medskip

Finally, because of the diversity of examples,
it would be of great interest to adapt some of the results mentioned in \cref{SectBeyondCLT}
to weighted dependency graphs.
An approach to do this would be to use recent results on mod-Gaussian convergence 
\cite{ModGaussian1,ModGaussian2}.
Unfortunately, this requires {\em uniform bounds} on cumulants of the sum $X_n$,
which are at the moment out of reach for weighted dependency graphs in general.

\subsection{Outline of the paper}
The paper is organized as follows.
\begin{itemize}
  \item Standard notation and definitions are given in \cref{SectPreliminaries}.
  \item \cref{sec:SpanningTree} gives some background about maximum spanning trees,
    a notion used in our bounds for cumulants.
  \item The definition of weighted dependency graphs and the associated normality criterion
    are given in \cref{SectWDG}.
  \item \cref{SectTools} provides tools to prove the correctness of weighted dependency graphs.
  \item The next five sections (from \ref{SectMatchings} to \ref{SectMarkov})
    are devoted to the applications described in \cref{SectIntroApplications}.
  \item Appendices give a technical proof, some variance estimations
    and adequate tightness criteria for the functional central limit theorems,
    respectively.
\end{itemize}

\section{Preliminaries}
\label{SectPreliminaries}

\subsection{Set partitions}
\label{SectPrelim}
\label{SubsecSetPartitions}
The combinatorics of set partitions is central in the theory of cumulants 
 and is important in this article.
 We recall here some well-known facts about them.

A {\em set partition} of a set $S$ is a (non-ordered) family of non-empty disjoint
subsets of $S$ (called {\em blocks} of the partition), whose union is $S$.
We denote by $\#(\pi)$ the number of blocks of $\pi$.

Denote $\PPP(S)$ the set of set partitions of a given set $S$.
Then $\PPP(S)$ may be endowed with a natural partial order:
the {\em refinement} order.
We say that $\pi$ is {\em finer} than $\pi'$ or $\pi'$ {\em coarser} than $\pi$
(and denote $\pi \leq \pi'$)
if every part of $\pi$ is included in a part of $\pi'$.

Endowed with this order, $\PPP(S)$ is a complete lattice, which means that
each family $F$ of set partitions admits a join (the finest set partition
which is coarser than all set partitions in $F$, denoted with $\vee$)
and a meet (the coarsest set partition
which is finer than all set partitions in $F$, denoted with $\wedge$).
In particular, there is a maximal element $\{S\}$ (the partition in only one
part) and a minimal element $\{ \{x\}, x \in S\}$ (the partition in singletons).


Lastly, denote $\mu$ the M\"obius function of the partition lattice $\PPP(S)$.
In this paper, we only use evaluations of $\mu$ at pairs $(\pi,\{S\})$,
{\em i.e.} where the second argument is the maximum element of $\PPP(S)$.
In this case, the value of the M\"obius function is given by:
\begin{equation}\label{EqValueMobius}
    \mu(\pi, \{S\})=(-1)^{\#(\pi)-1} (\# (\pi)-1)!.
\end{equation}

\subsection{Joint cumulants}
For random variables $X_1,\dots,X_r$ with finite moments
living in the same probability space
(with expectation denoted $\esper$),
we define their {\em joint cumulant} (or {\em mixed cumulant}) as
\begin{equation}
    \kappa (X_1,\dots,X_r) = [t_1 \dots t_r] \log 
    \bigg( \esper \big( \exp(t_1 X_1 + \dots + t_r X_r) \big) \bigg).
    \label{EqDefCumulant}
\end{equation}
As usual, $[t_1 \dots t_r] F$ stands for the coefficient of $t_1 \dots t_r$ 
in the series expansion of $F$ in positive powers of $t_1, \dots, t_r$.
The finite moment assumption ensures that the function is analytic
around $t_1=\cdots=t_r=0$.
If all random variables $X_1,\cdots,X_r$ are equal to the same variable $X$,
we denote $\kappa_r(X)=\kappa(X,\dots,X)$ and this is the usual {\em cumulant}
of a single random variable.
\medskip.

Joint cumulants have a long history in statistics
and theoretical physics and it is rather hard
to give a reference for their first appearance.
Their most useful properties are summarized in \cite[Proposition 6.16]{JansonRandomGraphs}
--- see also \cite{LeonovShiryaevCumulants}.
\begin{itemize}
    \item It is a symmetric multilinear functional.
    \item If the set of variables $\{X_1,\dots,X_r\}$ can be split
        into two mutually independent sets of variables, then
        the joint cumulant vanishes;
    \item Cumulants can be expressed in terms of joint moments and vice-versa, as follows:
\begin{align}
    \esper \big( X_1 \cdots X_r \big) &= \sum_{\pi \in \PPP([r])} \prod_{C \in \pi}
    \kappa(X_i; i \in C);
    \label{EqCumulant2Moment} \\
    \kappa (X_1,\dots,X_r) &= \sum_{\pi \in \PPP([r])} \mu(\pi, \{[r]\})
    \prod_{C \in \pi} \esper\left( \prod_{i \in C} X_i \right).
    \label{EqMoment2Cumulant}
\end{align}
        Hence, knowing all joint cumulants amounts to knowing all joint moments.
\end{itemize}
Because of the symmetry, it is natural to consider joint cumulants of {\em multisets} 
of random variables. 

The second property above has a converse.
Since we have not been able to find it in the literature,
we provide it with a proof.
\begin{proposition}
    Let $A=A_1 \sqcup A_2$ be a finite set, partitioned into two parts.
    Let $\{Y_\a,\a \in A\}$ be a family of random variables
    defined on the same probability space,
    such that each $Y_\a$ is determined by its moments.
    We assume that for each multiset $B$ of $A$ that contains
    elements of both $A_1$ and $A_2$,
    \[ \ka(Y_\a ; \a \in B)= 0.\]
    Then $\{Y_\a,\a \in A_1\}$ and $\{Y_\a,\a \in A_2\}$ are independent.
    \label{PropVanishingCumImpliesInd}
\end{proposition}
\begin{proof}
    Since each $Y_\a$ is determined by its moments, from a theorem of Petersen \cite{PetersenMultidimMoment},
    we know that the multivariate random variable $(Y_\a,\a \in A_1 \cup A_2)$ is also determined by its joint moments,
    or equivalently by its joint cumulants.
    Consider random variables $(Z_\a,\a \in A_1 \cup A_2)$ such that $(Z_\a,\a \in A_1)$ (resp. $(Z_\a,\a \in A_2)$)
    has the same (multi-variate) distribution than $(Y_\a,\a \in A_1)$ (resp. $(Y_\a,\a \in A_2)$) and such that
    $\{Z_\a,\a \in A_1\}$ and $\{Z_\a,\a \in A_2\}$ are independent.
    Because of the equalities of multi-variate distribution,
    if the multiset $B$ is composed either only by elements of $A_1$
    or only by elements of $A_2$, then
    \[\ka(Z_\a,\a \in B) = \ka(Y_\a,\a \in B). \]
    On the other hand, if $B$ contains elements of both $A_1$ and $A_2$, 
    then $\{Z_\a,\a \in B\}$ can be split into two mutually independent sets:
    $\{Z_\a,\a \in B \cap A_1\}$ and $\{Z_\a,\a \in B \cap A_2\}$. Therefore,
    \[\ka(Z_\a,\a \in B) =0.\]
    But, for such $B$, one has $\ka(Y_\a,\a \in B) =0$ by hypothesis.

    Finally all joint cumulants of $(Y_\a,\a \in A_1 \cup A_2)$ and $(Z_\a,\a \in A_1 \cup A_2)$
    coincide and, therefore, both random vectors have the same distribution
    (recall that the first one is determined by its joint moments).
    Therefore $\{Y_\a,\a \in A_1\}$ and $\{Y_\a,\a \in A_2\}$ are independent, as claimed.
\end{proof}
\begin{remark}
    We do not know whether the hypothesis ``determined by their moment'' can be relaxed or not.
\end{remark}

\subsection{Multisets}
As mentioned above it is natural to consider joint cumulants
of {\em multisets} of random variables, so let us fix some terminology.

For a multiset $B$, we denote by $|B|$ the total number of elements
({\em i.e.} counted with multiplicities)
and $\#(B)$ the number of {\em distinct} elements.
Furthermore $B_1 \uplus B_2$ is by definition the disjoint union of 
the multisets $B_1$ and $B_2$, {\em i.e.} the multiplicity of an element
in $B_1 \uplus B_2$ is the sum of its multiplicity in $B_1$ and $B_2$.

The set of multisets of elements of $A$ is denoted by $\Mset(A)$,
while $\Mset_{\le m}(A)$ is the subset of multisets with $|B|\le m$.

\subsection{Graphs}
\begin{definition}
  A graph is a pair $(V,E)$, where $V$ is the vertex set
  and $E$ the edge set.
  Elements of $E$ are 2-element subsets of $V$
  (our graphs are simple loopless graphs).
  All graphs considered in this paper are finite.

  We denote by $\CC(\Dep)$ the partition of the vertex set of a graph $\Dep$ into connected components.
  Consequently, $|\CC(\Dep)|$ is the number of connected components of $\Dep$.
\end{definition}
Two types of graphs appear here:
dependency graphs throughout the paper
and random graphs in \cref{SectErdosRenyi}.
The former are tools to prove central limit theorems,
while the latter are the objects of study, and they should not be confused.
Following \cite{JansonRandomGraphs}, we use the letter $\Dep$ for dependency graphs,
and we reserve the more classical $G$ for random graphs.
\medskip

If $B$ is a multiset of vertices of $\Dep$,
we can consider the graph $\Dep[B]$ induced by $\Dep$ on $B$ and defined as follows:
the vertices of $\Dep[B]$ correspond to elements of $B$ (if $B$ contains an element with multiplicity $m$,
then $m$ vertices correspond to this element),
and there is an edge between two vertices if the corresponding vertices of $\Dep$ are equal 
or connected by an edge in $\Dep$.
\medskip

Finally we say that two subsets (or multisets) $A_1$ and $A_2$ of vertices of $\Dep$ are {\em disconnected}
if they are disjoint and 
there is no edge in $\Dep$ that has an extremity in $A_1$ and an extremity in $A_2$.

\subsection{Weighted graphs}
\label{subsec:weighted_graphs}
An {\em edge-weighted graph} $\WDep$, or {\em weighted graph} for short,
is a graph $\Dep$ in which
each edge $e$ is assigned a weight $w_e$.
In this article we restrict ourselves to weights $w_e$
with $w_e \in [0,1]$.
Edges not in the graph can be thought of as edges of weight $0$,
all our definitions are consistent with this convention.
\medskip

The induced graph of a weighted graph $\WDep$ on a multiset $B$
has a natural weighted graph structure.
We put on each edge of $\WDep[B]$ the weight of the corresponding edge in $\WDep$;
if the edge connects two copies of the same vertex of $\WDep$,
there is no corresponding edge in $\WDep$ and we put weight $1$.
\medskip

If $I$ and $J$ are subsets (or multisets) of vertices of a weighted graph $\WDep$,
we write $W(I,J)$ for the maximal weight of an edge connecting a vertex of $I$ and a vertex of $J$.
If $I \cap J \neq \emptyset$, then $W(I,J)=1$.
On the contrary, if $I$ and $J$ are disconnected, we set $W(I,J)=0$.\medskip

This enables to define powers of weighted graphs.
\begin{definition}
    Let $\WDep$ be a weighted graph with vertex set $A$
    and $m$ be a positive integer.
    The $m$-th power of $\WDep$ is the graph with vertex set $\Mset_{\le m}(A)$ such that
    $I$ and $J$ are linked by an edge unless $I$ and $J$ are disjoint and disconnected in $\Dep$.
    Moreover the edge $(I,J)$ has weight $W(I,J)$.
    We denote this weighted graph $\WDep^m$.
    \label{DefPowerWG}
\end{definition}

\subsection{Asymptotic notation}
We use the symbol $u_n \asymp v_n$ (resp. $u_n \ll v_n$, $u_n \gg v_n$) to say that
$\lim_{n \to \infty} \frac{u_n}{v_n}$ is a nonzero constant (resp. $0$, $+\infty$) as $n \to \infty$.
In particular, $v_n$ should be nonzero for $n$ sufficiently large.

\section{Spanning trees}
\label{sec:SpanningTree}

As we shall see in the next section,
our definition of weighted dependency graphs involves
the {\em maximal weight of a spanning tree of a given weighted graph}.
In this section, we recall this notion
and prove a few lemmas that we use later in the paper.

\subsection{Maximum spanning tree}
\begin{definition}
A spanning tree of a graph $\Dep=(V,E)$ is
a subset $E'$ of $E$ such that $(V,E')$ is a tree.

More generally, we say that a subset $E'$ of $E$ forms a spanning subgraph of $\Dep$
if $(V,E')$ is connected.
\end{definition}

If $\WDep$ is a weighted graph, 
we say that the weight $w(T)$ of a spanning tree of $\WDep$
is the {\em product} of the weights of the edges in $T$.
The maximum weight of a spanning tree of $\WDep$
is denoted $\MWST{\WDep}$.
This parameter is central in our work.

If $\WDep$ is disconnected, we set $\MWST{\WDep}=0$ for convenience.

\begin{example}
An easy case which appears a few times in the paper is the case of a connected graph $\WDep$
with $r$ vertices and all weights equal to the same value, say $\eps$.
Then all spanning trees have weight $\eps^{r-1}$ so that $\MWST{\WDep}=\eps^{r-1}$.\smallskip

For a less trivial example, consider the weighted graph of \cref{fig:exWDep}. 
    The red edges form a spanning tree of weight $\eps^2 \cdot (\eps)^2 = \eps^4$.
    It is easy to check that there is no spanning trees with bigger weight so that
    $\MWST{\WDep}=\eps^4$ in this case.
\end{example}
\begin{figure}[ht]
    \begin{center}
        \includegraphics{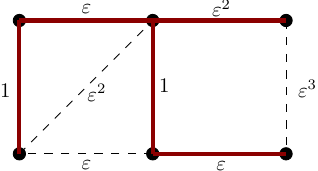}
    \end{center}
    \caption{Example of a weighted graph with a spanning tree of maximal weight.
    Fat red edges are edges of the maximum weight spanning tree, the other edges of the graph are dotted
    for more readability.}
    \label{fig:exWDep}
\end{figure}

Finding a spanning tree with maximum weight is a well-studied
question in the algorithmics literature: see \cite[Chapter 23]{IntroductionAlgorithms}
(the usual convention is to define the weight of a spanning
tree as the {\em sum} of the weights of its edges
and to look for a spanning tree of
{\em minimal} weight, but this is of course equivalent,
up to replacing weights with the logarithms of their inverses).

\subsection{Prim's algorithm and the reordering lemma}
There are several classical algorithms to find a spanning tree with maximum weight.
We describe here Prim's algorithm, which is useful for our work.
\medskip

Assume $\WDep$ is a connected weighted graph.
Choose arbitrarily a vertex $v$ in the graph and set initially $A=\{v\}$ and $T= \emptyset$.
We iterate the following procedure:
find the edge with maximum weight connecting a vertex $v$ in $A$ with a vertex $w$ outside $A$
(since $\WDep$ is connected, there is at least one such edge),
then add $w$ to $A$ and $\{v,w\}$ to $T$.
It is easy to check that at each step, $T$ is always a tree with vertex set $A$
and a general result ensures that at each step, $T$ is included in a spanning tree of maximum weight of $\WDep$
\cite[Corollary 23.2]{IntroductionAlgorithms}.
Note also that the weight of the edge $\{v,w\}$ is equal to $W(\{w\},A)$.
We stop the iteration when $A$ is the vertex set of $\WDep$,
and $T$ is then a spanning tree of maximum weight.
\medskip

The correctness of this algorithm implies the following lemma.
\begin{lemma}
  Let $\WDep$ be a weighted graph with $r$ vertices.
  There exists an ordering $(\beta_1,\dots,\beta_r)$
  of its vertex set such that
  \begin{equation}
    \prod_{j=1}^{r-1} W \big( \{\beta_{j+1}\};\left\{ \beta_1,\cdots,\beta_{j} \right\} \big)
  = \MWST{\WDep}.
  \label{Eq:Well_Ordered}
\end{equation}
  \label{LemReordering}
\end{lemma}
\begin{proof}
  Adding edges of weight $0$ to the graph does not change any side of the above equality,
  so we can assume that $\WDep$ is connected.
\medskip

  We apply Prim's algorithm, as described above, and we denote vertices of $\WDep$
  by $\beta_1,\dots,\beta_r$ in the order in which they are added to the set $A$.
  Then $W \big( \{\beta_{j+1}\};\left\{ \beta_1,\cdots,\beta_{j} \right\} \big)$
  is the weight of the edge added in the $j$-th iteration of the algorithm.
  Therefore the LHS of \cref{Eq:Well_Ordered} is the weight of the spanning tree
  constructed by Prim's algorithm.
  Since this is a spanning tree of maximum weight, this weight is $\MWST{\WDep}$.
\end{proof}

\begin{remark}
    \label{RkReordering}
    In the special case where $\WDep$ has only edges of weight $1$,
    the lemma states the following:
    if $\WDep$ is connected,
    there exists an ordering $(\beta_1,\ldots,\beta_r)$ of its vertices
    such that each $\beta_\ell$ is in the neighbourhood of $(\beta_1,\ldots,\beta_{\ell-1})$.
    This easy particular case is used in the dependency graph literature,
    but with weighted dependency graphs, we need \cref{LemReordering}
    in its full generality.
\end{remark}

\subsection{Inequalities on maximal weights of spanning trees}

We now state some inequalities on maximal weights, that are useful in the sequel.
We first introduce some notation.

If $\Delta$ is a subset of $B$, we denote $\Pi(\Delta)$ the multiset partition of $B$
which has $\Delta$ and singletons as blocks.
Furthermore, if $\bm{\Delta}=(\Delta_1,\cdots,\Delta_\ell)$ is a family of subsets of $B$,
then we denote 
\hbox{$\pi_{\bm{\Delta}} = \Pi(\Delta_1) \vee \cdots \vee \Pi(\Delta_\ell)$.}
Note that if $\Delta_1$, \dots, $\Delta_\ell$ are the parts of a partition $\pi$,
then trivially $\pi_{\bm{\Delta}}=\pi$.

Finally, edges of weight $1$ will play a somewhat special role in weighted dependency graphs.
We therefore denote $\WDepOne$ the subgraph formed by edges with weight $1$.
\begin{lemma}
    \label{LemmaSetPartAndPi1}
    Let $\WDep$ be a weighted graph with vertex set $B$
    and $(\Delta_1,\dots,\Delta_s)$ a family of subsets of $B$. 
    We assume that
    $\pi_{\bm{\Delta}} \vee \CC(\WDepOne)=\{B\}$.
    Then
   \[ \left( \prod_{i=1}^s \MWST{\WDep[\Delta_i]} \right) \le \MWST{\WDep} .\]
\end{lemma}
\begin{proof}
  Consider a spanning tree $T_i$ of maximum weight
  in each induced graph $\WDep[\Delta_i]$.
  Each $T_i$ can be seen as a subset of edges of the original graph $\WDep$.
  Let $S$ be the union of the $T_i$ and of the set of edges of weight $1$.
  The condition $\pi_{\bm{\Delta}} \vee \CC(\WDepOne)=\{B\}$ ensures
  that the edge set $S$ forms a spanning subgraph of $\WDep$.
  Therefore we can extract from it a spanning tree $T$. Then
  \[ w(T) = \prod_{e \in T} w(e) \ge \prod_{e \in S} w(e) \ge \prod_{i=1}^s \prod_{e \in T_i}  w(e) = \prod_{i=1}^s \MWST{\WDep[\Delta_i]}.\]
  But, since $T$ is a spanning tree of $\WDep$, we have $w(T) \le \MWST{\WDep}$,
  which completes the proof.
\end{proof}

Our next lemma uses the notion of $m$-th power of a weighted graph,
which was defined in \cref{subsec:weighted_graphs}.
\begin{lemma}
    Let $I_1,\cdots,I_r$ be multisets of vertices of a weighted graph $\WDep$.
    We consider a partition $\pi$ of $I_1 \uplus \cdots \uplus I_r$ such that
    \begin{equation}
        \pi \vee \big\{ I_1,\cdots,I_r \big\} = \big\{ I_1 \uplus \cdots \uplus I_r \big\}.
        \label{eq:Pi_is_Rough}
      \end{equation}
     Then we have   
    \[ \prod_{i=1}^s \MWST{\WDep[\pi_i]} \le \MWST{\WDep^m [\{I_1,\cdots,I_r\}] },\]
    where $\WDep^m$ is the $m$-th power of $\WDep$. 
    \label{LemProdMWSTPi}
\end{lemma}
\begin{proof}
The multiset $B:=I_1 \uplus \cdots \uplus I_r$ can be explicitly represented by
\[\{(v,j):\ j \le r \text{ and } v \in I_j\}. \]
Let $\pi_i$ be a part of $\pi$ and 
consider a spanning tree $T_i$ of minimum weight of $\WDep[\pi_i]$.
Edges of $T_i$ are pairs $\{(v,j),(v',j')\}$.
For such an edge $e$ with $j \ne j'$, we can consider the corresponding edge $\bar{e}=\{I_j,I_{j'}\}$ in $\WDep^m$.
By definition of power graphs, $\bar{e}$ has at least the same weight as $e$.
Doing so for each edge of $T_i$ with $j \ne j'$,
we get a set $S_i$ of edges in $\WDep^m$ such that 
\[\prod_{\bar{e} \in S_i} w(\bar{e}) \ge \prod_{e \in T_i \atop j \ne j'} w(e) \ge \prod_{e \in T_i} w(e).\]
As in the proof of the previous lemma,
we now consider the union $S$ of the $S_i$'s.
  The condition \eqref{eq:Pi_is_Rough} ensures that $S$ forms 
  a spanning subgraph of $\WDep^m [ \{I_1,\cdots,I_r\}]$ 
  and hence we can extract from it a spanning tree $T$.
 Then
 \[
   w(T) = \prod_{\bar{e} \in T} w(\bar{e}) \ge \prod_{\bar{e} \in S} w(\bar{e}) \ge \prod_{i=1}^s \prod_{\bar{e} \in S_i}  w(\bar{e}) 
   \ge \prod_{i=1}^s \prod_{e \in T_i} w(e) \ge
  \prod_{i=1}^s \MWST{\WDep[\Delta_i]}.
  \]
But, since $T$ is a spanning tree of $\WDep^m[\{I_1,\cdots,I_r\}]$,
we have $w(T) \le \MWST{\WDep^m [\{I_1,\cdots,I_r\}]}$, which concludes the proof.
\end{proof}

\section{Weighted dependency graphs}
\label{SectWDG}
\subsection{Usual dependency graphs}
\label{subsec:usual_dependency_graphs}
Consider a family of random variables $\{Y_\a, \a \in A\}$.
A {\em dependency graph} for this family is an encoding
of the dependency relations 
between the variables $Y_\a$ in a graph structure.
We take here the definition given
by Janson \cite{JansonDependencyGraphs};
see also papers of Malyshev \cite{Malyshev:DepGraphs}
and Petrovskaya/Leontovich \cite{PetrovskayaLeontovich:Dep_Graph}
for earlier appearances of the notion with slightly different names.
\begin{definition}
    A graph $\Dep$ is a dependency graph for the family $\{Y_\a,\a \in A\}$
    if the two following conditions are satisfied:
    \begin{enumerate}
        \item the vertex set of $\Dep$ is $A$.
        \item if $A_1$ and $A_2$ are disconnected subsets in $\Dep$,
            then  $\{Y_\a,\a \in A_1\}$ and $\{Y_\a,\a \in A_2\}$ are independent.
    \end{enumerate}
\end{definition}
A trivial example is that any family of independent variables $\{Y_\a,\a \in A\}$
admits the graph with vertex-set $A$ and no edges as a dependency graph.
A more interesting example is the following.
\begin{example}
    Consider the Erd\H{o}s-Rényi random graph model $G(n,p_n)$,
    that is $G$ has vertex set $[n]:=\{1,\ldots,n\}$
    and it has an edge between $i$ and $j$ with probability $p_n$,
    all these events being independent from each other.
    Let $A$ be the set of 3-element subsets of $[n]$
    and if $\a=\{i,j,k\} \in A$, let $Y_\a$ be the indicator function of the event
    ``the graph $G$ contains the triangle with vertices $i$, $j$ and $k$''.

    Let $\Dep$ be the graph with vertex set $A$ and the following edge set:
    $\a$ and $\beta$ are linked if $|\a \cap \beta|=2$ (that is,
    if the corresponding triangles share an edge in $G$).
    Then $\Dep$ is a dependency graph for the family $\{Y_\a,\a \in A\}$.
    \label{ExDepGraphTriangle}
\end{example}
Note also that the complete graph on $A$ is a dependency graph
for any family of variables indexed by $A$.
In particular, given a family of variables, it may admit several dependency graphs.
The fewer edges a dependency graph has, the more information it encodes
and, thus, the more interesting it is.
It would be tempting to consider the dependency graph with fewest edges,
but such a graph is not always uniquely defined.
\medskip

As said in the introduction,
dependency graphs are a valuable toolbox to prove central limit theorems 
for sums of partially dependent variables.
Denote $\N(0,1)$ a standard normal random variable.
The following theorem is due to Janson \cite[Theorem 2]{JansonDependencyGraphs}.
\begin{theorem}[Janson's normality criterion]
    Suppose that, for each $n$, $\{Y_{n,i}, 1\le i \le N_n\}$ is a family of bounded
    random variables; $|Y_{n,i}| < M_n$ a.s. 
    Suppose further that $\Dep_n$ is a dependency graph for this family
    and let $\Delta_n-1$ be the maximal degree of $\Dep_n$.
    Let $X_n = \sum_{i=1}^{N_n} Y_{n,i}$ and $\si_n^2= \Var(X_n)$.
    
    Assume that there exists an integer $s$ such that
    \begin{equation}
      \left( \tfrac{N_n}{\Delta_n} \right)^{1/s}\, \tfrac{\Delta_n}{\sigma_n}\, M_n
      \to 0\text{ as }n \to \infty.
        \label{EqHypoJanson}
    \end{equation}
    Then, in distribution,
    \begin{equation}
      \tfrac{X_n- \esper X_n}{\si_n} \to_d \N(0,1)\text{ as }n \to \infty.
        \label{EqConclusionJanson}
    \end{equation}
    \label{ThmJansonDepGraphs}
\end{theorem}

\begin{example}
    We use the same model and notation as in \cref{ExDepGraphTriangle}.
    Assume to simplify that $p_n$ is bounded away from $1$.
    Then one has $N_n \asymp n^3$, $\Delta_n \asymp n$ and $M_n=1$.
    An easy computation --- see, \eg, \cite[Lemma 3.5]{JansonRandomGraphs} ---
    gives \hbox{$\sigma_n^2 \asymp \max(n^3 p_n^3, n^4 p_n^5)$}.
    Thus the hypothesis \eqref{EqHypoJanson} in Janson's theorem
    is fulfilled if $p_n \gg n^{-1/3+\eps}$ for some $\eps>0$.

    When this holds, \cref{ThmJansonDepGraphs} implies
    that, after rescaling, the number $X_n$ of triangles in $G(n,p_n)$ 
    is asymptotically normal.
    The latter is in fact true under the less restrictive hypothesis $p_n \gg n^{-1}$,
    as proved by Ruciński \cite{RucinskiCLTSubgraphs},
    but this cannot be obtained from \cref{ThmJansonDepGraphs}.
\end{example}

To finish this section, let us mention a stonger normality criterion,
due to Mikhailov \cite{MikhailovDependencyGraphs}.
Roughly, he replaces the number of vertices $N_n$ and the degree $\Delta_n$
by some quantities defined using conditional expectations of variables.
If \eqref{EqHypoJanson} holds with these new quantities, 
then we can also conclude that one has Gaussian fluctuations.
His theorem has a larger range of applications than Janson's:
{\em e.g.}, for triangles in random graphs, it proves asymptotic normality
in its whole range of validity, that is if $p_n \gg n^{-1}$ and $1- p_n \gg n^{-2}$;
see \cite[Example 6.19]{JansonRandomGraphs}.

\subsection{Definition of weighted dependency graphs}
The goal of the present article is to relax the independence hypothesis
in the definition of dependency graphs.
As we shall see in the next sections, this enables to include many more examples.

As above, $\{Y_\a,\a \in A\}$ is a family of random variables 
defined on the same probability space.
We suggest the following definition.
\begin{definition}
\label{Def:Main}
Let $\bC=(C_1,C_2,\cdots)$ be a sequence of positive real numbers.
Let $\Psi$ be a function on multisets of elements of $A$.

A weighted graph $\WDep$ is a $(\Psi,\bC)$ weighted dependency graph
for $\{Y_\a,\a \in A\}$ if, 
for any multiset \hbox{$B=\{\a_1,\ldots,\a_r\}$} of elements of $A$,
one has
\begin{equation}
    \bigg| \ka\big( Y_\a ; \a \in B \big) \bigg| \le
    C_r \, \Psi(B) \, \MWST{\WDep[B]}. 
    \label{EqFundamental}
\end{equation}
\end{definition}
Our definition implies in particular that all cumulants,
or equivalently all moments of the $Y_\a$ are finite.
This might seem restrictive but in most applications,
the $Y_\a$ are Bernoulli random variables.
Note also that we already have this restriction
in Janson's and Mikhailov's normality criteria.

\begin{remark}
    It is rather easy to ensure inequality \eqref{EqFundamental}.
    For any family $\{Y_\a ; \a \in A\}$, take
    \[\Psi(B)=\big| \ka( Y_\a ; \a \in B) \big|, \quad \bC=(1,1,\cdots),\] 
    and $\WDep$ the complete graph on $A$ with weight 1 on each edge.
    Then $\WDep$ is trivially a $(\Psi,\bC)$ weighted dependency graph
    for $\{Y_\a ; \a \in A\}$.
    But this type of examples do not yield interesting results.
    
    We are interested in constructing examples, where:
    \begin{itemize}
        \item $C_r$ may depend on $r$, but is constant along a sequence of weighted dependency graphs;
        \item $\Psi$ has a rather simple form, such as $p^{\#(B)}$ for some $p$
           (the case $\Psi \equiv 1$ gives a good intuition);
        \item Edge weights also have a very simple expression
            and most of them tend to $0$ 
            along a sequence of weighted dependency graphs;
    \end{itemize}
Intuitively, \cref{EqFundamental} should be thought of as follows:
variables that are linked by edges of small weight in $\WDep$ are {\em almost independent},
in the sense that their joint cumulants are required to be small
(because of the factor $\MWST{\WDep[B]}$).
Indeed, the smaller the weights in $\WDep[B]$ are,
the smaller $\MWST{\WDep[B]}$ is.
\end{remark}


\begin{example}
    \label{ExWeightedDepGraphTriangle}
    Most of this paper is devoted to the treatment of examples:
    proving that they are indeed weighted dependency graphs and 
    inferring some central limit theorems.
    Nevertheless, to guide the reader's intuition,
    let us give right away an example  without proof.

    Consider the Erd\H{o}s-Rényi random graph model $G(n,m_n)$,
    \ie $G$ is a graph with vertex set $[n]$ and an edge set
    $E$ of size $m_n$, chosen uniformly at random among all 
    possible edge set of size $m_n$.
    
    If we set $p_n=m_n/\binom{n}{2}$, then each edge $\{i,j\}$
    belongs to $E$ with probability $p_n$,
    but the corresponding events are not independent anymore.
    Indeed, since the total number of edges is fixed,
    if we know that one given edge is in $G$,
    it is less likely that another given edge is also in $G$.

    As in \cref{ExDepGraphTriangle},
    let $A$ be the set of 3-element subsets of $[n]$
    and if $\a=\{i,j,k\} \in A$, let $Y_\a$ be the indicator function of the event
    ``the graph $G$ contains the triangle with vertices $i$, $j$ and $k$''.
    Since presences of edges are no longer independent event,
    neither are presences of edge-disjoint triangles
    and
    the only dependency graph of this family in the classical sense
    is the complete graph on $A$.

    Consider the complete graph $\WDep$ with vertex set $A$ and weights on the edges determined as follows:
    \begin{itemize}
        \item If $|\a \cap \beta| \ge 2$ (that is,
    if the corresponding triangles share an edge in $G$),
    then the edge $\{\a,\beta\}$ in $\WDep$ has weight $1$; 
  \item If $|\a \cap \beta| \le 1$, then the edge $\{\a,\beta\}$ in $\WDep$  has weight $1/m_n$.
    \end{itemize}
    We will prove in \cref{SectErdosRenyi} that
    $\WDep$ is a $(\Psi_n,\bC)$ weighted dependency graph 
    with $\Psi_n(B)=p_n^{e(B)}$ where
    $e(B)$ is the total number of distinct edges in $B$ 
    (recall that $B$ is here a multiset of triangles)
    and the sequence $\bC=(C_r)$ does not depend on $n$.

    Intuitively, this means that presences of edge-disjoint triangles
    are almost independent events.
    Moreover, the weight $1/m_n$ quantifies this almost-independence.
    This is rather logical:
    the bigger $m_n$ is,
    the less knowing that a given edge is in $G$
    influences the probability that another given edge is also in $G$
    (and hence the same holds for presence of edge-disjoint triangles).
\end{example}

\subsection{A criterion for asymptotic normality}
\label{SubsecNormalityCriterion}
Let $\WDep$ be a $(\Psi,\bC)$ weighted dependency graph 
for a family of variables $\{Y_\a, \a \in A\}$.
We introduce the following parameters (for $\ell \ge 1$)
\begin{align}
    \R &= \sum_{\a \in A} \Psi(\{\a\});
    \label{EqDefR}\\
    \label{EqDefT}
    T_\ell &= \max_{\substack{\a_1,\ldots,\a_{\ell} \in A}} \left[ \sum_{\beta \in A} 
    W(\{\beta\},\{\a_1,\cdots,\a_\ell\}) 
    \frac{\Psi\big(\{\a_1,\cdots,\a_{\ell},\beta\}\big)}{\Psi\big(\{\a_1,\cdots,\a_{\ell} \}\big)}   \right].
\end{align}
\begin{remark}
    Despite the complicated definition of $T_\ell$,
    its order of magnitude is usually not hard to determine in examples
    (recall that $\Psi$ and the weights usually have rather simple expression).
\end{remark}
\begin{remark}
  \label{rmk:RAndTPsiOne}
    Let us consider
     the special case where $\Psi$ is the constant function equal to $1$.
    One has
    \begin{itemize}
        \item $\R=|A|$, which is the number of vertices of $\Dep$;
        \item using the easy observation 
          $w_{\{\beta,\a_1\}} \, \le \, W(\{\beta\},\{\a_1,\cdots,\a_\ell\}) 
          \, \le \, \sum_{i=1}^\ell w_{\{\beta,\a_\ell\}}$,
          we see that \[\Delta \le T_\ell \le \ell \, \Delta,
          \text{ where }
          \Delta:=\max_{\a \in A} \ \sum_{\beta \in \a} w_{\{\beta,\a\}};\]
          note that $\Delta-1$ is the maximal {\em weighted degree} in $\WDep$
          (the weighted degree of a vertex is $\sum_{\beta \in \a,\beta \ne \a} w_{\{\beta,\a\}}$;
          the condition $\beta \ne \a$ in the summation index explains the shift by $-1$).
            In particular, each $T_\ell$ has the same order of magnitude
            as $\Delta$.
    \end{itemize}
    In general, $\R$ and $T_\ell$ should be thought of as deformations
    of the number of vertices and the maximal weighted degree.
    Considering $\R$ and $T_\ell$ rather than simply $|A|$ and $\Delta$
    leads to a more general normality criterion,
    in a similar way that 
    Mikhailov's criterion extends Janson's.
\end{remark}

The following lemma bounds cumulants in terms of the two above defined quantities.
\begin{lemma}
    Let $\WDep$ be a $(\Psi,\bC)$ weighted dependency graph
    for a family of variables $\{Y_\a, \a \in A\}$.
    Define $\R$ and $T_\ell$ (for $\ell \ge 1$) as above.
    Then, for $r \ge 1$, 
    \[ \left| \ka_r \left( \sum_{\a \in A} Y_\a \right) \right|
    \le   C_r\, r! \, \R\, T_1 \cdots T_{r-1}.   \] 
    \label{LemBorneCumulant}
\end{lemma}
\begin{proof}
    By multilinearity
    \[\ka_r \left( \sum_{\a \in A} Y_\a \right)
    = \sum_{\a_1,\ldots,\a_r \in A} \ka \left( Y_{\a_1},\ldots,Y_{\a_r} \right).\]
    Applying the triangular inequality and \cref{EqFundamental},
    \[ \left| \ka_r \left( \sum_{\a_1,\ldots,\a_r \in A} Y_\a \right) \right|
    \le C_r \sum_{\a_1,\ldots,\a_r \in A} M_{\a_1,\ldots,\a_r},\]
    where, by definition, $M_{\a_1,\ldots,\a_r}=\MWST{\WDep[B]} \, \Psi(B)$
    for $B=\{\a_1,\cdots,\a_r\}$ (in particular $M_{\a_1,\ldots,\a_r}$
    is invariant by permutation of the indices).
    
    We also define
    \[ M'_{\a_1,\ldots,\a_r} = \left[ \prod_{j=1}^{r-1}
    W \big( \{\a_{j+1}\};\left\{ \a_1,\cdots,\a_{j} \right\} \big) \right]
    \, \Psi(B)\]
    We say that a list $(\beta_1,\cdots,\beta_r)$ of elements of $A$ is {\em well-ordered} if
    \begin{equation}
            \MWST{\WDep[B]} \, = \,
        \prod_{j=1}^{r-1} W \big( \{\beta_{j+1}\};\left\{ \beta_1,\cdots,\beta_{j} \right\} \big),
        \label{EqDefWellOrdered}
    \end{equation}
    which implies $M_{\beta_1,\ldots,\beta_r} = M'_{\beta_1,\ldots,\beta_r}$. 
    From \cref{LemReordering}, each list $(\a_1,\ldots,\a_r)$ admits a well-ordered permutation.
    Conversely, a well-ordered list $(\beta_1,\cdots,\beta_r)$ is a permutation
    of at most $r!$ lists $(\a_1,\ldots,\a_r)$.
    Therefore
    \begin{equation}
        \left| \ka_r \left( \sum_{\a_1,\ldots,\a_r \in A} Y_\a \right) \right|  \le r!\, C_r
        \sum_{\ldots} M_{\beta_1,\ldots,\beta_r} = r!\, C_r \sum_{\ldots} M'_{\beta_1,\ldots,\beta_r},
\end{equation}
    where both sums run over {\em well-ordered} lists $(\beta_1,\cdots,\beta_r)$ of elements of $A$.
    Extending the sum to all lists $(\beta_1,\cdots,\beta_r)$ of elements of $A$ only increases
    the right-hand side, so that we get:
    \begin{equation}
        \left| \ka_r \left( \sum_{\a_1,\ldots,\a_r \in A} Y_\a \right) \right|  
         \le r!\, C_r \sum_{\beta_1,\ldots,\beta_r \in A} M'_{\beta_1,\ldots,\beta_r}.
    \label{EqBoundCumMp}
\end{equation}
    By definition, one has, for any $\ell<r$ and elements $\beta_1,\cdots,\beta_{\ell+1}$ in $A$:
    \[
        M'_{\beta_1,\ldots,\beta_{\ell+1}} = 
        W(\{\beta_{\ell+1}\},\{\beta_1,\cdots,\beta_\ell\}) 
        \, \frac{\Psi(\{\beta_1,\cdots,\beta_\ell,\beta_{\ell+1}\})}{\Psi(\{\beta_1,\cdots,\beta_\ell\})}
        M'_{\beta_1,\ldots,\beta_{\ell}}
    \]
    Fixing $(\beta_1,\cdots,\beta_\ell)$ and summing over $\beta_{\ell+1}$ in $A$, we get
    \[ \sum_{\beta_{\ell+1} \in A} M'_{\beta_1,\ldots,\beta_{\ell+1}} 
    \le T_\ell M'_{\beta_1,\ldots,\beta_{\ell}}.\]
    
    Since $\sum_{\beta \in A} M'_\beta = \R$, an immediate induction yields
    \[ \sum_{\beta_1,\dots, \beta_r \in A} M'_{\beta_1,\ldots,\beta_r} \le \R\, T_1 \, \cdots\, T_{r-1}.\]
    Together with \cref{EqBoundCumMp}, this ends the proof of the lemma.
\end{proof}

We can now give an asymptotic normality criterion, using weighted dependency graphs.
\begin{theorem}
    Suppose that, for each $n$, $\{Y_{n,i}, 1\le i \le N_n\}$ is a family of 
    random variables with finite moments defined on the same probability space.
    For each $n$, let 
    $\Psi_n$ a function on multisets of elements of $[N_n]$.
    We also fix a sequence $\bC=(C_r)_{r \ge 1}$, {\em not depending} on $n$.

    Assume that, for each $n$, one has a $(\Psi_n,\bC)$ weighted dependency graph $\WDep_n$\,
    for 
    \hbox{$\{Y_{n,i}, 1\le i \le N_n\}$} and define the corresponding quantities
    $\R_n$, $T_{1,n}$, $T_{2,n}$, \ldots, by \cref{EqDefT,EqDefR}.
    
    Let $X_n = \sum_{i=1}^{N_n} Y_{n,i}$ and $\si_n^2= \Var(X_n)$.
    
    Assume that there exist numbers $D_r$ and $Q_n$ and an integer $s\ge 3$ such that
    \begin{align}
        T_{r,n} &\le D_r Q_n 
        \label{EqBoundUnifT} \\
        \left(\tfrac{\R_n}{Q_n}\right)^{1/s}\, \tfrac{Q_n}{\sigma_n} &\to 0\text{ as }n \to \infty,
        \label{EqHypoMainThm}
    \end{align}
    then, in distribution,
    \begin{equation}
      \tfrac{X_n- \esper X_n}{\si_n} \to_d \N(0,1)\text{ as }n \to \infty.
        \label{EqConclusionMainThm}
    \end{equation}
    \label{ThmMain}
\end{theorem}
\begin{proof}
    From \cref{LemBorneCumulant}, we know that, for $r \ge 2$,
    \begin{equation}
      \left|\ka_r\left( \sum_{i=1}^{N_n} Y_{n,i} \right)\right| \le
        C_r \, r! \, \R_n \, D_1 \cdots D_{r-1}\, Q_n^{r-1}.
        \label{EqBoundCumInMainThm}
    \end{equation}
    Setting $C'_r= C_r \, r! \, D_1 \cdots D_{r-1}$ and 
    $\widetilde{X_n}=(X_n- \esper X_n)/\si_n$,
    we get that for $r \ge s$,
    \[ \left|\ka_r(\widetilde{X_n}) \right| =
    \frac{1}{\sigma_n^r} \left|\ka_r(X_n) \right| \le
    C'_r \frac{\R_n \, Q_n^{r-1}}{\sigma_n^r} = 
    C'_r \left( \frac{\R_n \, Q_n^{s-1}}{ \sigma_n^s } \right)^{\frac{r-2}{s-2}}
    \left( \frac{\si_n^2}{\R_n Q_n} \right)^{\frac{r-s}{s-2}}.\]
    \cref{EqBoundCumInMainThm} for $r=2$ ensures that the last factor is bounded
    while the middle factor tends to $0$ from our hypothesis \eqref{EqHypoMainThm}.
    We conclude that $\kappa_r(\widetilde{X_n})$ tends to $0$ for $r \ge s$.
    The convergence towards a normal law then follows from
    \cite[Theorem 1]{JansonDependencyGraphs}.
\end{proof}

\begin{remark}
  Continuing \cref{rmk:RAndTPsiOne}, when $\Psi$ is constant equal to $1$,
  one can choose $D_r=r$ and $Q_n=\Delta_n$, where $\Delta_n$ is the maximal
  weighted degree in $\WDep_n$.
  Then hypothesis \cref{EqHypoMainThm} says that the quotient $\tfrac{\Delta_n}{\sigma_n}$
  tends to $0$ reasonably fast (faster than some power of $\tfrac{R_n}{\Delta_n}$).
  Roughly, one has a central limit theorem as soon as the weighted degree
  is smaller than the standard deviation.
  (In particular, except in pathological cases, the standard deviation should tend to infinity.)
\end{remark}
\begin{remark}
  In most examples of application, $R_n$ is immediate to evaluate,
  while a good upper bound for $T_{\ell,n}$ and thus a sequence $Q_n$ as in the theorem
  can be found by a relatively easy combinatorial case analysis.
  The most difficult part in applying the theorem is to find a lower bound for $\sigma_n$
  (\cref{LemBorneCumulant} gives a usually sharp upper bound).
  In this sense, the weighted dependency graph structure, once uncovered,
  reduces the central limit theorem to a variance estimation.
\end{remark}
\begin{remark}
  \cite[Theorem 1]{JansonDependencyGraphs} also ensures the convergence of all moments.
  Therefore, in \cref{ThmMain} above and in all applications,
  we have convergence of all moments, in addition to the convergence in distribution.
\end{remark}
\begin{remark}
    Except \cref{LemReordering} --- see \cref{RkReordering} ---, 
    the proof of our normality criterion is largely inspired
    from the case of usual dependency graphs.
    The difficulty here was to find a {\em good definition}
    of weighted dependency graphs, not to adapt the theorem
    to this new setting.
\end{remark}

\subsection{Multidimensional convergence and bounds for joint cumulants}
Bounds on cumulants, and thus weighted dependency graphs, can also be used to obtain
the convergence of a random vector towards a multidimensional
Gaussian vector or the convergence of a random function
towards a Gaussian process.

To avoid a heavily technical theorem, we do not state a general result,
but refer the reader to examples in
\cref{SubsecSingleIndexPermStat,SubsecDoubleIndexPermStat,SubsecSSEPParticles}.
We nevertheless give here a useful bound on joint cumulants,
whose proof is a straightforward adaptation of the one of \cref{LemBorneCumulant}.

\begin{lemma}
Let $\WDep$ be a $(\Psi,\bC)$ be a weighted dependency graph 
for a family of variables $\{Y_\a, \a \in A\}$.
Consider subsets $A_1,\cdots,A_r$ of $A$.
Then, with the notation of the previous section,
\[\left|\ka\left(\sum_{\a \in A_1} Y_\a,\ldots,\sum_{\a \in A_r} Y_\a\right)\right|
    \le   C_r\, r! \, R\, T_1 \cdots T_{r-1}.   \] 
    \label{LemBoundJointCumulants}
\end{lemma}
\begin{remark}
    It is also possible in the above bound to replace $R$ by 
    \[R^1= \sum_{\bm{\a \in A_1}} \Psi(\{\a_1\})\]
    and/or the product $T_1 \cdots T_{r-1}$ by $T^2_{\le r-1} \cdots T^r_{\le r-1}$, where
    \[ 
    T^i_{\le r-1} =\bm{\max_{\ell \le r-1}} \max_{\substack{\a_1,\ldots,\a_{\ell} \in A}} \left[ \sum_{\bm{\beta \in A_i}} 
    W(\{\beta\},\{\a_1,\cdots,\a_\ell\}) 
    \frac{\Psi\big(\{\a_1,\cdots,\a_{\ell},\beta\}\big)}{\Psi\big(\{\a_1,\cdots,\a_{\ell} \}\big)}   \right].
    \]
    The maximum over $\ell$ in the equation above comes from the reordering argument,
    that is the use of \cref{LemReordering} in the proof of \cref{LemBorneCumulant}.
    We do not know what is the index of the element taken from $A_i$
    in the reordered sequence $(\beta_1,\cdots,\beta_r)$.
    The only thing we can ensure is that $\beta_1=\alpha_1$
    (since we can choose arbitrarily the first vertex in Prim's algorithm;
    see the proof of \cref{LemReordering}),
    which allows us to use $R^1$ instead of $R$.

    This slight improvement of the bound is not used in the applications given in this paper.
    It could however be useful if we wanted to prove, say, a multivariate
    convergence result for numbers of copies of subgraphs of different sizes in $G(n,m)$;
    see \cref{SectGraphs} for the corresponding univariate statement.

    Note that, with this improvement, the bound given for the joint cumulant
    is not symmetric in $A_1$,\ldots,$A_r$,
    while the quantity to bound obviously is.
\end{remark}

\subsection{Comparison between usual and weighted dependency graphs}
\label{SectComparisonWeightedUsual}
In this Section, we compare at a formal level
the notions of weighted dependency graphs and
of usual dependency graphs.
The results of this Section are not needed in the rest of the paper
and it can safely be skipped.
\medskip

The key observation here is the following:
if the induced weighted graph $\WDep[B]$ is disconnected,
then $\MWST{\WDep[B]}$ is $0$ by definition,
and hence \eqref{EqFundamental} states that the corresponding
joint cumulant should be $0$.
\medskip

For the next proposition, we need to introduce some terminology.
   Let $\{Y_\a,\a \in A\}$ be a family of random variables         
       defined on the same probability space.
We say that a function $\Psi$ on multisets of $A$ dominates joint moments,
if for any multiset $B$ and multiset partition $\pi$ of $B$:
    \[\left| \prod_{C \in \pi} \esper\left( \prod_{\a \in C} Y_\a \right) \right|
    \le \Psi(B).\]
    Examples include:
    \begin{itemize}
        \item  Assume that the variables $\{Y_\a,\a \in A\}$ are uniformly bounded by a constant $M$,
            \ie, for any $\a$, one has $|Y_\a| \le M$ a.s. Then
            for any multiset $B$ and multiset partition $\pi$ of $B$, one has
            \[\left| \prod_{C \in \pi} \esper\left( \prod_{\a \in C} Y_\a \right) \right|          
            \le M^{|B|}.\]
            In other terms, the function $\Psi$ defined by $\Psi(B)=M^{|B|}$
            dominates joint moments.
        \item  More generally, a repetitive use of Hölder inequality,
            together with the monotonicity of the $r$-th norm yields the following:
            for any multiset $B$ and multiset partition $\pi$ of $B$, one has
            \[\left| \prod_{C \in \pi} \esper\left( \prod_{\a \in C} Y_\a \right) \right|          
            \le \prod_{C \in \pi} \prod_{\a \in C} \esper\left( |Y_\a|^{|C|} \right)^{1/|C|}
            \le \prod_{\a \in B} \esper\left( |Y_\a|^{|B|} \right)^{1/|B|}.\]
            In other terms, the function $\Psi$ defined by 
            $\Psi(B)=\prod_{\a \in B} \esper\left( |Y_\a|^{|B|} \right)^{1/|B|}$
            dominates joint moments.
        \item As a more concrete example, consider triangles in random graphs,
            as in \cref{ExDepGraphTriangle,ExWeightedDepGraphTriangle}.
            In both models $G(n,p_n)$ and $G(n,M_n)$, the function $\Psi(B)=p_n^{e(B)}$
            dominates joint moments.
    \end{itemize}
\begin{proposition}
    \label{PropFromUsualToWeightOne}
    Let $\{Y_\a,\a \in A\}$ be a family of random variables
    defined on the same probability space,
    with a dependency graph $\Dep$.
    
    Set $C_r=(r!)^2$ and consider a function $\Psi$ on multisets of $A$
    that dominates joint moments.
    Consider also the weighted graph $\WDep$,
    obtained by assigning weight $1$ to each edge.

    Then $\WDep$ is a $(\Psi,\bC)$ weighted dependency graph for $\{Y_\a,\a \in A\}$.
\end{proposition}
\begin{proof}
    We have to check that the inequality \eqref{EqFundamental} holds for any multiset $B$.
    Consider two cases:
    \begin{itemize}
        \item Assume $B$ is disconnected in $\Dep$. Since $\Dep$ is a dependency graph for $\{Y_\a,\a \in A\}$,
            this implies that the set of variables $\{Y_\a, \a \in A\}$
            can be split into two mutually independent sets of variables
            and $\ka(Y_\a ; \a \in B)=0$, as wanted.
        \item Otherwise, $\WDep$ contains at least one spanning tree, and since all edges have weight $1$,
          all spanning trees have weight $1$.
         Thus $\MWST{\WDep[B]}=1$ and we should prove:
            \[ \bigg| \ka\big( Y_\a ; \a \in B \big) \bigg| \le
            (r!)^2 \Psi(B).\]
            This can be deduced easily from \cref{EqMoment2Cumulant},
            the fact that $\Psi$ dominates joint moments and
            the inequalities $|\mu(\pi,\{[\ell]\}))| \le r!$
            and $|\PPP([r])| \le r!$. \qedhere
    \end{itemize}
\end{proof}
\noindent Conversely, the unweighted version of a $(\Psi,\bC)$ weighted dependency graph 
is also a usual dependency graph,
as soon as each variable $Y_\a$ is determined by its moments,
as shown by the following proposition.
\begin{proposition}
    Let $\{Y_\a,\a \in A\}$ be a family of random variables
    with finite moments defined on the same probability space,
    such that each $Y_\a$ is determined by its moments.
    Let $\bC$ and $\Psi$ be arbitrary
    and assume that we have a $(\Psi,\bC)$ weighted dependency graph $\WDep$ 
    for the family $\{Y_\a,\a \in A\}$.
    Denote $\Dep$ the unweighted version of $\WDep$.

    Then $\Dep$ is a usual dependency graph for the family $\{Y_\a,\a \in A\}$.
\end{proposition}
\begin{proof}
    Let $A_1$ and $A_2$ be disconnected subsets of $A$ in $\Dep$.
    We should prove that $\{Y_\a,\a \in A_1\}$ and $\{Y_\a,\a \in A_2\}$
    are independent.

    Let $B$ be a multiset of elements of $A_1 \sqcup A_2$
    that contains elements in both $A_1$ and $A_2$.
    Then the induced weighted graph $\WDep[B]$ has at least two connected component
    because $B \cap A_1$ and $B \cap A_2$ are disconnected.
    Therefore $\MWST{\WDep[B]}=0$.
    Since $\WDep$ is $(\Psi,\bC)$ weighted dependency graph for $\{Y_\a,\a \in A\}$,
    \cref{EqFundamental} implies that
    \[\ka(Y_\a,\a \in B) =0.\]
    From \cref{PropVanishingCumImpliesInd}, we conclude that
    $\{Y_\a,\a \in A_1\}$ and $\{Y_\a,\a \in A_2\}$
        are independent.
\end{proof}

We can now argue that \cref{ThmMain}
contains Janson's normality criterion.
For each $n\ge 1$, let $\{Y_{n,i},\, 1\le i \le N_n\}$ be a family
of bounded random variables with dependency graph $\Dep_n$.
Consider the weighted graph $\WDep_n$ obtained from $\Dep_n$
by assigning weight $1$ to each edge
and set $\Psi(B)=M_n^{|B|}$, where $M_n$ is an upper bound for all $|Y_{n,i}|$.
From \cref{PropFromUsualToWeightOne},
$\WDep$ is a $(\Psi,\bC)$ weighted dependency graph for $\{Y_{n,i},\, 1\le i \le N_n\}$
with $C_r=(r!)^2$.
Define $R_n$ and $T_{\ell,n}$ as in \cref{SubsecNormalityCriterion}.
If $\Delta_n-1$ is the maximal degree in $\Dep_n$,
then $R_n=M_n\, N_n$ and $T_{\ell,n} \le \ell M_n (\Delta_n)$ for $\WDep_n$.
In particular we can choose $Q_n= M_n \Delta_n$
and condition \eqref{EqHypoMainThm} in our normality criterion 
reduces to \eqref{EqHypoJanson} in Janson's.

On the other hand our theorem does not contain formally
Mikhailov normality criterion \cite{MikhailovDependencyGraphs}.
But it contains classical examples. Again, one should see the dependency graph in each example
as a weighted dependency graph with weight 1 on each edge and choose $\Psi$ as follows:
\begin{itemize}
    \item in the example at the end of Mikhailov's paper \cite{MikhailovDependencyGraphs},
        variables are indexed by pairs of elements of $[n]$, so that a multiset
        $B$ of such pairs can be interpreted as a multigraph $G(B)$ of vertex set $[n]$.
        Then the function $\Psi(B)=N^{-|\CC(G(B))|}$ dominates joint moments
        and we can apply our theorem to prove asymptotic normality,
        in exactly the same way as with Mikhailov's theorem.
    \item for triangles in random graphs \cite[Example 6.19]{JansonRandomGraphs},
        choose $\Psi(B)=p_n^{e(B)}$, as suggested before \cref{PropFromUsualToWeightOne}.
\end{itemize}
In each case, we leave details to the reader.

\section{Finding weighted dependency graphs}
\label{SectTools}
In general, the main difficulty in order to apply \cref{ThmMain}
is to check that $\Dep_n$ is indeed a weighted dependency graph
for the family $\{Y_{n,i}, 1\le i \le N_n\}$ of random variables.
Indeed, one should establish the bound \eqref{EqFundamental},
which may be quite cumbersome.
In this section, we give a few lemmas and propositions
that help in this task in different contexts.

\subsection{An alternate formulation}
In this section, we will see that instead of \eqref{EqFundamental},
one can show a slightly different set of inequalities.
Intuitively, this set of inequalities puts an emphasis on edges of weight $1$,
which, in most applications, relate incompatible events.

We require an extra assumption on the function $\Psi$.
\begin{definition}
    Let $A$ be a set
    and $\Psi$ a 
    function on multisets of elements of $A$.
    Then $\Psi$ is called {\em super-multiplicative} if,
    for any multisets $B_1$ and $B_2$,
    $\Psi( B_1 \uplus B_2 ) \ge \Psi(B_1) \Psi(B_2)$.
\end{definition}
\begin{proposition}
Let $\{Y_\a, \a \in A\}$ be a family of random variables 
defined on the same probability space.
Consider a weighted graph $\WDep$ with vertex set $A$,
a {\em super-multiplicative} function $\Psi$ on multisets of elements of $A$
and a sequence $\bD=(D_r)_{r \ge 1}$.

Assume that, for any multiset $B$ of elements of $A$,
one has
\begin{equation}
    \left| \ka\left( \prod_{\a \in B_1} Y_{\a},
    \cdots, \prod_{\a \in B_\ell}  Y_{\a}  \right) \right| \le
    D_{|B|} \, \Psi(B) \, \MWST{\WDep[B]}, 
    \label{EqFundamentalVariant}
\end{equation}
where $B_1$, \ldots, $B_\ell$ are the vertex sets of the connected components
of the graph $\WDepOne[B]$, 
that is the graph induced by edges of weight $1$ of $\WDep$ on $B$.
       
Then $\WDep$ is a $(\Psi,\bC)$ weighted dependency graph 
for the family $\{Y_\a, \a \in A\}$,
for some sequence $\bC$ that depends only on $\bD$.
\label{PropAlternate}
\end{proposition}
\begin{proof}
    We have to check that the inequality \eqref{EqFundamental} holds for any multiset $B$.
    We proceed by induction on the size $r$ of the multiset $B$.

    Consider the case $r=1$. From \cref{EqFundamentalVariant},
    we know that, for any $\a \in A$, one has:
    \[| \esper(Y_\a) | \le D_1 \Psi(\{\a\}),\]
    so that, if we set $C_1=D_1$, \cref{EqFundamental} holds for 
    all $1$-element sets $B=\{\a\}$.

    Let $r>1$ and assume that \eqref{EqFundamental} holds for all multisets $\tilde{B}$ of size $\ell < r$.
    Fix a multiset $B$ of size $r$ and define $B_1,\dots,B_\ell$ as in above.
    Using a formula of Leonov and Shiryaev for cumulants of products \cite{LeonovShiryaevCumulants}
    --- see also \cite[Theorem 4.4]{Sniady2006fluctuations} ---, one has
    \[ \ka\left( \prod_{\a \in B_1} Y_{\a}, 
                \cdots, \prod_{\a \in B_\ell}  Y_{\a}  \right) 
    =\sum_{\pi \perp \bm{B}} \, \prod_{i=1}^s \ka \big( Y_\a; \a \in \pi_i \big), \]
    where the sum runs over multiset partitions $\pi=\{\pi_1,\cdots,\pi_s\}$ of $B$
    such that $\pi \vee \{B_1,\cdots,B_\ell\} = \{B\}$; 
    we denote this condition by $\pi \perp \bm{B}$
    (for a discussion on multiset partitions, see \cref{rmk:Mset_partitions}
    at the end of the proof).
    We isolate the term corresponding to $\pi=\{B\}$ on the right hand-side and 
    rewrites this as:
    \begin{equation}
        \ka(Y_\a, \a \in B) = \ka\left( \prod_{\a \in B_1} Y_{\a},                               
                    \cdots, \prod_{\a \in B_\ell}  Y_{\a}  \right) -
                    \sum_{\substack{\pi \perp \bm{B} \\ \pi \ne \{B\}}} \, \prod_{i=1}^s \ka \big( Y_\a; \a \in \pi_i \big).
     \label{EqTech1}
 \end{equation}
     But by assumption
     \begin{equation}
         \left| \ka\left( \prod_{\a \in B_1} Y_{\a},                               
          \cdots, \prod_{\a \in B_\ell}  Y_{\a}  \right) \right|
          \le D_r \, \Psi(B) \, \MWST{\WDep[B]}.
          \label{EqHypoAlternate}
      \end{equation}
    Moreover, the induction hypothesis asserts that if $\pi_i$ is a strict subset of $B$, one has
    \[ \big|\ka \big( Y_\a; \a \in \pi_i \big) \big| \le
    C_{|\pi_i|} \, \Psi(\pi_i) \, \MWST{\WDep[\pi_i]}. \]
    If $\pi$ is a set partition of $B$ different from $\{B\}$, all its parts are strict subsets of $B$
    and we have
    \begin{equation}
        \left| \prod_{i=1}^s \ka \big( Y_\a; \a \in \pi_i \big) \right| \le
    \left( \prod_{i=1}^s C_{|\pi_i|} \right) \,
    \left( \prod_{i=1}^s \Psi(\pi_i) \right) \, 
    \left( \prod_{i=1}^s \MWST{\WDep[\pi_i]} \right).
    \label{EqBoundProductCumulants}
\end{equation}
    From the super-multiplicativity, the middle factor is at most $\Psi(B)$.
    Moreover, under the hypothesis $\pi \perp \bm{B}$,
    the last factor is at most $\MWST{\WDep[B]}$, as proved in \cref{LemmaSetPartAndPi1}
    (for the graph $\WDep[B]$ with $\Delta_i=\pi_i$).
    Finally, from \cref{EqTech1,EqHypoAlternate,EqBoundProductCumulants},
    we get:
    \[ \big|\ka(Y_\a, \a \in B) \big| \le
    \left(D_r + \sum_{\substack{\pi \perp \bm{B} \\ \pi \ne \{B\}}} C_{|\pi_i|}\right)
    \, \Psi(B) \, \MWST{\WDep[B]}.\]
    This ends the proof of \eqref{EqFundamental} by setting
    \[C_r= D_r + \sum_{\substack{\pi \in \PPP(B) \\ \pi \ne \{B\}}} C_{|\pi_i|};\]
    observe that the right-hand side depends indeed only on the size $r$ of $B$,
    and not on $B$ itself.
\end{proof}
\begin{remark}
  \label{rmk:Mset_partitions}
  In the previous proof and in \cref{SectProd} below,
  we sum over all multiset partitions $\pi$ of a multiset $B$.
If $B=\{b_1,\dots,b_r\}$, this means that we consider all set partitions of $\{1,\dots,r\}$
and associate with each one a multiset partitions of $B$ by replacing $i$ by $b_i$ within each part.
For example, the multiset $\{a,b,b\}$ has five multiset partitions:
$\{\{a\},\{b\},\{b\}\}$, 
$\{\{a,b,b\}\}$,
$\{\{a\},\{b,b\}\}$ and twice
$\{\{a,b\},\{b\}\}$.
In particular, the number of multiset partitions counted with multiplicity
of a multiset of size $r$ is the $r$-th Bell number,
independently of whether this multiset has repeated elements or not.

With this convention, Leonov and Shiryaev formula clearly holds
with cumulants of multisets.
Indeed the case with equal variables can be obtained from specialization
of the generic case and this does not change the summation set.
\end{remark}
\begin{remark}
  \label{RkAlternateConverse}
  We will see in \cref{SectProd} a converse of \cref{PropAlternate}:
   for any weighted dependency graph with a super-multiplicative  function $\Psi$,
   \cref{EqFundamentalVariant} holds. In fact, a more general bound for cumulants
   of products of the $Y_\a$ holds; see \cref{EqBoundCumulantProducts,RkJustifAlternateConverse}.
\end{remark}
\begin{remark}
    \label{RmkDiscreteBernoulli}
    \cref{PropAlternate} is in particularly useful
    when $\WDep$ has no edges of weight $1$  and $Y_\a$ are Bernoulli variables.
    In this case, each connected component $B_i$ of the induced graph $\Dep[B_i]$
    contains only one distinct element $\beta_i$, with multiplicity $m_i \ge 1$.
    Then 
    \[\prod_{\a \in B_i} Y_\a= Y_{\beta_i}^{m_i}=Y_{\beta_i}, \]
    where the last equality comes from the assumption that $Y_{\beta_i}$ is a Bernoulli variable.
    Therefore, to prove that $\WDep$ is a $(\Psi,\bC)$ weighted dependency graph
    for the family $\{Y_\a, \a \in A\}$,
    it is enough to bound 
    $\ka(Y_\a,\a \in B)$, for {\em subsets} $B$ of $A$ (and not all {\em multisets}).
\end{remark}

\subsection{Small cumulants and quasi-factorization}
\label{SectSCQF}
Let $\ell \ge 1$ and $\uu=(u_\Delta)_{\Delta \subseteq [\ell]}$ be a family of 
real numbers indexed by subsets of $[\ell]$. We shall always assume $u_\emptyset\neq 0$.
Typically, $u_\Delta$ are the joint moments $\esper\left( \prod_{j \in \Delta} Y_j \right)$
of a family $(Y_1,\cdots,Y_\ell)$ of random variables,
but it is convenient not to assume this.

For any  subset $\Delta$ of $[\ell]$, we set
\begin{equation}
    \kappa_{\Delta}(\uu) := \sum_{\substack{\pi \in \PPP(\Delta) } } 
\mu(\pi, \{\Delta\}) \prod_{B \in \pi} \frac{u_B}{u_\emptyset}. 
\label{EqDefCumulants}
\end{equation}
If $\uu$ is the family of joint moments of $(Y_1,\cdots,Y_\ell)$,
then $u_\emptyset=1$ and $\kappa_{\Delta}(\uu)$ is simply the joint cumulant of the subfamily $\{Y_j, \, j \in \Delta\}$.

\begin{definition}\label{DefSmallCumulants}
    Fix some $\ell \ge 1$ and consider a {\em sequence} $\big(\uu^{(n)}\big)_{n\ge 1}$
    of lists, each indexed by subsets of $[\ell]$.
    Let also, for each $n \ge 1$, $\WDep^n$ be a weighted graph with vertex set $[\ell]$.
    We say that $\big(\uu^{(n)}\big)_{n\ge 1}$ has the $\WDep^n$ small cumulant property 
    if, for any subset $\Delta \subseteq [\ell]$ of size at least $2$, one has
    \begin{equation}\label{EqSmallCumulants}
      \big| \kappa_{\Delta}(\uu^{(n)}) \big| =
     \left(\textstyle \prod_{i \in \Delta} \frac{u_{\{i\}}}{u_\emptyset} \right) \cdot O\big( \MWST{\WDep^n[\Delta]} \big).
    \end{equation}
\end{definition}
\bigskip

Note that \cref{EqSmallCumulants} is similar to \cref{EqFundamental},
so that we are interested in establishing the small cumulant property.
We will see that it is equivalent to another property,
that we call quasi-factorization property and is in some cases easier to establish.
\medskip

We now assume that, for any $\Delta \subseteq [\ell]$, one has $u_\Delta \ne 0$.
Then we also introduce the auxiliary quantity $\P_\Delta(\uu)$ implicitly defined by
the property: 
for any subset $\Delta \subseteq [\ell]$,
\begin{equation}
     u_\Delta/u_\emptyset = \prod_{\delta \subseteq \Delta} \P_{\delta}(\uu).
    \label{EqUP}
\end{equation}
In particular, we always have $\P_\emptyset(\uu)=1$ and $\P_{\{i\}}(\uu)=u_{\{i\}}/u_\emptyset$. 
Using M\"obius inversion on the boolean lattice, we have explicitly:
for any subset $\Delta \subseteq [\ell]$ with $\Delta \ne \emptyset$,
\[ \P_\Delta(\uu)= \prod_{\delta \subseteq \Delta}                       
\left( \frac{u_\delta}{u_\emptyset} \right)^{(-1)^{|\Delta|-|\delta|}}
=\prod_{\delta \subseteq \Delta} \left( u_\delta \right)^{(-1)^{|\Delta|-|\delta|}}.\]
\begin{definition}\label{DefQuasiFactorization}
Fix some $\ell \ge 1$ and consider a {\em sequence} $\big(\uu^{(n)}\big)_{n\ge 1}$
of lists, each indexed by subsets of $[\ell]$, such that for $n$ large enough and any $\Delta \subseteq [\ell]$,
one has $u^{(n)}_\Delta \ne 0$.
We also consider, for each $n \ge 1$, a weighted graph $\WDep^n$ with vertex set $[\ell]$.
We say that $\big(\uu^{(n)}\big)_{n\ge 1}$
has the $\WDep^n$ quasi-factorization property
if, for any subset $\Delta \subseteq [\ell]$ of size at least $2$, one has
\begin{equation}\label{EqQuasiFactorization}
    \P_\Delta(\uu^{(n)}) = 1 + \O\big( \MWST{\WDep^n[\Delta]} \big).
\end{equation}
\end{definition}

The following proposition, generalizing
\cite[Lemma 2.2]{FerayRandomPermutationsCumulants},
is be used repeatedly in this article.
It says that the two above properties are equivalent.
\begin{proposition}
Let $\ell \ge 1$ and
$\WDep^n$ be a sequence of weighted graph, each with vertex set $[\ell]$.
We also consider a {\em sequence} $\big(\uu^{(n)}\big)_{n\ge n_0}$
of lists of real numbers, each indexed by subsets of $[\ell]$.
Finally assume that, for each $n \ge n_0$ and each $\Delta \subseteq [\ell]$, we have
$u^{(n)}_\Delta \ne 0$.

If $\big(\uu^{(n)}\big)_{n\ge n_0}$ has the $\WDep^n$ quasi-factorization property,
then it also has the $\WDep^n$ small cumulant property.
Assume moreover that the maximal weight of $\WDep^n$ tends to $0$.
Then the converse also holds:
$\big(\uu^{(n)}\big)_{n\ge n_0}$ has 
the $\WDep^n$ small cumulant property
if and only if it has
the $\WDep^n$ quasi-factorization property.
\label{PropEqSCQF}
\end{proposition}
\begin{proof}
   The proof is an adaptation of 
   the one of \cite[Lemma 2.2]{FerayRandomPermutationsCumulants}.
    \bigskip

    We first assume that $u_\emptyset=1$ and $u_{\{i\}}=1$ for all $i$ in $[\ell]$,
    so that the product in \cref{EqUP} can be taken over subsets $\delta$
    with $|\delta| \ge 2$.

    Let us start by the fact that the quasi-factorization property
    implies the small cumulant property.
    For $n \ge n_0$ and a subset $\Delta \subseteq [\ell]$
    we set $R^{(n)}_\Delta=\P_{\Delta}(\uu^{(n)})-1$.
    The quasi-factorisation property asserts that
    that 
    $R^{(n)}_\Delta = \O(\MWST{\WDep^n[\Delta]})$ whenever $|\Delta| \ge 2$.
    We need to prove that this implies the small cumulant property,
    {\em i.e.} that, for any $\Delta \subseteq [\ell]$ with $|\Delta| \ge 2$,
    we have \hbox{$\kappa_{\Delta}(\uu^{n}) =\O\big(\MWST{\WDep^n[\Delta]}\big)$}.
    It is in fact enough to prove it for $\Delta = [\ell]$.
    The case of smaller $\Delta$ then follows by considering
    a smaller family of sequence $\big(\uu^{(n)}\big)_{n\ge n_0}$, indexed by subsets of $\Delta$.

    Fix a set partition $\pi \in \PPP(\ell)$.
    For a block $B$ of $\pi$, one has, expanding the product in \eqref{EqUP}:
    \[ u^{(n)}_{B} = \prod_{\Delta \subseteq B \atop |\Delta| \ge 2} (1+R^{(n)}_{\Delta})
    = \sum_{\{\Delta_1,\dots,\Delta_m\}} R^{(n)}_{\Delta_1} \dots R^{(n)}_{\Delta_m},
    \]
    where the sum runs over all finite sets of (distinct) subsets of $B$
    of size at least $2$ (in particular, the size $m$ of the set is not fixed).
    Therefore,
    \[\prod_{B \in \pi} u^{(n)}_B= \sum_{\{\Delta_1,\dots,\Delta_m\}} R^{(n)}_{\Delta_1} \dots R^{(n)}_{\Delta_m},
    \]
    where the sum runs over all finite sets  
    of (distinct) subsets of $[\ell]$ of size at least $2$
    such that each $\Delta_i$ is contained in a block of $\pi$.
    In other terms, for each $i \in [m]$, $\pi$ must be coarser
    than the partition $\Pi(\Delta_i)$,
    which, by definition, has $\Delta_i$ and singletons as blocks.
    Finally, from \cref{EqDefCumulants}
    \begin{equation}\label{EqCumulantsSumSets}
      \kappa_{[\ell]}(\uu^{(n)}) = \sum_{\{\Delta_1,\dots,\Delta_m\} \atop \Delta_i \subseteq [\ell]} 
      R^{(n)}_{\Delta_1} \dots R^{(n)}_{\Delta_m} 
    \left( \sum_{\pi \in \PPP([\ell]) \atop \forall i, \ \pi \geq \Pi(\Delta_i)}
    \mu(\pi, \{[\ell]\}) \right).
\end{equation}
The condition on $\pi$ can be rewritten as
\[\pi \geq \Pi(\Delta_1) \vee \dots \vee  \Pi(\Delta_m).\]
Hence, by definition of the M\"obius function, the sum in the parenthesis
is equal to $0$, unless we have \hbox{$\Pi(\Delta_1) \vee \dots \vee  \Pi(\Delta_m) = \{[\ell]\}$}.
From \cref{LemmaSetPartAndPi1},
this implies the inequality
\[ \prod_{i=1}^m \MWST{\WDep^n[\Delta_i]} \le \MWST{\WDep^n}.\]
But recall that by hypothesis $R^{(n)}_\Delta = \O(\MWST{\WDep^n[\Delta]})$.
Therefore, 
if \hbox{$\Pi(\Delta_1) \vee \dots \vee  \Pi(\Delta_m) = \{[\ell]\}$}, then
\[R^{(n)}_{\Delta_1} \cdots R^{(n)}_{\Delta_m} =\O \left( \prod_{i=1}^m \MWST{\WDep^n[\Delta_i]} \right)
=\O( \MWST{\WDep^n} ).\]
In other words, all non-zero summands in \eqref{EqCumulantsSumSets} are $\O( \MWST{\WDep^n} )$.
Since the summation index set in \eqref{EqCumulantsSumSets} does not depend on $n$,
we conclude that $\kappa_{[\ell]}(\uu^{(n)})=O\big( \MWST{\WDep^n} \big)$,
which ends the proof of the first implication.
\bigskip

Let us now consider the converse statement.
We proceed by induction on $\ell$ and we assume that, 
for all $\ell'$ smaller than a given $\ell \geq 2$, 
the $\WDep_n$ small cumulant property implies the 
$\WDep_n$ quasi factorization property.

Consider a sequence of lists $(\uu^{(n)})_{n \ge n_0}$
such that, for any $\Delta \subseteq [\ell]$ with $|\Delta| \ge 2$,
one has $\ka_\Delta(\uu^{(n)}) = \O(\MWST{\WDep^n[\Delta]} )$.
By induction hypothesis, 
for all $\Delta \subsetneq [\ell]$,
one has $\P_\Delta(\uu^{(n)})-1 = \O(\MWST{\WDep^n[\Delta]} )$.

Since the maximal weight of $\WDep^n[\Delta]$ tends to $0$,
the quantity $\MWST{\WDep^n[\Delta]}$ also tends to $0$ for all $\Delta \subseteq [\ell]$ with $|\Delta| \ge 2$.
Thus $\P_\Delta(\uu^{(n)})$ tends to $1$.
From \cref{EqUP}, this implies that, for any $\Delta \subseteq [\ell]$,
the sequence $u^{(n)}_\Delta$ also tends to $1$.
This estimate is useful below.

Back to the proof, we have to establish that 
\[ \P_{[\ell]}(\uu^{(n)})-1 =
\prod_{\Delta \subseteq [\ell]} (u^{(n)}_\Delta)^{(-1)^{\ell-|\Delta|}} -1 =O\bigg( \MWST{\WDep^n}\bigg) \]
Thanks to the estimates above for $u^{(n)}_\Delta$, this is equivalent to the fact that
\begin{equation}\label{EqRewriteQuasiFact}
  u^{(n)}_{[\ell]} - \prod_{\Delta \subsetneq [\ell]} 
  (u^{(n)}_{\Delta})^{(-1)^{\ell-1-|\Delta|}} = O\bigg( \MWST{\WDep^n}\bigg)
  \end{equation}
  Define now an auxiliary family $(\vv^{(n)})_{n \ge n_0}$ defined by:
  \[v^{(n)}_\Delta = \begin{cases}
      u^{(n)}_\Delta & \text{ if }\Delta \subsetneq [\ell];\\
      \prod_{\delta \subsetneq [\ell]}          
      (u^{(n)}_{\delta})^{(-1)^{\ell-1-|\delta|}} & \text{ for }\Delta=[\ell].
  \end{cases}\]
Clearly, $\P_\Delta(\vv)=\P_\Delta(\uu)$ for $\Delta \subsetneq [\ell]$
and $\P_{[\ell]}(\vv)=1$, so that
the family $\vv$ has the $\WDep^n$ quasi-factorization property.
Thus, using the first
part of the proof, it also has the $\WDep^n$ small cumulant property.
In particular:
\[ \kappa_{[\ell]}(\vv^{(n)}) =O\bigg( \MWST{\WDep^n}\bigg).\]
But, by hypothesis
\[ \kappa_{[\ell]}(\uu^{(n)})  =O\bigg( \MWST{\WDep^n}\bigg).\]
As $v_\Delta=u_\Delta$ for $\Delta \subsetneq [\ell]$, one has:
\[u_{[\ell]} - v_{[\ell]} =
\kappa_{[\ell]}(\uu) - \kappa_{[\ell]}(\vv) = O\bigg( \MWST{\WDep^n}\bigg),\]
which proves \eqref{EqRewriteQuasiFact}.
\bigskip

The general case follows directly from the case $u_\emptyset=u_{\{i\}}=1$
by considering the family 
\[w^{(n)}_\Delta=\frac{(u_\Delta/u_\emptyset)}{\prod_{i \in \Delta} \frac{u_{\{i\}}}{u_\emptyset}}.\]
    Indeed, for $|\Delta| \ge 2$,
    \begin{align*}
        \P_\Delta(\ww)&=\P_\Delta(\uu);\\
        K_\Delta(\ww)&=K_\Delta(\uu)/\prod_{h \in \Delta} \left( \frac{u_{\{h\}}}{u_\emptyset} \right). \qedhere
    \end{align*}
\end{proof}
When the maximal weight in $\WDep^n$ tends to zero,
we write ``$\big(\uu^{(n)}\big)_{n\ge n_0}$ has the $\WDep^n$ SC/QF property''
(since the two properties are equivalent in this case).
Furthermore, when $\WDep^n$ is a complete graph with weight $\eps_n$ on each edge,
we say that ``$\big(\uu^{(n)}\big)_{n\ge n_0}$ has the $\eps_n$ SC/QF property''
(instead of the ``$\WDep^n$ SC/QF property'').
In the following lemma, we collect a few easy facts on the SC/QF property.

\begin{lemma}
    \begin{enumerate}
        \item If, for each $n$, $u^{(n)}_\Delta=u^{(n)}$ does not depend on $\Delta$,
            then $(\uu^{(n)})_{n \ge 1}$ has the $\bm{0}$-SC/QF property,
            where $\bm{0}$ stands for the  graph on vertex-set $[\ell]$ with no edges.
        \item If, for each $n$, $(\uu^{(n)})$ is multiplicative, 
            that is $u^{(n)}_\Delta=\prod_{i \in \Delta} u^{(n)}_{\{i\}}$,
            then $(\uu^{(n)})_{n \ge 1}$ has the $\bm{0}$-SC/QF property.
        \item Let $(\WDep^n)_{n \ge n_0}$ and $(\WK^n)_{n \ge n_0}$ two sequences
            of weighted graphs with maximal weight tending to $0$ and
            assume that the weight of $\{i,j\}$ in $\WDep^n$ is always smaller
            than or equal to the corresponding weight in $\WK^n$.

            If a sequence $(\uu^{(n)})_{n \ge 1}$ has the $\WDep^n$-SC/QF property,
            then it also has the $\WK^n$-SC/QF property.
        \item Consider two {\em sequences} $\big(\uu^{(n)}\big)_{n\ge n_0}$                 
            and $\big(\vv^{(n)}\big)_{n\ge n_0}$, both with the $\WDep^n$ SC/QF property.      
            Then their entry-wise product $\bm{u}^{(n)} \cdot \bm{v}^{(n)}$                   
            and their entry-wise quotient $\uu^{(n)}/\vv^{(n)}$ both have the $\WDep^n$ SC/QF property.
        \item Moreover, if $u_\emptyset=v_\emptyset$,
          then any linear combination $\la \uu^{(n)} + \mu \vv^{(n)}$
            with only non-zero terms for $n$ sufficiently large
            also has the $\WDep^n$-SC/QF property.
    \end{enumerate}
    \label{LemTrivialSCQF}
\end{lemma}
\begin{proof} 
    For (1) and (2), observe that $\P_\Delta\big(\bm{u}^{(n)})=1$. Item (3) is trivial.
    (4) follows from the following easy identities:
    for any $\Delta \subseteq [\ell]$ and $n$ sufficiently large (to avoid a division by $0$),
\[ \P_\Delta\big(\bm{u}^{(n)} \cdot \bm{v}^{(n)}\big) = \P_\Delta(\bm{u}^{(n)}) \cdot  \P_\Delta(\bm{v}^{(n)});
\quad
\P_\Delta\big(\uu^{(n)}/\vv^{(n)} \big) = \frac{\P_\Delta(\bm{u}^{(n)})}{\P_\Delta(\bm{v}^{(n)})}.\]
Moreover, if $u_\emptyset=v_\emptyset$,
\[\ka_\Delta\big(\la \uu^{(n)} + \mu \vv^{(n)}\big) = 
\frac{1}{(\la+\mu)^{|\Delta|}} \big(\la \ka_\Delta(\bm{u}^{(n)})
+ \mu \ka_\Delta(\vv^{(n)})\big),\]
which implies (5).
\end{proof}
We end this section by a family of examples, for which the SC/QF property holds.

Let $(X_n)_{n \ge 1}$ be a sequence of integers such that $X_n \ge 1$ (for all $n \ge 1$)
and $\lim_{n \to \infty} X_n=+\infty$. Fix $\ell \ge 1$ and nonnegative integers $a_1,\cdots,a_\ell$.
We consider the {\em factorial sequences}
\[u^{(n)}_\Delta(a_1,\cdots,a_\ell)= \left(X_n-\textstyle \sum_{i \in \Delta} a_i \right) !\]
For $n$ sufficiently large, say $n \ge n_0$, the integer $X_n - \sum_{i=1}^\ell a_i$ is non-negative
and the truncated family $\big(\uu^{(n)}(a_1,\cdots,a_\ell)\big)_{n\ge n_0}$ is well-defined.

\begin{proposition}
  We use the notation above and set $\eps_n=1/X_n$.
Then the family $\big(\uu^{(n)}(a_1,\cdots,a_\ell)\big)_{n\ge n_0}$
has the $\eps_n$ SC/QF property.
\label{PropExSCQF}
\end{proposition}
The proof is a combination of easy but technical inductions.
It is given in \cref{Sect:Proof_SCQF_Factorials}.

Combining this result with \cref{LemTrivialSCQF} (item 4),
we get that products and quotients of these factorial sequences have the SC/QF property.
Therefore, if the joint moments of some random variables 
are of this form,
we get bounds on their joint cumulants without any computation.
This is used in \cref{SectPerm,SectErdosRenyi,SectMatchings}.

\subsection{Powers of weighted dependency graphs}
\label{SectProd}
The propositions and lemmas of the two previous sections help to establish
that a family of random variables admits a given weighted dependency graph.
In this section, we shall see that when we have a weighted dependency graph
for a family $\{Y_\a,\a \in A\}$,
we can automatically construct a new one for 
monomials $Y_I = \prod_{\a \in I} Y_\a$ in the original variables $Y_\alpha$
(here, the index $I$ is a multiset of elements of $A$).

\begin{proposition}
    Let $\{Y_\a,\a \in A\}$ be a family of random variables
    with a $(\Psi,\bC)$ weighted dependency graph $\WDep$.
    We fix a positive integer $m$ and consider the $m$-the power $\WDep^m$
    of $WDep$, as defined in \cref{DefPowerWG}.

    Assume that $\Psi$ is super-multiplicative.
    Then $\WDep^m$ is a $(\bm{\Psi},\bD_m)$ weighted dependency graph for the family
    $\{Y_I,\, I \in \Mset_{\le m}(A)\}$,
    where:
    \begin{equation}
        \bm{\Psi}(\{I_1,\cdots,I_r\}) = \Psi(I_1 \uplus \cdots \uplus I_r)
    \end{equation}
    and $D_{m,r}$ depends only on
    $m$, $r$, $C_1$, \ldots, $C_{mr}$.
    \label{PropProducts}
\end{proposition}
\begin{proof}
    Let $I_1,\cdots, I_r$ be in $\Mset_{\le m}(A)$.
    As above, we use the formula of Leonov and Shiryaev 
    for cumulants of products \cite{LeonovShiryaevCumulants}
    (see also \cite[Theorem 4.4]{Sniady2006fluctuations}):
    \[ \ka\big( Y_{I_1},\cdots,Y_{I_m} \big) =
    \sum_{\pi \perp \bm{I}} \ \prod_{i=1}^\ell \ka( Y_\beta; \beta \in \pi_i),\]
    where the sum runs over multiset partitions $\pi$ of the multiset $I_1 \uplus \cdots \uplus I_r$
    such that 
    \begin{equation}
        \pi \vee \big\{ I_1,\cdots,I_r \big\} = \big\{ I_1 \uplus \cdots \uplus I_r \big\}.
        \label{EqCondIrred}
    \end{equation}
    Since $\WDep$ is a $(\Psi,\bC)$ weighted dependency graph for
    $\{Y_\a,\a \in A\}$, one has the bound
    \[ \bigg| \ka( Y_\beta; \beta \in \pi_i) \bigg| \le
    C_{|\pi_i|} \, \Psi(\pi_i) \, \MWST{\WDep[\pi_i]}.\]
    Hence, for any partition $\pi$ of $I_1 \uplus \cdots \uplus I_r$,
    \[ \left| \prod_{i=1}^\ell \ka( Y_\beta; \beta \in \pi_i) \right|
    \le \left( \prod_{i=1}^\ell C_{|\pi_i|} \right)
    \, \left( \prod_{i=1}^\ell \Psi(\pi_i) \right)
    \left( \prod_{i=1}^\ell \MWST{\WDep[\pi_i]} \right).\]
    But, from the super-multiplicativity of $\Psi$, one has 
    \[\prod_{i=1}^\ell \Psi(\pi_i) \le \Psi(\pi_1 \uplus \cdots \uplus \pi_\ell) = \Psi(I_1 \uplus \cdots \uplus I_r).\]
    On the other hand, from \cref{LemProdMWSTPi},
    when \eqref{EqCondIrred} is satisfied, 
    one has
    \[ \prod_{i=1}^\ell \MWST{\WDep[\pi_i]} \le \MWST{\WDep^m[\{I_1,\cdots,I_r\}]}.\]
    Bringing everything together, we have
    \begin{equation}
      | \ka\big( Y_{I_1},\cdots,Y_{I_m} \big) | \le
    \left( \sum_{\pi \perp \bm{I}}\  \prod_{i=1}^\ell C_{|\pi_i|} \right) \,
    \MWST{\WDep^m[ \{I_1,\cdots,I_r\}] } \, \Psi(I_1 \uplus \cdots \uplus I_r).
    \label{EqBoundCumulantProducts}
  \end{equation}
    The quantity $\left( \sum_{\pi \perp \bm{I}}\, \prod_{i=1}^\ell C_{|\pi_i|} \right)$
    only depends on the sizes of $I_1$, \ldots, $I_r$ and on the values of $C_1$, \ldots, $C_{mr}$
    and thus can be bounded by some $D_{m,r}$, depending on $m$, $r$, $C_1$, \ldots, $C_{mr}$.
\end{proof}

\begin{remark}
  \label{RkJustifAlternateConverse}
  When the $I_j$ are the connected components of the graph $\WDep_{1}[B]$,
  \cref{EqBoundCumulantProducts} specializes to \eqref{EqFundamentalVariant},
  which justifies \cref{RkAlternateConverse}.
\end{remark}
\begin{remark}
    In general we are only interested in a subfamily of $\{Y_I,\, I \in \Mset_{\le m}(A)\}$.
    But clearly, if we have a weighted dependency graph for some family of variables,
    then any subfamily admits the corresponding weighted subgraph
    as weighted dependency graph.
\end{remark}

\section{Crossings in random pair partitions}
\label{SectMatchings}

\subsection{Definitions and basic considerations}
\label{SubsectDefPP}
Recall that $[2n]$ denotes the set of integers $\{1,\cdots,2n\}$
\begin{definition}
    A pair partition of $[2n]$ is a set $H$ of disjoint
    2-element subsets of $[2n]$ whose union is $[2n]$.
\end{definition}
%
Observe that, by definition
for each integer $i$ in $[2n]$, there is a unique $j \ne i$ such
    $\{i,j\}$ is in $H$. We call $j$ the {\em partner} of $i$.

We are interested in the uniform model on pair partitions of $[2n]$.
A uniform random pair partition of $[2n]$ can be constructed as follows.
Take $i_1$ arbitrarily ({\em e.g.} $i_1=1$) and choose its partner $j_1$ uniformly at random
among numbers different for $i_1$ ({\em i.e.} each number different from
$i_1$ is taken with probability $1/(2n-1)$);
then take $i_2$ arbitrarily different from $i_1$ and $j_1$
and choose its partner $j_2$ uniformly at random
among numbers different from $i_1$, $j_1$ and $i_2$ 
(each such number is taken with probability $1/(2n-3)$); 
and so on, until all pairs are created.
In particular, given distinct numbers $i_1,\cdots,i_t$ and $j_1,\cdots,j_t$,
the probability that all pairs $(\{i_s,j_s\})_{s \le t}$
belong to a uniform random pair partition $H$ of $[2n]$ is
\[\frac{1}{(2n-1) \, \cdots \, ( 2n-2t+1 )}.\]
This simple observation is the key to find 
a weighted dependency graph associated to uniform
random pair partitions.

To illustrate the use of this weighted dependency graph,
we study a classical statistics on pair partitions,
called {\em crossing};
see, {\em e.g.}, \cite{CrossingsNestings} and references therein
for enumerative results on this statistics.
\begin{definition}
    A {\em crossing} in a pair partition $H$ is a quadruple $(i,j,k,l)$
    with $i<j<k<l$ such that $\{i,k\}$ and $\{j,l\}$
    belong to $H$.
%
\end{definition}
It is customary to represent pair partitions by putting the numbers
$1$, \ldots, $2n$ on a line and linking partners with an arch in the upper-half plane.
With this representation, crossings as defined above 
correspond to crossings of the corresponding arches.
For example $(1,4,5,7)$ and $(4,6,7,8)$ are the only two crossings of
$H_\ex=\big\{ \{1,5\},\{2,3\},\{4,7\},\{6,8\} \big\}$.
The corresponding graphical representation is given in \cref{FigPP}.
\begin{figure}[ht]
    \begin{center}
        \includegraphics{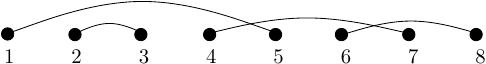}
    \end{center}
    \caption{Example of a pair partitions with two crossings}
    \label{FigPP}
\end{figure}

\subsection{A weighted dependency graphs for random pair partitions}
Let $A_n$ be the set of two element subsets of $[2n]$.
For $\{i,j\} \in A_n$,
we define a random variable $Y_{i,j}$
such that $Y_{i,j}=1$ if $\{i,j\}$ belongs to the random pair partition $\RPP_n$,
and $0$ otherwise.
\begin{proposition}
    Consider the weighted graph $\WDep$ on vertex set $A_n$ defined as follows:
    \begin{itemize}
        \item if two pairs $\a_1$ and $\a_2$ in $A_n$ have an element in common,
            then they are linked in $\WDep$ by an edge of weight $1$;
        \item if two pairs $\a_1$ and $\a_2$ in $A_n$ are disjoint,
            then they are linked in $\WDep$ by an edge of weight $1/n$.
    \end{itemize}

    Then $\WDep$
    is a $(\Psi_n,\bC)$ weighted dependency graph
    for the family $\{ Y_{i,j},\, \{i,j\} \in A_n \}$,
    where
    \begin{itemize}
        \item $\Psi_n(B)=n^{-\#(B)}$ for any multiset $B$
    of elements of $A_n$
        \item   and $\bC=(C_r)_{r \ge 1}$ is a sequence that does not depend on $n$.
    \end{itemize}  
    \label{PropWDGInRandomPairPartitions}
\end{proposition}
\begin{proof}
    Clearly $\Psi_n$ is super-multiplicative.
    From \cref{PropAlternate}, it is enough to prove that,
    for any multiset $B$ of elements of $A_n$ of size $r$, one has
\begin{equation}
    \left| \ka\left(\prod_{\a \in B_1} Y_{\a},
    \dots, \prod_{\a \in B_\ell}  Y_{\a}  \right) \right| \le
    D_r \, \Psi(B) \, \MWST{\WDep[B]}, 
    \label{EqFundamentalVariantPP}
\end{equation}
where $B_1$, \ldots, $B_\ell$ are the vertex sets of the connected components
of the graph $\WDepOne[B]$,
and for some $D_r$ that does not depend on $n$.
    
But, if $\a_1$ and $\a_2$ are different and linked by an edge of weight $1$,
the product $Y_{\a_1} Y_{\a_2}$ is identically equal to $0$.
Therefore the left-hand side of \cref{EqFundamentalVariantPP} is $0$ unless
each $B_i$ contains only one element (possibly with multiplicity $m$).
Since the $Y_\a$ take value in $\{0,1\}$, we have $Y_\a^m=Y_\a$ 
and the multiplicity does not play any role.

Finally, it is enough to prove
that for {\em disjoint} pairs $\a_1,\cdots,\a_r$ in $A_n$,
we have:
\begin{equation}
    |\ka(Y_{\a_1},\cdots,Y_{\a_r})| \le D_r \, \left(\tfrac{1}{n}\right)^{2r-1}.
    \label{EqToProvePP}
\end{equation}

Assume $n \ge r$, otherwise the above statement is vacuous.
From the discussion in \cref{SubsectDefPP}, we have that, for any subset $\Delta$ of $[r]$,
\[M^{(n)}_\Delta := \esper \prod_{i \in \Delta} Y_{\a_i} 
= \frac{1}{(2n-1)\, \cdots \, (2n-2|\Delta|+1)}
=\frac{(2n-2|\Delta|)! \, n! \, 2^{|\Delta|}}{(2n)!\, (n-|\Delta|)!}.\]
Note that it does not depend on $\a_1$, \ldots, $\a_r$.
From \cref{LemTrivialSCQF} (items 1, 2 and 3) and \cref{PropExSCQF}, each factor of the above expression
has the $\tfrac{1}{n}$ SC/QF property 
and thus $\MM^{(n)}=(M^{(n)}_\Delta)_{\Delta \subseteq [r]}$ also has this property.

Therefore, since $M^{(n)}_\emptyset=1$, one has:
\[ \ka_{[r]}(\MM^{(n)}) = \left( \textstyle\prod_{i=1}^r M^{(n)}_{\{i\}} \right) \cdot \O(n^{-r+1}).\]
But $\ka_{[r]}(\MM^{(n)})= \ka(Y_{\a_1},\cdots,Y_{\a_r})$ and, for each $i$,
one has $M^{(n)}_{\{i\}}=\tfrac{1}{2n-1}$, so that \cref{EqToProvePP} is proved.
\end{proof}

\subsection{Asymptotic normality of the number of crossings}
Let $A'_n$ be the set of quadruples $(i,j,k,l)$ of elements of $[2n]$ with $i<j<k<l$.
For $(i,j,k,l)$ in $A'_n$, we set
$Y'_{i,j,k,l}=Y_{i,k} Y_{j,l}$.
Equivalently, $Y'_{i,j,k,l}=1$ if $(i,j,k,l)$ is a crossing in the random
pair partition $\RPP_n$ and $0$ otherwise.
We also consider
\[\Cr_n = \sum_{i<j<k<l} Y'_{i,j,k,l},\]
which is the number of crossings in the random
pair partition $\RPP_n$.
We will prove the asymptotic normality of $\Cr_n$,
using the weighted dependency graph of the previous section.
\medskip

First, we use \cref{PropProducts} to find a weighted dependency graph
for the variables $Y'_{i,j,k,l}$.
For a multiset $B=\{Y'_{i_t,j_t,k_t,l_t}, \, 1 \le t \le |B|\}$,
we define $\pairs(B)$ as 
\[\pairs(B) = \# \bigg(
\big\{(i_t,k_t), \, 1 \le t \le |B|\} \cup \{(j_t,l_t), \, 1 \le t \le |B|\big\}\bigg).\]
This is the number of {\em distinct} $Y$ variables that appear in the $(Y'_{\a})_{\a \in B}$.

\begin{proposition}
Let $\WDep'$ be the complete graph on $A'_n$ with the following weights:
\begin{itemize}
    \item if two quadruples $\a'_1$ and $\a'_2$ have a non-empty intersection,
        they are linked by an edge of weight $1$;
    \item if they are disjoint, then they are linked by an edge of weight $1/n$.
\end{itemize}
Then $\WDep'$ is a $(\Psi'_n,\bC')$ weighted dependency graph for the family
$\{Y'_{i,j,k,l}, (i,j,k,l) \in A'_n\}$,
where $\Psi'_n(B)=n^{-\pairs(B)}$ and
$\bC'=(C'_r)_{r \ge 1}$ is a sequence that does not depend on $n$.
\end{proposition}
\begin{proof}
    This is a direct application of \cref{PropProducts}
    to the weighted dependency graph given in \cref{PropWDGInRandomPairPartitions}.
\end{proof}

We can now prove the asymptotic normality result.
\begin{theorem}
    As above, we denote $\Cr_n$ the number of crossings in 
    a uniform random pair partition of the set $[2n]$.
    Then, in distribution,
    \[\frac{\Cr_n-\esper \Cr_n}{\sqrt{\Var(\Cr_n)}} \to \N(0,1).\]
    \label{ThmAsympNormPP}
\end{theorem}
\begin{proof}
We use the notation of \cref{SubsecNormalityCriterion},
for the above described sequence of weighted dependency graphs.
We have 
\begin{equation}
    R_n= \sum_{\a \in A'_n} \Psi_n'(\{\a\}) = \binom{n}{4} \, \tfrac{1}{n^2} \asymp n^2.
    \label{EqBoundRPP}
\end{equation}
To find an upper bound for $T_{\ell,n}$, first fix $\a_1,\cdots,\a_\ell$ in $A'_n$.
We want to give an upper bound for
\begin{equation}
    \sum_{\beta \in A'_n} W(\{\beta\},\{\a_1,\cdots,\a_\ell \}) 
    \frac{\Psi'_n\big(\{\a_1,\cdots,\a_\ell,\beta\}\big) }{\Psi'_n\big(\{\a_1,\cdots,\a_\ell\}\big) }.
    \label{EqTech2}
\end{equation}
To do that let us split the sum into different parts
(all constants in $\O$ symbols in the discussion below
depend on $\ell$ but can be chosen independent from $\a_1,\cdots,\a_\ell$):
\begin{itemize}
    \item if $\beta$ has no element in common with any of the $\a_i$,
        then $W(\{\beta\},\{\a_1,\cdots,\a_\ell \})=1/n$ and 
        \[\frac{\Psi'_n\big(\{\a_1,\cdots,\a_\ell,\beta\}\big) }
        {\Psi'_n\big(\{\a_1,\cdots,\a_\ell\}\big) }=1/n^2.\]
        The number of such terms is obviously bounded by $\O(n^4)$
        (which bounds the total number of terms in $A'_n$)
        so that the total contribution of this case is $\O(n)$.
    \item Assume that $\beta$ has an element in common with at least one of the $\a_i$,
        but that we nevertheless have
        \[\pairs\big(\{\a_1,\cdots,\a_\ell,\beta\}\big)
        =\pairs\big(\{\a_1,\cdots,\a_\ell\}\big) +2.\] 
        In this case, we have $W(\{\beta\},\{\a_1,\cdots,\a_\ell \})=1$ and
        \[\frac{\Psi'_n\big(\{\a_1,\cdots,\a_\ell,\beta\}\big) }
        {\Psi'_n\big(\{\a_1,\cdots,\a_\ell\}\big) }=1/n^2.\]
        The number of such terms is bounded by $\O(n^3)$:
        indeed we should choose which element of which $\a_i$ is in common with 
        $\beta$ (constant number of choices)
        and then choose other elements of $\beta$ ($\O(n^3)$ choices).
        Finally, the total contribution of such terms is also $\O(n)$.
    \item We now look at the case where
        \[\pairs\big(\{\a_1,\cdots,\a_\ell,\beta\}\big)               
        =\pairs\big(\{\a_1,\cdots,\a_\ell\}\big) +1.\]
        This implies that $\beta$ has at least two elements in common with $\bigcup_{i \le \ell} \a_i$,
        thus the number of such terms is $\O(n^2)$.
        But in this case $W(\{\beta\},\{\a_1,\cdots,\a_\ell \})=1$ and                           
        \[\frac{\Psi'_n\big(\{\a_1,\cdots,\a_\ell,\beta\}\big) }                                      
        {\Psi'_n\big(\{\a_1,\cdots,\a_\ell\}\big) }=1/n,\]
        so that the total contribution of such terms is also $\O(n)$.
    \item The last case consist is $\beta \in A'_n$ such that
        \[\pairs\big(\{\a_1,\cdots,\a_\ell,\beta\}\big)               
        =\pairs\big(\{\a_1,\cdots,\a_\ell\}\big).\]
        This implies in particular that $\beta$ is included in $\bigcup_{i \le \ell} \a_i$,
        hence there is only a constant number of such terms.
        In this case $W(\{\beta\},\{\a_1,\cdots,\a_\ell \})=1$ and
        \[\frac{\Psi'_n\big(\{\a_1,\cdots,\a_\ell,\beta\}\big) }                                      
        {\Psi'_n\big(\{\a_1,\cdots,\a_\ell\}\big) }=1,\]
        so that the total contribution of such terms is $\O(1)$.
\end{itemize}
Finally, we see that, for any $\a_1,\cdots,\a_\ell$ in $A'_n$,
the quantity \eqref{EqTech2} is $\O(n)$, 
with a constant in $\O$ symbol depending on $\ell$, but not on $\a_1,\cdots,\a_\ell$.
Thus $T_{\ell,n}$ is $\O(n)$ 
and we can choose $Q_n=n$ in \cref{ThmMain}.
The variance of $\Cr_n$ is computed in \cref{AppVarPP}
and we see that $\sigma_n \asymp n^{3/2}$.
Therefore \eqref{EqHypoMainThm} is fulfilled for $s=3$
and we infer from \cref{ThmMain} the asymptotic normality of $\Cr_n$.
\end{proof}

\section{Erd\H{o}s-Rényi model $G(n,m)$}
\label{SectErdosRenyi}
\label{SectGraphs}
\subsection{The model}
For each $n$, let $m_n$ be an integer between $0$ and $\binom{n}{2}$.
As in \cref{ExWeightedDepGraphTriangle}, we consider the
Erd\H{o}s-Rényi random graph model $G(n,m_n)$,
\ie $G$ is a graph with vertex set $V:=[n]$ and an edge set
$E$ of size $m_n$, chosen uniformly at random among all 
possible edge sets of size $m_n$.

Set $p_n=m_n/\binom{n}{2}$.
For any 2-element subset $\{i,j\}$ of $V$,
we define a random variable $Y_{i,j}$
such that $Y_{i,j}=1$ if the edge $\{i,j\}$ belongs to the random graph $G$,
and $0$ otherwise.
Clear,  $Y_{i,j}=1$
with probability $p_n$.
However, unlike in $G(n,p_n)$, these random variables are not independent.
We can nevertheless compute their joint moments:
if $\a_1$,\ldots,$\a_r$ are distinct 2-element subsets of $V$,
then 
\[ \esper\big( Y_{\a_1}\, \dots\,Y_{\a_r}  \big) 
= \binom{E_n -r}{m_n-r} \bigg/ \binom{E_n}{m_n},\]
where $E_n=\binom{n}{2}$. Indeed, 
the numerator is the number of graphs with vertex set $[n]$ and $m_n$ edges
containing $\a_1$,\ldots,$\a_r$, while the denominator is the total number of
graphs with vertex set $[n]$ and $m_n$ edges.
This simple explicit formula for joint moments is
the starting point to find a weighted dependency graph in $G(n,m)$,
as we shall do in \cref{subsec:WDG_Gnm}.

We then use this dependency graph structure to give 
a new proof of Janson's central limit theorem
for subgraph count statistics in $G(n,m_n)$;
see \cref{subsec:CLT_Subgraph}.

\subsection{A weighted dependency graph in $G(n,m)$.}
\label{subsec:WDG_Gnm}
Let $A_n$ be the set of two element subsets of $[n]$.
\begin{proposition}
    Assume $m_n$ tends to infinity.
    Set $\eps_n=1/m_n$ and $\Psi_n(B)=p_n^{\#(B)}$ for any multiset $B$
    of elements of $A_n$.

    Then the complete graph on $A_n$ with weight $\eps_n$ on each edge
    is a $(\Psi_n,\bC)$ weighted dependency graph
    for the family $\{ Y_{i,j},\, \{i,j\} \in A_n \}$,
    where $\bC=(C_r)_{r \ge 1}$ is a sequence that does not depend on $n$.
    \label{PropWDGInGnm}
\end{proposition}
\begin{proof}
    Clearly $\Psi_n$ is super-multiplicative.
    From \cref{PropAlternate} --- see also \cref{RmkDiscreteBernoulli} ---, it is enough to prove that,
    for any {\em distinct} $\a_1,\cdots,\a_r$, one has
    \begin{equation}
        |\ka(Y_{\a_1},\cdots,Y_{\a_r})| \le C'_r \, \left(\tfrac{1}{m_n}\right)^{r-1} \, p_n^r, 
        \label{EqToProve}
    \end{equation}
    for some $C'_r$ that does not depend on $n$.

    If $\Delta$ is a subset of $[r]$, denote
    \[M^{(n)}_\Delta = \esper\left( \prod_{i\in \Delta} Y_{\a_i} \right).\]
    Recall from the previous section that this has an explicit expression:
\[M^{(n)}_\Delta = \binom{E_n -|\Delta|}{m_n-|\Delta|} \bigg/ \binom{E_n}{m_n}.\]
Note that it does not depend on $\a_1,\cdots,\a_r$.
    Moreover, as soon as $m_n \ge r$, which happens for $n$ big enough, say $n \ge n_0$, one can write
    \begin{equation}
      M^{(n)}_\Delta = \frac{(E_n -|\Delta|)! \, m_n!}{(m_n-|\Delta|)! \, E_n!}.
      \label{eq:MnDelta}
    \end{equation}
    We see $\MM^{(n)}:=(M^{(n)}_\Delta)_{\Delta \subseteq [r]}$ as a sequence of lists, 
    each indexed by subset of $[r]$, and 
    we use the notation and terminology of \cref{SectSCQF}.
    All the factors in \eqref{eq:MnDelta} have the $\tfrac{1}{m_n}$
    SC/QF property, and hence $\MM^{(n)}$ also has it --- see Lemmas \ref{LemTrivialSCQF} (items 1,3 and 4) and \ref{PropExSCQF}.
    Therefore, since $M^{(n)}_\emptyset=1$, one has:
    \[ \ka_{[r]}(\MM^{(n)}) =\left(\textstyle \prod_{i=1}^r M^{(n)}_{\{i\}}\right) \cdot  \O(m_n^{-r+1}).\]
    But $\ka_{[r]}(\MM^{(n)})= \ka(Y_{\a_1},\cdots,Y_{\a_r})$ and, for each $i$,
    one has $M^{(n)}_{\{i\}}=\tfrac{m_n}{E_n}=p_n$, so that \cref{EqToProve} is proved.
\end{proof}

\subsection{A CLT for subgraph counts in $G(n,m_n)$}
\label{subsec:CLT_Subgraph}
Fix some graph $H$ with at least one edge.
Let $A^H_n$ be the set of subgraphs $H'$ of
the complete graph $K_n$ on vertex set $[n]$
that are isomorphic to $H$:
there are $n(n-1)\cdots(n-v_H+1)/\Aut(H)$ such subgraphs,
where $\Aut(H)$ is the number of automorphisms of $H$.

As before, let $G$ be a random graph with the distribution of the model $G(n,m_n)$.
For $H'$ in $A^H_n$, we denote
\[Y_{H'} = \bm{1}_{H' \subset G} = \prod_{\{i,j\} \in E_{H'}} Y_{i,j}.\]
Then the random variable
\[X^H_n = \sum_{H' \in A^H_n} Y_{H'}\]
counts the number of subgraphs of $G$ that are isomorphic to $H$.
This is called the {\em subgraph count statistics}
and is a classical object of study in random graph theory
--- see, {\em e.g.}, \cite[Sections 3 and 6]{JansonRandomGraphs}.
The goal of this section is to prove the asymptotic normality of this statistics,
using weighted dependency graphs.

We first observe that
the above-defined family $\{Y_{H'}, \, H' \in A_n\}$ 
admits a weighted dependency graph.
To do that, if $B=\{H'_1,\cdots,H'_r\}$ is a multiset
of elements of $A^H_n$, we define $e(B)$ as the total number of edges in this multiset,
that is:
\[e(B) = \left| \bigcup_{i=1}^r E_{H'_i} \right|.\]
\begin{proposition}
    Assume $m_n$ tends to infinity.
    Set $\eps_n=1/m_n$ and 
    $\Psi_n(B)=p_n^{e(B)}$ for any multiset $B$
    of elements of $A^H_n$.
    
    Consider the complete graph with vertex set $A^H_n$ and assign weights on edges as follows:
    \begin{itemize}
        \item 
    if two copies $H'_1$ and $H'_2$ of $H$ have an edge in common
    (as subgraphs of $K_n$), then the edge $(H'_1,H'_2)$ gets weight $1$;
\item otherwise, the edge $(H'_1,H'_2)$ gets weight $1/m_n$.
    \end{itemize}
    We denote the resulting weighted graph $\WDep^H$.

    Then $\WDep^H$ is a $(\Psi_n,\bC)$ weighted dependency graph
    for the family $\{Y_{H'}, \, H' \in A^H_n\}$,
    for some sequence $\bC=(C_r)_{r \ge 1}$ that does not depend on $n$
    (but depends on $H$).
    \label{PropWDGSubgraphsInGnm}
\end{proposition}
\begin{proof}
    Indeed, $\WDep^H$ is a subgraph of the $e_H$-th power of
    the weighted dependency graph $\Dep$ given in \cref{PropWDGInGnm}
    --- see \cref{PropProducts}.
\end{proof}
We use the notation of \cref{ThmMain}.
Then we have 
\begin{equation}
    R_n= \sum_{H' \in A^H_n} \Psi(\{H'\})=
\frac{n(n-1)\cdots(n-v_H+1)}{\Aut(H)} \, p_n^{e_H} \asymp n^{v_H} \, p_n^{e_H}.
\label{EqRGnm}
\end{equation}

Estimates of $T_{\ell,n}$ and the variance $\Var(X^H_n)$ are given in the \cref{LemTSigmaGnm} below.
Let us introduce the notation involved in these estimates.
\begin{itemize}
  \item 
As in \cite{JansonRandomGraphs}, we denote
\begin{equation}
    \Phi_H = \min_{K \subseteq H, e_K>0} n^{v_K} \, p_n^{e_K}.
    \label{EqDefPhiH}
\end{equation}
In particular, $\Phi_H \le n^2 p_n$: indeed, $H$ has at least one subgraph $K$
  with two vertices and one edge.
In the following, we assume $\Phi_H$ tends to infinity.
\item
We also consider the following quantity:
\[\tPhi_H = \min_{K \subseteq H, e_K>1} n^{v_K} \, p_n^{e_K}.\]
Note that, unlike in the definition of $\Phi$,
the minimum is taken over graphs $K$ with {\em at least 2 edges}.
In the following, we assume that the graph $L_2$ with three vertices and two edges
is included in $H$ --- see a discussion on this hypothesis at the end of the Section.
In particular, this implies that $\Phi_H, \tPhi_H \le n^3\, p_n^2$
and $n^3\, p_n^2 \to \infty$ (since $\Phi_H \to \infty$).
\end{itemize}
\begin{lemma}
    Fix $\ell \ge 1$. Then
    \begin{equation}
        T_{\ell,n} \le  C_{H,\ell} \, \frac{n^{v_H} \, p_n^{e_H}}{\Phi_H}, 
        \label{EqBoundTGnm}
    \end{equation}
    for some constant $C_{H,\ell}$ depending on $H$ and $\ell$, but not on $n$.

    Assume furthermore $n(1-p_n)^2 \gg 1$.
    Then we have the following estimate for the variance:
    \begin{equation}
        \Var(X^H_n) \ge C \, \frac{(n^{v_H} \, p_n^{e_H})^2}{\tPhi_H}\, (1-p_n)^2,
        \label{EqAympVarianceGnm}
    \end{equation}
    for some constant $C>0$ and $n$ sufficiently large.
    \label{LemTSigmaGnm}
\end{lemma}
\begin{remark}
    Note in particular that, in many case ({\em e.g.} $p_n=p$ constant)
    the variance of $\Var(X^H_n)$
    has a different order of magnitude than in the 
    {\em independent model} $G(n,p_n)$.
    This phenomenon has already been observed by Janson 
    \cite{JansonTCLSubgraphsGnt}.
\end{remark}
\begin{proof}
    We prove here only \cref{EqBoundTGnm}.
    The proof of \cref{EqAympVarianceGnm} is postponed to \cref{AppVarGnm}.
    
    We denote $\Lambda=\WDep^H_{1}$ the subgraph of $\WDep^H$ formed by edges of weight $1$.
    Since $\WDep^H$ has only edges of weight $1$ and $1/m_n$, we have:
    \begin{multline}
        T_{\ell,n} = \max_{\substack{H'_1,\ldots,H'_{\ell} \in A^H_n}} 
        \left( \sum_{H'' \in \cN_{\Lambda}(H'_1,\ldots,H'_{\ell})} 
    \frac{\Psi\big(\{H'_1,\cdots,H'_{\ell},H''\}\big)}{\Psi\big(\{H'_1,\cdots,H'_{\ell} \}\big)}  
    \right. \\  \left.
    + \, \frac{1}{m_n} \cdot \sum_{H'' \not\in \cN_{\Lambda}(H'_1,\ldots,H'_{\ell})} 
        \frac{\Psi\big(\{H'_1,\cdots,H'_{\ell},H''\}\big)}{\Psi\big(\{H'_1,\cdots,H'_{\ell} \}\big)}\right).
        \label{EqRappelDefT}
    \end{multline}
    Fix some  $H'_1$, $\ldots$, $H'_{\ell}$ in $A^H_n$
    and consider the first summand in the above definition.
    In this summand, we sum over graphs $H''$ in $\cN_{\Lambda}(H'_1,\ldots,H'_{\ell})$,
    that is over graphs $H''$ with vertex set included in $[n]$
    that have at least an edge in common with either
    $H'_1$, $H'_2$, \ldots or $H'_{\ell}$.
    Denote $K$ the intersection of $H''$ with the union $\bigcup_{i=1}^\ell H'_i$.
    Then
    \[
        \frac{\Psi\big(\{H'_1,\cdots,H'_{\ell},H''\}\big)}{\Psi\big(\{H'_1,\cdots,H'_{\ell} \}\big)}
        = p_n^{e\left[H'' \cup \left (\bigcup_{i=1}^\ell H'_i \right) \right]
        - e\left[ \left (\bigcup_{i=1}^\ell H'_i \right) \right]} 
     = p_n^{e_{H''} - e\left[H'' \cap \left (\bigcup_{i=1}^\ell H'_i \right) \right]}
     = p_n^{e_H-e_K}.
     \]
     On the other hand, for a fixed $K$, the number of graphs $H''$ with $V_{H''} \subset [n]$,
     which are isomorphic to $H$, and whose intersection
     with $\bigcup_{i=1}^\ell H'_i$ is given by $K$ is bounded by $(\ell v_H)^{v_K} n^{v_H-v_K}$.
     Indeed the latter is an upper bound for the number of ordered choices of $v_H$ vertices,
    the first $v_K$ of them among the vertices of $\bigcup_{i=1}^\ell H'_i$ 
    (which has at most $\ell v_H$ vertices) and the last $v_H-v_K$ are chosen freely in $[n]$.
    Therefore
    \[
        \sum_{H'' \in \cN_{\Lambda}(H'_1,\ldots,H'_{\ell})}                                         
    \frac{\Psi\big(\{H'_1,\cdots,H'_{\ell},H''\}\big)}{\Psi\big(\{H'_1,\cdots,H'_{\ell} \}\big)}
    \le \sum_{K \subseteq H} (\ell v_H)^{v_K} n^{v_H-v_K} p_n^{e_H-e_K}
    \le D_{H,\ell} \frac{n^{v_H} p_n^{e_H}}{\Phi_H},
    \]
    where $D_{H,\ell}$ is a constant depending only on $H$ and $\ell$.
\medskip

Consider now the second summand in \cref{EqRappelDefT} 
($H'_1$, \ldots, $H'_{\ell}$ are still fixed).
Here we sum over graphs $H''$ which are not in $\cN_{\Lambda}(H'_1,\ldots,H'_{\ell})$,
which means that they do not share any edge with any of the $H'_i$.
In this case
 \[\frac{\Psi\big(\{H'_1,\cdots,H'_{\ell},H''\}\big)}{\Psi\big(\{H'_1,\cdots,H'_{\ell} \}\big)}
 = p_n^{e(H)}.\]
 There are at most $n^{v_H}$ graphs $H''$ with $V_{H''} \subset [n]$ that are isomorphic to $H$
 and we shall use this upper bound for the number of $H''$ not in $\cN_{\Lambda}(H'_1,\ldots,H'_{\ell})$.
 Therefore
 \[\frac{1}{m_n}\sum_{H'' \notin \cN_{\Lambda}(H'_1,\ldots,H'_{\ell})} 
     \frac{\Psi\big(\{H'_1,\cdots,H'_{\ell},H''\}\big)}{\Psi\big(\{H'_1,\cdots,H'_{\ell} \}\big)}
    \le \frac{n^{v_H} p_n^{e(H)}}{m_n} \asymp n^{v_H-2} p_n^{e_H-1}.\]
  Recall that $\Phi_H \le n^2 p_n$   so that $n^{v_H-2} p_n^{e_H-1} \le \frac{n^{v_H} p_n^{e_H}}{\Phi_H}$.
 \medskip

Putting everything together, we get that 
\(T_{\ell,n} \le C_{H,\ell} \frac{n^{v_H} p_n^{e_H}}{\Phi_H},\)
as claimed.
\end{proof}

We can now establish the following central limit theorem,
originally proved by Janson \cite{JansonTCLSubgraphsGnt,JansonOrthogonalDecomposition}.
\begin{theorem}
  \label{Thm:CLT_SubgraphCounts}
    \cite[Theorem 19]{JansonOrthogonalDecomposition}
    Let $m_n$ be an integer sequence tending to infinity with $m_n \le \binom{n}{2}$.
    Set $p_n=m_n/\binom{n}{2}$ and consider a random graph $G$
    taken with Erd\H{o}s-Rényi distribution $G(n,m_n)$.
    
    Fix some graph $H$ that contains $L_2$.
    Assume $\Phi_H$ tends to infinity and that for some $\eps >0$,
    we have
    $n^{1-\eps}(1-p_n)^2 \gg 1$.
    We denote $X^H$ the number of copies of $H$ in the random graph $G$.

    Then, in distribution
    \[\frac{X^H - \esper X^H}{\sqrt{\Var X^H}} \to \N(0,1).\]
\end{theorem}
\begin{proof}
    Since, for each $n \ge 1$, the family $\{Y_{H'}, \, H' \in A^H_n\}$ admits a weighted dependency graph
    --- \cref{PropWDGSubgraphsInGnm} ---, it is enough to check the hypothesis of our normality
    criterion, \cref{ThmMain}.

    From \cref{LemTSigmaGnm}, one can choose $Q_n=\frac{n^{v_H} \, p_n^{e_H}}{\Phi_H}$,
    while $\sigma_n^2$ is bounded from below by \cref{EqAympVarianceGnm}.
    We therefore have
    \[\tfrac{Q_n}{\sigma_n} \le C^{-1/2} \frac{\sqrt{\tPhi_H}}{\Phi_H (1-p_n)}.\]
    Note also that $\tfrac{R_n}{Q_n}=\Phi_H \le n^2$.
    We distinguish two cases.
    \begin{itemize}
        \item If the minimum in \eqref{EqDefPhiH} (the definition of $\Phi_H$) is achieved by the graph $H$
            with two vertices and one edge, then $\Phi_H=n^2 p_n$ and 
            we use the inequality $\tPhi_H \le n^3 p_n^2$.
            Thus
            \[\tfrac{Q_n}{\sigma_n} \le C^{-1/2} \frac{1}{n^{1/2} (1-p_n)} \le C^{-1/2} \frac{1}{n^{\eps/2}}.\]
            In particular \eqref{EqHypoMainThm} is fulfilled for any integer $s \ge 4/\eps$.
        \item Otherwise, one has $\Phi_H= \tPhi_H$. We also know that $p_n$ tends to $0$
            (otherwise $n^2 p_n$ clearly minimizes \eqref{EqDefPhiH}), so that
            \[\tfrac{Q_n}{\sigma_n} \le 2\, C^{-1/2} \frac{1}{\sqrt{\Phi_H}}.\]
            Since $\Phi_H$ tends to infinity, \eqref{EqHypoMainThm} is fulfilled for $s=3$.\qedhere
    \end{itemize}
\end{proof}

\begin{remark}
  [Discussion of the hypotheses]
  The hypothesis ``$\Phi_H \to \infty$'' is clearly necessary for asymptotic normality:
otherwise, with probability not tending to zero,
$G(n,m_n)$ does not contain any copy of $H$ \cite[Section 3.1]{JansonRandomGraphs},
which rules out the possibility that $X^H_n$ satisfies a central limit theorem.

On the other hand, the hypotheses ``$H$ contains a copy of $L_2$'' and
``$n^{1-\eps}(1-p_n)^2 \gg 1$'' are limits of our method.
Indeed, Janson prove asymptotic normality with the less restrictive
hypotheses ``$n^3 p_n^2 \to \infty$'' and ``$n^3 (1-p_n)^2 \to \infty$".

Janson describes also the limit distributions of {\em induced subgraph counts}
\cite[Theorems 21 and 23]{JansonOrthogonalDecomposition}.
Some of these results could be also derived with weighted dependency graphs,
but certainly not all since the limit law is not always Gaussian.

The method presented in this article has nevertheless an important advantage:
it can be applied to other combinatorial objects where a coupling with an independent model
is not available, as illustrated in the other sections of this article.
\end{remark}

\section{Random permutations}
\label{SectPerm}
\subsection{A weighted dependency graph for random permutations}
We consider in this section a uniform random permutation $\Pi_n$ of size $n$.
Let $A_n$ be the set $[n]^2$. For $(i,l) \in A_n$, we denote 
\[Y_{i,l} =\begin{cases}
    1 &\text{if } \Pi_n(i)=l;\\
    0 & \text{otherwise.}
\end{cases}\]
Joint moments of these variables have simple expressions.
If either $i=j$ or $l=k$, but not both, then $Y_{i,l}$ and $Y_{j,k}$
are incompatible, {\em i.e.} $Y_{i,l}\, Y_{j,k}=0$.
Moreover, if we consider distinct integers $i_1,\cdots,i_r$ and $l_1,\cdots,l_r$,
then 
\begin{equation}
    \esper \left( \prod_{h=1}^r Y_{i_h,l_h} \right) = \frac{1}{n \, (n-1)\, \cdots \, (n-r+1)}.
    \label{EqJointMomentPerm}
\end{equation}

\begin{proposition}
    Consider the weighted graph $\WDep$ on vertex set $A_n$ defined as follows:
    \begin{itemize}
        \item if two pairs $\a_1=(i_1,l_1)$ and $\a_2=(i_2,l_2)$ in $A_n$
            satisfy either $i_1=i_2$ or $l_1=l_2$,
            then they are linked in $\WDep$ by an edge of weight $1$.
        \item otherwise,
            they are linked in $\WDep$ by an edge of weight $1/n$.
    \end{itemize}

    Then $\WDep$
    is a $(\Psi_n,\bC)$ dependency graph,
    for the family $\{ Y_{i,l},\, (i,l) \in A_n \}$,
    where
    \begin{itemize}
        \item $\Psi_n(B)=n^{-\#(B)}$ for any multiset $B$ of elements of $A_n$
        \item   and $\bC=(C_r)_{r \ge 1}$ is a sequence that does not depend on $n$.
    \end{itemize}  
    \label{PropWDGInRandomPerm}
\end{proposition}
\begin{proof}
    The proof is similar to that of \cref{PropWDGInRandomPairPartitions,PropWDGInGnm}.
    Again $\Psi_n$ is clearly multiplicative and $Y_{\a_1}\, Y_{\a_2}=0$
    whenever $\a_1$ and $\a_2$ are linked by an edge of weight $1$,
    so that it is enough to prove the following (analogue of \cref{EqToProvePP}):
  for disconnected $\a_1,\cdots,\a_r$, one has
 \begin{equation}
    |\ka(Y_{\a_1},\cdots,Y_{\a_r})| \le D_r \, \left(\frac{1}{n}\right)^{2r-1}.
    \label{EqToProvePerm}
 \end{equation}
 This inequality is proved exactly as in \cref{PropWDGInRandomPairPartitions},
 using the explicit expression \cref{EqJointMomentPerm} for joint moments.
\end{proof}
Using \cref{PropProducts}, we also have dependency graphs for monomials in the variables $Y_{i,l}$.
In particular, in \cref{SubsecDoubleIndexPermStat}, we consider degree $2$ monomials $Y_{i,j}\, Y_{k,l}$.
Following \cref{SectProd}, we denote:
\begin{itemize}
    \item $A'_n:=\Mset_{2}(A_n)$ is the set of multisets of size $2$ of elements of $A_n$.
    \item $\WDep^2$ is the complete graph on $A'_n$ such that
        the weight of the edge between $\{\a_1,\a_2\}$ and $\{\beta_1,\beta_2\}$
        is $1$ if some $\a_i$ shares its first, respectively second,
        element with some $\beta_j$ and $1/n$ otherwise.
      \item $\bm{\Psi}$ is the function of multiset of $A'_n$ defined by:
        $\bm{\Psi}(\{\a'_1,\cdots,\a'_r\})=n^{-p(\{\a'_1,\cdots,\a'_r\})}$,
        where $p(\{\a'_1,\cdots,\a'_r\})=\#(\a'_1 \cup \cdots \cup \a'_r)$ is the number of distinct pairs in 
        $\a'_1 \cup \cdots \cup \a'_r$.
\end{itemize}
\begin{proposition}
    \label{PropWDGInRandomPermForPairs}
    The weighted graph $\WDep^2$ is a $(\bm{\Psi},\bD)$-dependency graph
    for the family of random variables $\{ Y_{i,l} Y_{j,k},\, \{(i,l),\, (j,k)\} \in A'_n \}$,
 where $\bD=(C_r)_{r \ge 1}$ is a sequence that does not depend on $n$.
\end{proposition}
\begin{remark}
    This weighted dependency graph and its powers (see \cref{PropProducts})
    correspond to the bounds on cumulants given in
    \cite[Theorem 1.4]{FerayRandomPermutationsCumulants}.
    Thanks to the results of this article, proving these bounds on cumulants
    is now easier (in particular we do not need to consider
    {\em truncated cumulants} anymore as in
    \cite[Section 2.4]{FerayRandomPermutationsCumulants}).
    Yet, some ideas of this article dedicated to random permutations
    are crucial here to build the general theory of weighted dependency graphs.
\end{remark}

\begin{remark}
 When $\Pi_n$ is distributed with {\em Ewens distribution} --- see, {\em e.g.},
 \cite{SurveyUniformPermutations} for background on this measure ---,
 the family $\{ Y_{i,l},\, (i,l) \in A_n \}$ still admits a weighted dependency graph.
 The only difference is that $Y_{i,l}$ and $Y_{j,k}$ share an edge of weight $1$ as
 soon as $\{i,l\} \cap \{j,k\} \neq \emptyset$.
 Nevertheless, most central limit theorems for Ewens distribution can be inferred
 from a corresponding central limit theorem for uniform random permutations
 using a coupling argument
 (the Chinese restaurant process yields a coupling between Ewens distributed permutations
 and uniform permutations, where only $\O_p(\ln(n))$ values differ).
 Therefore we have decided to restrict here to the uniform model.
\end{remark}

\subsection{A functional central limit theorem for simply indexed permutation statistics}
\label{SubsecSingleIndexPermStat}
In this section, we prove a weaker version
of a functional central limit theorem, due to Barbour and Janson \cite{BarbourJansonFunctionalCLT}.

Let $(a_0^{(n)}(i,l))_{i,l \le n}$ ($n \ge 1$) be a sequence of real matrices.
Take $t$ in $[0,1]$, an integer $n$ and a permutation $\pi$ of size $n$.
If $nt$ is an integer,
then we define
\[X^\pi_{n}(t)= \sum_{i=1}^{nt} a_0^{(n)}(i,\pi(i)).\]
We then extend $X^\pi_n$ to a continuous function on $[0,1]$,
by requiring that $X^\pi_n$ is affine on each interval $[j/n,(j+1)/n]$ (for $0 \le j \le n-1$).
More explicitly we set, for $t$ in $[0,1]$,
\[X^\pi_{n}(t)= \sum_{i=1}^{\lfloor nt \rfloor} a_0^{(n)}(i,\pi(i)) + 
(nt- \lfloor nt \rfloor) a_0^{(n)}(\lfloor nt \rfloor+1,\pi(\lfloor nt \rfloor+1)),\]
where $\lfloor x \rfloor$ denotes, as usual, the integer value of $x$.

Consider now a uniform random permutation $\Pi$ of size $n$ and set $X_{n}=X^\Pi_n$.
Then $X_n$ is a random continuous function on $[0,1]$
and we want to study its asymptotics.

The quantity $X_n(1)=\sum_{i=1}^{n} a_0^{(n)}(i,\pi(i))$ is a classical combinatorial statistics on permutation,
originally introduced by Hoeffding \cite{HoeffdingCombinatorialCLT},
while the process $X_n$ is a slight deformation of the one considered by Barbour and Janson in \cite{BarbourJansonFunctionalCLT}
(theirs is a step function, while ours is continuous piecewise-affine).

We now perform a centering by defining
\[a^{(n)}(i,l)=a_0^{(n)}(i,l) - n^{-1} \sum_{k=1}^n a_0{(n)}(i,k).\]
Then, for all $i$ and $n$, $\sum_{k=1}^n a^{(n)}(i,k)=0$ and, for $t$ in $[0,1]$,
\[X_n(t) - \esper X_n(t) = \sum_{i=1}^{nt} a^{(n)}(i,\Pi(i)).\]

We assume that:
\begin{itemize}
    \item the entries of the matrices $a^{(n)}$ are uniformly bounded by a constant $M$;
    \item The functions $f_n$ and $g_n$ defined by
        \begin{align}
            f_n(t) &= n^{-2} \sum_{i=1}^{\lfloor nt \rfloor} \sum_{l=1}^n \big(a^{(n)}(i,l)\big)^2,
            \label{EqDefF}\\
            g_n(t,u) &= n^{-3} \sum_{i=1}^{\lfloor nt \rfloor}
            \sum_{j=1}^{\lfloor nt \rfloor} \sum_{l=1}^n a^{(n)}(i,l) \, a^{(n)}(j,l)
            \label{EqDefG}
        \end{align}
        have pointwise limits $f$ and $g$.
\end{itemize}
Note that these hypotheses are in particular fulfilled when $a^{(n)}(i,l)=\a(i/n,l/n)$ for some
fixed piecewise continuous function $\a:[0,1]^2 \to \R$ independent of $n$.
The latter is a natural hypothesis to get a limit for a renormalized version of $X_n$.

We consider convergence in the space $C[0,1]$ of real-valued continuous functions on $[0,1]$,
endowed with the uniform metric. Denote $t \wedge u=\min(t,u)$.
\begin{theorem}
    We use the notation and assumptions above.
    Then there exists a zero-mean continuous Gaussian process $Z$ on $[0,1]$
    with covariance function given by
    \[\Cov(Z(t),Z(u)) = \si(t,u) := f(t \wedge u)-g(t,u)\]
    and, in distribution in $C[0,1]$, we have
    \[\frac{X_n(t) - \esper X_n(t)}{\sqrt{n}} \to Z.\]
    \label{ThmFunctionalCLTSimplePerm}
\end{theorem}
\begin{proof}
    The first step is to prove the convergence of the finite-dimensional laws
    (note that this step does not require the existence of $Z$).
    We do that by proving the convergence of joint cumulants;
    since a multidimensional
    Gaussian vector is determined by its joint moments, 
    this is enough to establish convergence in distribution.

    Both sides are centered so that there is nothing to prove for the expectation.

    For covariances, first write, for $t \in [0,1]$,
    \[\widetilde{X_n(t)}:= \frac{X_n(t) - \esper X_n(t)}{\sqrt{n}}
    = n^{-1/2} \sum_{i=1}^{\lfloor nt \rfloor} \sum_{l=1}^n a^{(n)}(i,l) Y_{i,l}.\] 
    Then for $0 \le t \le u \le 1$, we have
    \begin{equation}
        \Cov\left(\widetilde{X_n(t)}, \widetilde{X_n(u)}\right) = \esper \left(\widetilde{X_n(t)}\, \widetilde{X_n(u)}\right)
        = n^{-1} \sum_{i=1}^{\lfloor nt \rfloor} \sum_{j=1}^{\lfloor nu \rfloor} 
        \left[ \sum_{1 \le l,k \le n} a^{(n)}(i,l) a^{(n)}(j,k) \esper(Y_{i,l} \, Y_{j,k}) \right]
    \label{EqCovSinglePerm}
\end{equation}
If $i=j$, then $\esper(Y_{i,l} \, Y_{j,k})=\tfrac{1}{n}$ if $l=k$ and $0$ otherwise.
Thus the expression
in the bracket reduces to $n^{-1} \sum_{l=1}^n a^{(n)}(i,l)^2$
and the total contribution of terms with $i=j$ in \eqref{EqCovSinglePerm} is $f_n(t \wedge u)$.

On the other hand, if $i \ne j$ then $\esper(Y_{i,l} \, Y_{j,k})=0$ if $l=k$ and $\tfrac{1}{n(n-1)}$ otherwise.
Thus, for $i \ne j$
\[
\left[ \sum_{1 \le l,k \le n} a^{(n)}(i,l) a^{(n)}(j,k) \esper(Y_{i,l} \, Y_{j,k}) \right]
= \frac{1}{n(n-1)} \sum_{1 \le l,k \le n \atop l \ne k} a^{(n)}(i,l) a^{(n)}(j,k).\]
Since $a^{(n)}$ is centered, the same sum without the restriction $l \ne k$ equals to $0$.
Thus, the sum with condition $l \ne k$ is the opposite of the sum with condition $l=k$ and, if $i \ne j$, one has
\[\left[ \sum_{1 \le l,k \le n} a^{(n)}(i,l) a^{(n)}(j,k) \esper(Y_{i,l} \, Y_{j,k}) \right]    
= \frac{-1}{n(n-1)} \sum_{l=1}^n a^{(n)}(i,l) a^{(n)}(i,l).\]
As a consequence, the total contribution of terms with $i \ne j$ in \eqref{EqCovSinglePerm} is
\[ -\frac{n}{n-1} g(t,u) + \frac{1}{n-1} f(t \wedge u).\]

Finally we get the piecewise limit 
\[\lim_{n \to \infty} \Cov\left(\widetilde{X_n(t)}, \widetilde{X_n(u)}\right) = f(t \wedge u)-g(t,u),\]
as wanted.
\medskip

Let us now consider higher order cumulants.
Recall that the family
$\{Y_{(i,l)}, (i,l) \in A_n\}$ admits $\WDep$ as a 
$(\Psi_n,\bC)$ weighted dependency graph where $\WDep$, $\Psi_n$ and $\bC$
are defined in \cref{PropWDGInRandomPerm}.
Since $a^{(n)}(i,l)$ is uniformly bounded by $M$, the family 
\[ \big\{ a^{(n)}(i,l) Y_{(i,l)}, (i,l) \in A_n \big\}\]
has the same dependency graph, replacing simply $\Psi_n$ by
\[\Psi'_n(B):=M^{|B|} \Psi_n(B).\]
For this dependency graph, using the notation of \cref{SubsecNormalityCriterion}, one has
\[R_n = \sum_{(i,l) \in A_n} \frac{M}{n} = M \, n.\]
Let us now establish a bound for $T_{r,n}$. Fix $\a_1,\cdots,\a_r$ in $A_n$
($\a_h=(i_h,l_h)$ for $h \le \ell$)
and consider the sum
\begin{equation}
    \sum_{\beta \in A_n} W(\{\beta\},\{\a_1,\cdots,\a_r \}) 
    \frac{\Psi'_n\big(\{\a_1,\cdots,\a_\ell,\beta\}\big) }{\Psi'_n\big(\{\a_1,\cdots,\a_r\}\big) }.
    \label{EqTech3}
\end{equation}
As in previous sections, we split this sum into different parts. Write $\beta=(i,l)$.
Constants in $\mathcal{O}$ symbols below can be chosen independent of $\a_1,\cdots,\a_r$, but depend on $r$.
\begin{itemize}
    \item If fulfills $i \ne i_1,\cdots,i_r$ and $l \ne l_1,\cdots,l_r$,
        then $W(\{\beta\},\{\a_1,\cdots,\a_r \})=1/n$ and 
        \[\frac{\Psi'_n\big(\{\a_1,\cdots,\a_\ell,\beta\}\big) }{\Psi'_n\big(\{\a_1,\cdots,\a_r\}\big) } = \frac{M}{n}.\]
        Since there are $\O(n^2)$ such terms, the total contribution of these terms is $\O(1)$.
    \item If $i \in \{i_1,\cdots,i_r\}$, but $(i,s) \notin \{(i_1,l_1),\cdots,(i_r,l_r)\}$,
        then $W(\{\beta\},\{\a_1,\cdots,\a_r \})=1$ and 
        \[\frac{\Psi'_n\big(\{\a_1,\cdots,\a_\ell,\beta\}\big) }{\Psi'_n\big(\{\a_1,\cdots,\a_r\}\big) } = \frac{M}{n}.\]
        There are $\O(n)$ such terms which gives a total contribution of $\O(1)$.
    \item The total contribution of terms with $s \in \{s_1,\cdots,s_r\}$ but $(i,s) \notin \{(i_1,l_1),\cdots,(i_r,l_r)\}$
        is $\O(1)$ from the same argument.
    \item Finally, if $(i,s) \in \{(i_1,l_1),\cdots,(i_r,l_r)\}$,
        we have $W(\{\beta\},\{\a_1,\cdots,\a_r \})=1$ and 
        \[\frac{\Psi'_n\big(\{\a_1,\cdots,\a_\ell,\beta\}\big) }{\Psi'_n\big(\{\a_1,\cdots,\a_r\}\big) } = M.\]
        But the number of such terms is bounded by $r$, so
        that their total contribution is also $\O(1)$.
\end{itemize}
Finally, we get that, for any $\a_1,\cdots,\a_r$, the quantity \eqref{EqTech3}
is bounded by a constant $D_r$, uniformly on $\a_1,\cdots,\a_r$.
Thus, for each $r \ge 1$,
the sequence $(T_{r,n})_{n \ge 1}$ is bounded.

Using \cref{LemBoundJointCumulants}, we can now write: for $r>2$ and $t_1,\cdots,t_r$ in $[0,1]$,
\begin{multline*}
    \left| \ka_r\left( \widetilde{X_n(t_1)},\cdots,\widetilde{X_n(t_r)} \right) \right|
    = n^{-r/2} \left| \ka_r\left( X_n(t_1),\cdots,X_n(t_r) \right) \right|  \\
    \le n^{-r/2} C_r r! \, R_n \, T_{1,n} \cdots T_{r-1,n} \le C_r r!\, D_1 \cdots D_{r-1} \, M \, n^{1-r/2}.
\end{multline*}
The right hand side tends to $0$ so that $\left| \ka_r\left( \widetilde{X_n(t_1)},\cdots,\widetilde{X_n(t_r)} \right) \right|$
tends to $0$.
This proves the convergence of the finite-dimensional laws towards Gaussian vectors.
\medskip

It remains now to prove that the sequence of random functions $\widetilde{X_n}$ is tight in $C[0,1]$.
This will prove the existence of the continuous Gaussian process $Z$,
and the convergence of $X_n$ towards $Z$ as well.

To do this, we use a moment criterion that can be found in a book of Kallenberg
\cite[Corollary 16.9 for $d=1$]{KallenbergBookProba}:
a sufficient condition for $\widetilde{X_n}$ to be tight is 
that $\widetilde{X_n}(0)$ is tight and that,
for some positive constants $a$, $b$ and $\la$,
\begin{equation}
    \esper \big[|\widetilde{X_n}(s)- \widetilde{X_n}(t)|^a \big] \le \la \, |s-t|^{1+b} \ \text{ for all }s,t \in [0,1], n \ge 1.
    \label{EqTightnessCriterion1D}
\end{equation}
In our case, $X_n(0)$ is identically equal to $0$ so that only the inequality \eqref{EqTightnessCriterion1D}
needs to be checked.
Moreover, since $\widetilde{X_n}$ is affine in each interval $[j/n, (j+1)/n]$,
it is in fact sufficient to prove this inequality when $nt$ and $ns$ are integers;
see \cref{AppOnlyOnLattice} (this reduction needs $a \ge 1+b$, which is the case in what follows).

Let $n \ge 1$ be an integer and $s$ and $t$ in $[0,1]$ such that $ns$ and $nt$ are integers.
Assume $t<s$.
We consider the case $a=4$, that is the fourth moment of $\widetilde{X_n}(s)- \widetilde{X_n}(t)$. 
Since $\widetilde{X_n}(s)- \widetilde{X_n}(t)$ is centered, 
from the moment cumulant formula \eqref{EqCumulant2Moment}, we get
\[ \esper \big[ (\widetilde{X_n}(s)- \widetilde{X_n}(t))^4 \big]
= \kappa_4( \widetilde{X_n}(s)- \widetilde{X_n}(t) ) +3 \kappa_2 ( \widetilde{X_n}(s)- \widetilde{X_n}(t) )^2.\]
But
\[n^{1/2} \big(\widetilde{X_n}(s)- \widetilde{X_n}(t) \big) = \sum_{i=nt+1}^{ns} \sum_{l=1}^n a^{(n)}(i,l) Y_{i,l} \]
and its cumulants can be bounded by \cref{LemBorneCumulant}.
Note that we consider here the restriction of the dependency graph
above to the family $\{a^{(n)}(i,l) Y_{i,l}, \ {\bf nt<i \le ns}\text{ and } 1\le l \le n\}$.
Then we have
\[R_n(s,t) := \sum_{i=nt+1}^{ns} \sum_{l=1}^n \Psi'_n(\{(i,l)\}) = M \, n \, (s-t).\]
On the other hand the parameter $T_{\ell,n}(s,t)$ associated to this restricted graph
is bounded by the same bound as in the non-restricted case above: $T_{\ell,n}(s,t)=\O(1)$.
Therefore, from \cref{LemBorneCumulant}, we have
\[ |n \ka_2\big(\widetilde{X_n}(s)- \widetilde{X_n}(t) \big) | \le D_2 \, n \, (s-t),\quad
|n^2 \ka_4\big(\widetilde{X_n}(s)- \widetilde{X_n}(t) \big) | \le D_4 \, n \, (s-t),\]
for some constants $D_2$ and $D_4$.
Putting all together,
\[\esper \big[ (\widetilde{X_n}(s)- \widetilde{X_n}(t))^4 \big] \le
D_4 \, n^{-1} \, (s-t) +3D_2^2 (s-t)^2 \le (D_4+3D_2)\, (s-t)^2,\]
where the last equality comes from the fact that $s-t \ge n^{-1}$ 
since $ns$ and $nt$ are distinct integers.
Thus \eqref{EqTightnessCriterion1D} is proved for $a=4$, $b=1$ and $\la=D_4+3D_2$,
which ends the proof of the theorem.
\end{proof}
To illustrate this theorem, we use the same example as Barbour and Janson 
\cite[Theorem 5.1]{BarbourJansonFunctionalCLT}.
Let $a_0^{(n)}(i,l)=[l \ge i]$, that is $1$ if $l \ge i$ and $0$ otherwise.
Then $X_n(t)$ is the number of {\em weak exceedances} of $\Pi$ of index at most $nt$. 

After centering, we have $a^{(n)}(i,l)=[l \ge i]- (n-i+1)/n$, which is obviously uniformly bounded.
As explained in \cite{BarbourJansonFunctionalCLT}, if $0 \le t \le u \le 1$, then one has
\[f(t) = \lim_{n \to \infty} f_n(t) = \tfrac{1}{2}\, t^2-\tfrac{1}{2}\, t^3; \ g(t,u)=\lim_{n \to \infty} g_n(t,u)=\tfrac{1}{2}\, t^2 u -\tfrac{1}{6}\, t^3-\tfrac{1}{4}\, t^2 u^2.\]
All our hypotheses are fulfilled and we obtain that there exists a continuous Gaussian process $Z$
with covariance function $\si(Z(t),Z(u))=\tfrac{1}{2}t^2(1-u+\tfrac{1}{2}u^2)-\tfrac{1}{6}t^3$
and that, in distribution in $C[0,1]$,
\[\frac{X_n - \esper X_n}{\sqrt{n}} \to Z.\]

\begin{remark}
    The hypotheses given here are stronger than 
    the ones of Barbour and Janson \cite{BarbourJansonFunctionalCLT},
    who use a bound on the Lyapounov ratio,
    instead of our uniformly bounded assumption.
    However, as seen above, the example of exceedances, which motivated their work,
    also fits in our framework.
    Note also that Barbour and Janson also give a bound on the speed of convergence,
    which we cannot achieve.

    Another difference between their theorem and ours is that
    they consider convergence in Skorohod space $D[0,1]$,
    while we work in $C[0,1]$, but since the limit is continuous,
    this is just a matter of taste.
\end{remark}

\subsection{A functional central limit theorem for doubly indexed permutation statistics}
\label{SubsecDoubleIndexPermStat}
An advantage of the method of the previous section is
that it can be easily adapted to more involved permutation statistics,
such as {\em doubly indexed permutation statistics (DIPS)}.
By definition a DIPS is a statistics of the following form:
let $\zeta_0^{(n)}(i,j,k,l)_{i,j,k,l \in [n]}$ be a sequence of multi-indexed
real numbers, then,
for a permutation $\pi$ of size $n$, we set
\[ X_n(\pi)= \sum_{1 \le i,j \le n} \zeta_0^{(n)}(i,j,\pi(i),\pi(j)).\]
A central limit theorem for DIPS with control on the speed of convergence
is given in \cite{CLT_DIPS}.
In this section, we provide a \,\!\! {\em functional} CLT for this class of statistics.
\medskip

To this end let us associate with a DIPS and a permutation $\pi$ a continuous function on $[0,1]^2$ as follows.
If $nt_1$ and $nt_2$ are integers, then
\[X^\pi_n(t_1,t_2) = \sum_{i=1}^{nt_1} \sum_{j=1}^{nt_2} \zeta_0^{(n)}(i,j,\pi(i),\pi(j)).\]
The function $X^\pi_n$ is then extended to $[0,1]^2$ by requiring
that, for any pair $(i,j)$ with $0 \le i,j \le n-1$,
the function $X^\pi$ is affine on the square $[i/n;(i+1)/n] \times [j/n;(j+1)/n]$.

We now consider a uniform random permutation $\Pi$ of size $n$ and 
the associated random function $X_n:=X_n^\Pi$.
We perform the following centering:
\[\zeta^{(n)}(i,j,k,l)=\begin{cases}
    \zeta_0^{(n)}(i,j,k,l) - \frac{1}{n(n-1)} \sum_{k' \neq l'} \zeta_0^{(n)}(i,j,k',l') & \text{if }i \ne j;\\
    \zeta_0^{(n)}(i,j,k,l) - \frac{1}{n} \sum_{k'} \zeta_0^{(n)}(i,j,k',k') & \text{if }i = j.
\end{cases}.\]
With this definition, if $nt_1$ and $nt_2$ are integers, we have
\[X_n(t_1,t_2) -\esper X_n(t_1,t_2) = \sum_{i=1}^{nt_1} \sum_{j=1}^{nt_2} \zeta^{(n)}(i,j,\pi(i),\pi(j)).\]
We assume that:
\begin{itemize}
    \item the real numbers $\zeta^{(n)}(i,j,k,l)$ ($n \ge 1, i,j,k,l \le n$)
        are uniformly bounded by a constant $M$;
    \item the rescaled covariance
        $n^{-3} \Cov\big(X_n(t_1,t_2), X_n(u_1,u_2) \big)$
        has a pointwise limit $\sigma(t_1,t_2;u_1,u_2)$.
        (It may be possible to give sufficient conditions such as \cref{EqDefF,EqDefG}
        for this convergence, but this would be technical and not enlightening.)
\end{itemize}
We consider here the space $C[0,1]^2$ of real-valued continuous functions on $[0,1]^2$,
with the topology of uniform convergence. We then have the following theorem.
\begin{theorem}
    We use the notation and assumptions above.
    Then there exists a zero-mean continuous Gaussian process $Z$ on $[0,1]^2$
    with covariance function given by
    \[\Cov(Z(t_1,t_2),Z(u_1,u_2)) = \si(t_1,t_2;u_1,u_2)\]
    and, in distribution in $C[0,1]^2$, we have
    \[ \widetilde{X_n}(t_1,t_2) := \frac{X_n(t_1,t_2) - \esper X_n(t_1,t_2)}{n^{3/2}} \to Z.\]
    \label{ThmFunctionalCLTDIPS}
\end{theorem}
\begin{proof}
    The structure of the proof is the same as for simply-indexed permutation statistics.
    We first prove the convergence of finite-dimensional laws by controlling joint cumulants.
    Both sides are centered and we have assumed the convergence of covariances,
    so that we can focus on joint cumulants of order at least 3.

    Note that, if $nt_1$ and $nt_2$ are integers, we can rewrite $\widetilde{X_n}(t_1,t_2)$ as
    \[\widetilde{X_n}(t_1,t_2) = n^{-3/2} \sum_{i=1}^{nt_1} \sum_{j=1}^{nt_2} \sum_{k=1}^{n} \sum_{l=1}^{n}
    \zeta^{(n)}(i,j,k,l) Y_{i,k} Y_{j,l}.\]
    Recall from \cref{PropWDGInRandomPermForPairs} that the family
$\{Y_{i,k} Y_{j,l} , (i,k),\, (j,l) \in A_n\}$ admits $\WDep^2$ as a 
$(\bm{\Psi},\bD)$ weighted dependency graph.
Since $\zeta^{(n)}(i,j,k,l)$ is uniformly bounded by $M$,
the family \[\{\zeta^{(n)}(i,j,k,l) Y_{i,k} Y_{j,l} , (i,k),\, (j,l) \in A_n\}\]
has the same dependency graph, replacing $\bm{\Psi}$ by $\bm{\Psi}'(B):= M^{|B|} \bm{\Psi}(B)$.
For this dependency graph, we have $R_n= M n^2$.
A case analysis similar to the one above shows that $T_{r,n}=\O(n)$
(with a constant depending on $r$).
We sketch here briefly the argument.
Recall that we want to bound, for fixed $\a'_1,\cdots,\a'_r$, the sum
\begin{equation}
    \sum_{\beta' \in A'_n} W(\{\beta'\},\{\a'_1,\cdots,\a'_r \}) 
    \frac{\bm{\Psi}'_n\big(\{\a'_1,\cdots,\a'_\ell,\beta'\}\big) }{\bm{\Psi}'_n\big(\{\a'_1,\cdots,\a'_r\}\big) }.
    \label{EqTech4}
\end{equation}
\begin{itemize}
    \item If $\beta'$ in $A'_n$ does not share any element with $\a'_1,\cdots,\a'_r$,
        then the quotient of $\bm{\Psi}'$ is $M/n^2$ and the $W$ factor is equal to $1/n$.
        Since there are fewer than $|A'_n| \asymp n^4$ such terms,
        their total contribution is $\O(n)$.
    \item If $\beta'$ has an element, but no pair in common with one of the $\a'_i$,
        then the quotient of $\bm{\Psi}'$ is also $M/n^2$, while the $W$ factor is $1$.
        But there are $\O(n^3)$ such terms, so that the total contribution of such terms is also $\O(n)$.
    \item If $\beta'$ has exactly pair in common with one of the $\a'_i$,
        then the quotient of $\bm{\Psi}'$ is also $M/n$ and the $W$ factor is also $1$.
        There are $\O(n^2)$ such terms, so that the total contribution of such terms is also $\O(n)$.
    \item Finally if both pairs in $\beta'$ already appear in the $\a'_i$,
        then the quotient of $\bm{\Psi}'$ is $M$ and the $W$ factor is also $1$.
        But this implies that the pairs of $\beta'$ are chosen within a finite family,
        so that there is only a constant number of such terms and their total contribution is $\O(1)$.
\end{itemize}
Finally, as claimed above, for $r \ge 1$, there exists a constant $D_r$ such that
$T_{r,n} \le D_r \, n$.

Let $((t_1^1,t_2^1),\cdots,(t_1^r,t_2^r))$ be an $r$-uple of points in $[0,1]^2$ ($r \ge 3$).
From \cref{LemBoundJointCumulants} and the discussion above, we get 
\[
    \left| \ka_r\left( \widetilde{X_n(t_1^1,t_2^1)},\cdots,\widetilde{X_n(t_1^r,t_2^r)} \right) \right|
    \le n^{-3r/2} C_r r! \, R_n \, T_{1,n} \cdots T_{r-1,n} \le C_r r!\, D_1 \cdots D_{r-1} \, M \, n^{1-r/2}.
    \]
The right hand side tends to $0$ so that all joint cumulants 
of the family $(X_n(t_1,t_2))_{(t_1,t_2) \in [0,1]^2}$ of order at least $3$
tend to $0$.
This proves the convergence of the finite-dimensional laws towards Gaussian vectors.

We now prove the tightness of the random functions $(\widetilde{X_n})_{n \ge 1}$
in the space $C[0,1]^2$.
We again use the moment criterion \cite[Corollary 16.9]{KallenbergBookProba}, but this time for $d=2$.
Since $\widetilde{X_n}(0,0)$ is tight (it is identically equal to $0$, for all $n$),
we should prove that there exist
positive constants $a$, $b$ and $\la$,
\begin{equation}
    \esper \big[|\widetilde{X_n}(s_1,s_2)- \widetilde{X_n}(t_1,t_2)|^a \big] \le \la \, 
    (|s_1-t_1|+|s_2-t_2|)^{2+b}
    \label{EqTightnessCriterion2D}
\end{equation}
for all $(s_1,s_2),(t_1,t_2)$ in $[0,1]^2$ and $n \ge 1$.
As in dimension $1$, since $\widetilde{X_n}$ is 
affine on each square $[i/n;(i+1)/n] \times [j/n;(j+1)/n]$ ($0 \le i,j \le n-1$),
it is enough to prove \eqref{EqTightnessCriterion2D}
when $n s_1$, $n s_2$, $n t_1$ and $n t_2$ are integers;
see \cref{AppOnlyOnLattice}.

Let us first give bounds depending on $(s_1,s_2),(t_1,t_2)$ for cumulants of the difference
$\widetilde{X_n}(s_1,s_2)- \widetilde{X_n}(t_1,t_2)$.
If $t_1<s_1$ and $t_2 <s_2$, then
\[\widetilde{X_n}(s_1,s_2)- \widetilde{X_n}(t_1,t_2)
= n^{-3/2} \sum_{i,j} \sum_{k=1}^n \sum_{l=1}^n \zeta^{(n)}(i,j,k,l) Y_{i,k} Y_{j,l},\]
where the first sum runs over pairs $(i,j)$ such that $i \le n s_1$, $j \le s_2$ and
either $nt_1 < i \le n s_1$ or $nt_2 < j \le n s_2$.
There are fewer than $n^2 (s_1-t_1+s_2-t_2)$ such pairs $(i,j)$, so that, 
by the same argument as in the one-dimensional case (restricting the dependency graph), we have
\begin{equation}
    \big|\ka_r\big(\widetilde{X_n}(s_1,s_2)- \widetilde{X_n}(t_1,t_2) \big) \big|
\le D_r (|s_1-t_1|+|s_2-t_2|) \, n^{1-r/2},
\label{EqBoundCumulantDiff2D}
\end{equation}
for some constant $D_r$ that depends only on $r$.
The same bound obviously holds without the assumption {\em $t_1<s_1$ and $t_2 <s_2$}.

We now have to consider the moment of order $6$
of $\Delta X:=\widetilde{X_n}(s_1,s_2)- \widetilde{X_n}(t_1,t_2)$.
In terms of cumulants it writes as:
\[\esper(\Delta X^6)=\ka_6(\Delta X) + 15 \ka_4(\Delta X) \ka_2 (\Delta X) + 10 \ka_3(\Delta X)^2
+ 15 \ka_2(\Delta X)^3.\]
Set $\delta= |s_1-t_1|+|s_2-t_2|$. From \eqref{EqBoundCumulantDiff2D}, we have
\[\esper \big[\Delta X^6 \big]
\le D_6 n^{-2} \delta + (15 D_4 D_2 + 10 D_3^2) n^{-1} \delta^2 +15 D_2^3 \delta^3
\le D \delta^3,\]
for some constant $D$.
For the last inequality, note that since $n s_1$, $n s_2$, $n t_1$ and $n t_2$ are integers,
we have $n^{-1} \le \delta=|s_1-t_1| + |s_2-t_2|$ (we can assume that either $s_1\ne t_1$ or $s_2 \ne t_2$,
otherwise \eqref{EqTightnessCriterion2D} is trivial).
This ends the proof of \eqref{EqTightnessCriterion2D} (for $a=6$ and $b=1$)
and hence of the theorem.
\end{proof}
As an example we consider {\em positive alignments} in random permutations.
A positive alignment in a permutation $\pi$ is a pair $(i,j)$ such that $j<i \le \si(i) < \si(j)$.
This statistics mixes somehow the classical notions of inversions and exceedances:
it is studied together with many similar statistics in \cite{CorteelCrossingsAlignmentsPerm}.
Let us set $\zeta_0^{(n)}(i,j,k,l)=[j<i \le k<l]$ (i.e. 1 if $j<i \le k<l$ and $0$ otherwise)
and define the associated random function $X_n^\Pi$ in $C[0,1]^2$ as above.
In particular $X_n^\Pi(1,1)$ is the number of positive alignments in the uniform random permutation $\Pi$.

It is clear that $\zeta_0^{(n)}(i,j,k,l)$ and hence $\zeta_0^{(n)}(i,j,k,l)$ is uniformly bounded.
Besides, an easy adaptation of \cref{LemVarPoly} shows that $\Cov\big(X_n(t_1,t_2), X_n(u_1,u_2) \big)$
is a polynomial in $n, t_1, t_2, u_1, u_2$.
Moreover, from the same arguments as in the proof above to bound 
general joint cumulants,
we know that, for fixed $t_1, t_2, u_1 u_2$, it behaves as $\O(n^3)$.
Thus, for any $t_1, t_2, u_1 u_2$ in $[0,1]$, the rescaled covariance
$n^{-3} \Cov\big(X_n(t_1,t_2), X_n(u_1,u_2) \big)$ has indeed a limit.

Thus, our theorem applies and $\widetilde{X_n}$ converges in probability in $C[0,1]^2$
towards a zero-mean continuous Gaussian process in $[0,1]^2$.
It is possible to compute the covariances of the limiting process,
but it would be a lengthly computation.

\begin{remark}
    The extension of the above result to $k$-indexed permutation statistics
    for fixed $k$ is straightforward (with a convergence in distribution in $C[0,1]^k$).
    However, it becomes rather difficult to do any explicit computation.
\end{remark}

\section{Symmetric simple exclusion process}

\subsection{Background on the model}
The symmetric simple exclusion process with open boundaries ({\em SSEP} for short) is a
continuous-time Markov chain defined as follows:
we consider particles on a discrete line with $N$ sites.
More formally, the space state is $\{0,1\}^{N}$:
a state of the {\em SSEP} is encoded as a word in $0$ and $1$ of length $N$,
where the entries with value $1$ correspond to the positions of the occupied sites.
The system evolves as follows:
\begin{itemize}
    \item each particle has an exponential clock with rate $1$.
      When it rings the particle jumps to the right if 
      it is not in the right-most site and if the site at its right is empty.
      Otherwise, the jump is suppressed.
    \item Similarly, each particle has another exponential clock with rate $1$
      and attempts to jump to its left when it rings (with similar rules as above).
    \item if the left-most (resp. right-most) site is empty,
      an exponential clock with rate $\alpha$ (resp. $\delta$) is associated with it.
      When it rings, a particle is added to the left-most (resp. right-most) site.
    \item if the left-most (resp. right-most) site is full,
      an exponential clock with rate $\gamma$ (resp. $\beta$) is associated with it.
      When it rings, the particle in the left-most (resp. right-most) site is removed.
\end{itemize}
All the above mentioned exponential clocks are independent.
The conditions for particles to jump or be added in the extremities
ensure that no two particles are in the same site at the same moment
(which explains the terminology {\em simple exclusion} in the name of the model).
The transition are schematically represented on \cref{StateSSEP}.

\begin{figure}[t]
  \begin{center}
    \includegraphics{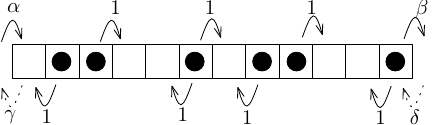}
  \end{center}
  \caption{A state of the SSEP with the possible transitions and the rates of the corresponding clocks.
  Dotted transtitions do not apply to the represented state.}
  \label{StateSSEP}
\end{figure}

As common in the field, we define $\rho_a=\tfrac{\a}{\a+\gamma}$ and $\rho_b=\tfrac{\beta}{\beta+\delta}$.
We are interested here in a random state $\tau$ in  $\{0,1\}^{N}$,
distributed according to the steady state of the SSEP,
that is the invariant measure of this Markov process.
The correlation functions of the particles, that are
the joint moments of the coordinates $(\tau_i)_{1 \le i \le N}$,
can be described using the so-called {\em matrix Ansatz}
of Derrida, Ewans, Hakim and Pasquier \cite{ASEPMatrixAnsatz}.
Based on this matrix Ansatz, Derrida, Lebowitz and Speer have found
an inductive formula to compute joint cumulants of the family $(\tau_i)_{1 \ge i \ge N}$
(called {\em truncated correlation functions} in this context).

To state it, we first need to introduce the {\em discrete difference operator} $\Delta$.
If $f$ in a function on positive integers, we set $\Delta f(N) = f(N)-f(N-1)$.
Note that $\Delta f$ is not defined for $N=1$,
but this is irrelevant as we shall make $N$ tends to infinity, while we apply
$\Delta$ a fixed number of times.

Fix a positive integer $r$ and a some integers $1 \le i_1 < \cdots < i_{r+1} \le N$.
As the formula involves SSEP with different number $N$ of sites,
we make it explicit in the notation
and denote $\ka^N_{r}(\tau_{i_1},\ldots,\tau_{i_r})$ the joint cumulants of 
$\tau_{i_1}$,\ldots, $\tau_{i_r}$.
Derrida, Lebowitz and Speer \cite[Eq (A.11)]{DerridaLongRangeCorrelation}  have proved
that
\begin{equation}
  \ka^N_{r+1}(\tau_{i_1},\ldots,\tau_{i_r},\tau_{i_{r+1}})
= (\esper(\tau_{i_{r+1}})-\rho_b) \sum_{\pi \in \PPP([r])}
\prod_{B \in \pi} \Delta \ka^N_{|B|}(\tau_{i_t}; t \in B).
\label{eq:InductionCumulantSSEP}
\end{equation}
Expectations can be easily computed (see, {\em e.g.} \cite[Eq. (42)]{derrida2007SurveyASEP}):
\begin{equation}
  \esper(\tau_i)= \frac{\rho_a\big(N+1/(\beta+\delta)-i\big) + \rho_b\big(i-1+1/(\a+\gamma)\big)}
{N+1/(\a+\gamma) +1/(\beta+\delta)-1 }.
\label{eq:EsperSSEP}
\end{equation}
\cref{eq:InductionCumulantSSEP,eq:EsperSSEP} determine the joint cumulants 
of distinct variables in the family $(\tau_i)_{1 \le i \le N}$.
We will use this to find a weighted dependency graph for this family in the next section.

\subsection{A weighted dependency graph in SSEP}
We start by a lemma, bounding {\em repetition-free} joint cumulants of the family $(\tau_i)_{1 \le i \le N}$.
\begin{lemma}
  Let $r \ge 1$. Then there exists a constant $D_r$ such that
  for each $N \ge r$ and $(i_1,\cdots,i_r)$ with $1 \le i_1 < \cdots < i_r \le N$,
  we have
  \[\big| \ka^N_r(\tau_{i_1},\ldots,\tau_{i_r}) \big|
  \le D_r N^{-r+1}. \vspace{-8mm}\]
  \label{LemBorneCumulantsASEP}
\end{lemma}
\begin{proof}
  We will in fact prove a stronger statement: 
  \begin{quote}
  The quantity $\ka^N_r(\tau_{i_1},\ldots,\tau_{i_r})$ is a polynomial in $i_1,\cdots,i_r$
  with coefficients that are rational functions in $N$.
    Moreover, its total degree in $N,i_1,\cdots,i_r$ is at most $-r+1$.
  \end{quote}
  To simplify the discussion below, we call such a function a {\em nice function}
  of degree at most $-r+1$.
  It is clear that, if $f(N;i_1,\cdots,i_r)$ is a nice function of degree at most $d$, then
  \[\max_{i_1,\ldots,i_r \in [N]} f(N;i_1,\cdots,i_r) = \mathcal{O}(N^d).\]
  Therefore proving the above claim proves the lemma.
  
  We prove this statement by induction on $r$.
  For $r = 1$, it follows immediately from the explicit formula \eqref{eq:EsperSSEP}.
  Take $r \ge 1$ and suppose that our statement holds for any $r' \le r$.
  We consider the quantity $\ka^N_{r+1}(\tau_{i_1},\ldots,\tau_{i_r},\tau_{i_{r+1}})$
  and its expression given in \cref{eq:InductionCumulantSSEP}.
  Fix a set partition $\pi$ in $\PPP([r])$.
  \begin{itemize}
    \item 
  By induction hypothesis, for each block $B$ of $\pi$,
   $\ka^N_{|B|}(\tau_{i_t}; t \in B)$ is a nice function of degree at most $-|B|+1$.
 \item 
   Applying the operator $\Delta$ turns it into a nice function of degree at most $-|B|$.
 \item 
   Multiplying these nice functions for different blocks $B$ of $\pi$ 
   gives a nice function of degree at most $-\sum_{B \in \pi} |B|=-r$.
  \end{itemize}
   The sum of these nice functions (over set partitions $\pi$ in $\PPP([r])$)
   is also a nice function of degree at most $-r$.
   We then multiply by $\esper(\tau_{i_{r+1}})-\rho_b$ which, as can be seen on \cref{eq:EsperSSEP}, 
   is a nice function of degree $0$ and we still have a nice function of degree at most $-r$.
   Therefore $\ka^N_{r+1}(\tau_{i_1},\ldots,\tau_{i_r},\tau_{i_{r+1}})$ is a nice function
   of degree  at most $-r$, which ends the proof of the lemma.
\end{proof}

We are ready to present a weighted dependency graph associated with SSEP.
\begin{proposition}
  Let $N \ge 1$ and $\tau=(\tau_1,\cdots,\tau_N)$ be a random $\{0,1\}$ vector
  distributed according to the steady state of SSEP on $N$ sites.
  Let $A_N=[N]$ and consider the family of variables $\{\tau_i, i \in A_N\}$.
  We consider the complete graph $\WDep$ with weight $1/N$ on each edge
  and the function $\Psi_N$ on multiset of elements of $A_N$ that is identically equal to $1$.
  Then $\WDep$ is a $(\Psi_N,\bC)$ weighted dependency graph for the family 
  $\{\tau_i, i \in N\}$,
  for some sequence $\bC=(C_r)_{r \ge 1}$ that does not depend on $N$.
  \label{PropWDGInSSEP}
\end{proposition}
\begin{proof}
  We note the three following fact: (1) $\Psi_n$ is trivially super-multiplicative, 
  (2) the $\tau_i$ are Bernoulli variables and (3) $\WDep$ has no edges of weight $1$.
  From \cref{RmkDiscreteBernoulli} (which uses \cref{PropAlternate}),
  it is enough to prove bounds on cumulants of sets of distinct variables 
  (instead of cumulants of all multisets of variables).
  Namely, we should prove that,
  for any $r \ge 1$ and any {\em distinct} $i_1,\cdots,i_r$ in $[N]$, one has
  \[\big| \ka^N_r(\tau_{i_1},\ldots,\tau_{i_r}) \big|
  \le D_r N^{-r+1},\]
  for a constant $D_r$ that does not depend on $N$.
  But this is exactly \cref{LemBorneCumulantsASEP}.
\end{proof}

\begin{remark}
  In \cite{DerridaLongRangeCorrelation}, Derrida, Lebowitz and Speer have proved that,
  for any $x_1,\cdots,x_r$ in $[0,1]$ the quantity
  $N^{r-1} \ka^N_r\big(\tau_{\lfloor N\, x_1\rfloor},\ldots,\tau_{\lfloor N\, x_r \rfloor}\big)$
  has a limit when $N$ tends to infinity.
  This of course implies that the joint cumulant is $\O(N^{-r+1})$.
  However the constant in the $\O$ symbol could a priori depend on $x_1,\cdots,x_r$,
  while we need a bound which is uniform in $i_1,\cdots, i_r$.
  This explains why we need \cref{LemBorneCumulantsASEP}
  and we can not use directly the result of Derrida, Lebowitz and Speer.
  Nevertheless, the key identity in the proof is the induction formula \eqref{eq:InductionCumulantSSEP},
  due to these authors.
   \end{remark}

\subsection{A functional central limit for the number of particles}
\label{SubsecSSEPParticles}
Let $N \ge 1$ and $t$ in $[0,1]$.
We consider a random state $\tau$ in $\{0,1\}^N$,
distributed according to the steady state of the SSEP on $N$ sites.
If $Nt$ is an integer, we define $X_N(t)$ as the number of particles
in the first $Nt$ cells of $\tau$. Formally, this means $X_N(t)= \sum_{i=1}^{Nt} \tau_i$.
We then extend $X_N$ to a continuous function on $[0,1]$,
by requiring that it is affine on each segment $[i/N;(i+1)/N]$.

This function measures the repartition of the particles in $\tau$.
Informally, it is the integral of the density of particles, 
often considered in the physics literature; see, {\em e.g.}, \cite[Section 3]{derrida2007SurveyASEP}.

Since there are explicit formulas for the expectations and covariances of the $\tau_i$
\cite[Eq. (2.3) and (2.4)]{DerridaLongRangeCorrelation}, the expectations and covariances
of $(X_N(t))_{t \in [0,1]}$ are easy to evaluate asymptotically:
\begin{align}
  \lim_{N \to \infty} N^{-1} \esper(X_N(t)) &= \rho_a (1-t) + \rho_b \, t; \\
  \label{eq:Cov_SEP} \lim_{N \to \infty} N^{-1} \Cov(X_N(t),X_N(u)) &= \int_0^{t \wedge u} \big(\rho_a(1-x)+\rho_b \, x \big)
  \, \big(1-\rho_a(1-x)-\rho_b \, x \big) dx \\
 \nonumber & \qquad \qquad - \int_0^t \int_0^u x\wedge y \, (1- x \vee y) \, (\rho_a-\rho_b)^2 dx dy.
\end{align}
In the last formula, $x\wedge y:=\min(x,y)$ and $x \vee y:=\max(x,y)$.
We denote $\si(u,v)$ the right-hand side of \cref{eq:Cov_SEP}.
\begin{theorem}
  We use the notation above. 
  There exists a zero-mean continuous Gaussian process $Z$ on $[0,1]$ 
  with covariance function given by
  \[\Cov(Z(t),Z(u))= \si(u,t) \]
  and, in distribution in $C[0,1]$, we have, when $N$ tends to infinity,
  \[\widetilde{X_N}(t):=\frac{X_N(t) - \esper X_N(t)}{\sqrt{N}} \to Z.\]
  \label{ThmFCLTParticles}
\end{theorem}
\begin{proof}
  As usual, we start by proving the convergence of the finite-dimensional laws.
  To do that, we prove the convergence of joint cumulants.
  Expectations clearly converge as both sides are centered.
  Covariances also converge, by definition of $\si(u,t)$.

  Let us consider now higher order cumulants.
  We recall that the family
  $\{\tau_i, i \in N\}$ admits a weighted dependency graph $\WDep$;
  see \cref{PropWDGInSSEP}. Call $R_N$ and $Q_N$ the associated parameters,
  as in \cref{SubsecNormalityCriterion}.
  From \cref{rmk:RAndTPsiOne}, since $\Psi_N$ is the constant function equal to $1$,
  $R_N$ is simply the number of vertices of $\WDep$, which is $N$.
  Moreover, $T_\ell \le \ell \Delta$, where $\Delta-1$ is the maximal weighted degree in the graph,
  which is here smaller than $1$ (i.e. $\Delta <2$).
  From \cref{LemBoundJointCumulants}, we get, that for any $r \ge 3$ and $t_1,\cdots,t_r$ in $[0,1]$
  (such that $Nt_1,\cdots,Nt_r$ are integers),
  \[\bigg|\ka_r\big(\widetilde{X_N}(t_1),\cdots,\widetilde{X_N}(t_r) \big)\bigg|
  = N^{-r/2} \bigg| \ka_r \big( \textstyle \sum_{i=1}^{Nt_1} \tau_i, \cdots, 
  \textstyle \sum_{i=1}^{Nt_r} \tau_i \big) \bigg|
  \le C_r \, r! \, 2^{r-1} \, (r-1)! \, N^{1-r/2}. \]
  In particular, all joint cumulants of order $3$ or more tend to $0$,
  which ends the proof of multidimensional laws toward Gaussian vectors.

  The proof of tightness is virtually identical to that in the proof of \cref{ThmFunctionalCLTSimplePerm}.
\end{proof}
\begin{remark}
  Thanks to the stability of weighted dependency graph by product
  (the function $\Psi$ here, identically equally to $1$, is super-multiplicative),
  it is possible to obtain functional central limits for more complicated quantities
  that involve products of $\tau_i$.
  For example if we are interested in the number and repartition of 
  {\em particles that can jump to their right},
  we should define 
$X'_N(t)= \sum_{i=1}^{Nt} \tau_i (1-\tau_{i+1})$.
Its joint cumulants are easily bounded, using the weighted dependency graph $\WDep^2$
for the set of products $\{\tau_i \tau_j, (i,j) \in [N]^2\}$.
It suffices then to compute the asymptotics of the covariances $(X'_N(t))_{t \in [0,1]}$,
which should be an elementary but cumbersome computation starting
from the explicit formulas that exist for (truncated) correlation functions
\cite{DerridaLongRangeCorrelation}.
\end{remark}

\begin{remark}
  In the recent years, combinatorial models have been given
  to describe the steady state $\tau$ of SSEP (and more generally of the {\em asymetric}
  simple exclusion process); see \cite{CorteelWilliamsASEPGeneral} and references therein.
  In the particular case where $\alpha=\beta=1$ and $\gamma=\delta=0$,
  this relates particles in $\tau$ to exceedances in permutations; 
  see \cite[Section 5.2]{FerayRandomPermutationsCumulants} for details.
  In this sense, the example at the end of \cref{SubsecSingleIndexPermStat}
  can be seen as a particular case of \cref{ThmFCLTParticles}.
\end{remark}

\section{Markov chains}
\label{SectMarkov}
%
%
We consider here an aperiodic irreducible
Markov chain $(M_k)_{k \ge 0}$ on a finite state space $S$.
We denote by $P$ the transition matrix, namely $P(s,t)$ is the probability
that $M_{k+1}=t$ if $M_k=s$ (for any $k \ge 0$).
Let $\pi_0$ be the initial distribution, that is the law of $M_0$.
We also denote $\pi$ the stationary distribution
(seen as a row vector), characterized by $\pi \, P= \pi$.

For $s \in S$ and $i \ge 0$, define $Y^s_i=1$ if $M_i=s$ and $0$ otherwise.
The joint moment of these variables have simple matrix expressions:
if $E_{s,s}$ denotes the matrix with entries $0$ except a $1$ at coordinates $(s,s)$,
 and $\One$ the column vector with all entries equal to $1$,
then we have
\begin{equation}
  \esper[Y^{s_1}_{i_1} \cdots Y^{s_r}_{i_r}]=
\pi_0 P^{i_1} E_{s_1,s_1} P^{i_2-i_1} E_{s_2,s_2} 
\cdots E_{s_{r-1},s_{r-1}} P^{i_r-i_{r-1}} E_{s_r,s_r} \One. 
\label{Eq:Joint_Moments_Markov}
\end{equation}

From now on, we shall suppose that the initial distribution $\pi_0$ 
is equal to the stationary distribution $\pi$.
We will prove in \cref{Subsec:WDG_Markov} that there is a natural weighted dependency graph
structure on the $(Y^s_i)_{i \ge 1; s \in S}$.
The weight of the edge joining $Y^s_i$ and $Y^t_j$ is $\la_2^{j-i}$,
where $\la_2 \in [0,1)$ is the second biggest modulus of an eigenvalue of the transition matrix $P$.
This encodes the fact that far apart elements of the Markov chains are almost-independent.
In \cref{Subsect:TCL_Markov},
this weighted dependency graph structure is used to prove a central limit
theorem for the number of occurrences of a given subword $u$ in $w_n=(M_0,\cdots,M_n)$,
as announced in the introduction.

\subsection{Bounds for boolean and classical cumulants}
The goal of this section is to bound the joint cumulants
of the variables $(Y^s_i)_{i \ge 1; s \in S}$.
Such bound of cumulants can be found in the monograph of
Saulis and Statulevičius \cite[Chapter 4]{LivreOrange:Cumulants};
nevertheless, to keep this section self-contained,
we present a proof here for the simple case of finite-state Markov chain.

Instead of working directly with classical (joint) cumulants,
we first give bounds for {\em boolean cumulants}.
Corresponding bounds for classical cumulants will then follow easily,
thanks to a formula linking these different types of cumulants
recently established by Arizmendi, Hasebe, Lehner and Vargas in \cite{ArizmendiHasebeLehnerVargas2014}
(see also \cite[Lemma 1.1]{LivreOrange:Cumulants};
in {\em loc. cit.}, boolean cumulants are called centered moments).

Let $Z_1,\cdots,Z_r$ be random variables with finite moments defined on the same probability space.
By definition, their boolean (joint) cumulant is
\begin{equation}
  B_r(Z_1,\cdots,Z_r) = \sum_{l=0}^{r-1} (-1)^{l} \sum_{1\le d_1 < \ldots <d_l \le r-1} 
\esper(Z_1 \cdots Z_{d_1}) \, \esper(Z_{d_1+1} \cdots Z_{d_2}) \, \cdots \, \esper(Z_{d_l+1} \cdots Z_r).
\label{Eq:def_boolean_cumulants}
\end{equation}
While not at first sight, this definition is quite similar to the definition of classical (joint) cumulants,
replacing the lattice of all set partitions by the lattice of interval set partitions;
see \cite[Section 2]{ArizmendiHasebeLehnerVargas2014} for details.
Note however that, unlike classical cumulants, boolean cumulants are not symmetric functionals.

%
We recall that $(M_k)_{k \ge 0}$ is an aperiodic irreducible Markov chain with transition matrix $P$,
such that $M_0$ is distributed according to the stationary distribution $\pi$ of the chain.
Recall also that $Y_i^s$ is the indicator function of the event $M_i=s$.
Finally, $\la_2$ is the biggest modulus of an eigenvalue of $P$, except $1$.
\begin{lemma}
  Let $r>0$. With the above notation, there exists a constant $C_{P,r}$ depending on $P$ and $r$
  with the following property.
  For any integers $i_1 < i_2 <\cdots < i_r$ and states $s_1,\cdots,s_r$, we have
  \[ |B_r(Y^{s_1}_{i_1},\cdots,Y^{s_r}_{i_r}) | \le C_{P,r} \la_2^{i_r-i_1}.\]
  \label{Lem:Bound_Boolean}
\end{lemma}
\begin{proof}
  Fix integers $i_1 < i_2 <\cdots < i_r$ and $s_1,\cdots,s_r$ in $S$.
  To make notation lighter, we write $E(j)=E_{s_j,s_j}$, $\,\ell(j)=i_{j+1}-i_j$ and $Z_j=Y^{s_j}_{i_j}$.
  As in the summation index of \eqref{Eq:def_boolean_cumulants},
  we consider $l \ge 0$ and $1\le d_1 < \ldots <d_l \le r-1$.
  Since the initial distribution $\pi_0$ is the stationary distribution $\pi$,
  one has $\pi_0 P^i=\pi$ and  
  formula \eqref{Eq:Joint_Moments_Markov} for joint moments simplifies a little.
  We have
  \[\esper(Z_{d_j+1} \cdots Z_{d_{j+1}}) =
  \pi \, E(d_j+1) \, P^{\ell(d_j+1)} \, E(d_j+2) \,
  \cdots \, E(d_{j+1}-1) \, P^{\ell(d_{j+1}-1)} \, E(d_{j+1}) \, \One.\]
  Multiplying such expressions, we get
  \begin{multline*}
\esper(Z_1 \cdots Z_{d_1}) \, \esper(Z_{d_1+1} \cdots Z_{d_2}) \, \cdots \, \esper(Z_{d_l+1} \cdots Z_r) \\
=\pi \, E(1)\, Q(1)\, E(2)\, Q(2)\, \cdots\, E(r-1) \, Q(r-1)\, E(r)\, \One,
\end{multline*}
where we set, for $1 \le k \le r-1$,
\[Q(k) = \begin{cases}
  \One\, \pi &\text{ if }k \in \{d_1,\cdots,d_l\}; \\
  P^{\ell(k)} &\text{ otherwise.}
\end{cases}\]
The boolean cumulant $B_r(Z_1,\cdots,Z_r)$ now writes as
\[B_r(Z_1,\cdots,Z_r) = \pi \, E(1) \, (P^{\ell(1)} - \One\, \pi)\, E(2) \, \cdots \, 
E(r-1) \, (P^{\ell(r-1)} - \One\, \pi) \, E(r) \, \One.\]
By Perron-Frobenius theorem, the matrix $P$ has a unique eigenvalue of modulus $1$
and $\One \pi$ is the projector on the corresponding eigenvector; see \cite[p 674]{meyer2000matrix}.
Therefore, for any $\ell$, the matrix $(P^\ell -\One \, \pi)$ has operator norm $\la_2^\ell$.
The result follows immediately.
\end{proof}
We now recall the expression of classical cumulants in terms of boolean cumulants
given in \cite{ArizmendiHasebeLehnerVargas2014}.
Let us first introduce some terminology.
A set partition $\rho$ of $[r]$ is called {\em reducible} if there exists $\ell$ in $\{1,\cdots,r-1\}$
such that $\rho \le \big\{ \{1,\cdots,\ell\}, \{\ell+1,\cdots,r\} \big\}$;
otherwise, it is called {\em irreducible}.
The set of irreducible set partitions of $[r]$ is denoted by $\PPP_{\irr}[r]$.
The following statement is a less precise version of \cite[Theorem 1.4]{ArizmendiHasebeLehnerVargas2014}
(see also \cite[Lemma 1.1]{LivreOrange:Cumulants}).
\begin{lemma}
  Let $r \ge 1$. There exist universal constants $d_\rho$, indexed by {\em irreducible}
  set partitions of $[r]$, with the following property.
  For any random variables $Z_1,\cdots,Z_r$ with finite moments 
  defined on the same probability space, one has
  \begin{equation} \ka_n(Z_1,\cdots,Z_r) = \sum_{\rho \in \PPP_{\irr}[r]} d_\rho \cdot
  \left[ \prod_{C \in \rho} B_{|C|}(Z_j; j \in C) \right]. 
  \label{Eq:From_boolean_to_classical}
\end{equation}
\end{lemma}
Arizmendi, Hasebe, Lehner and Vargas relate $d_\rho$ 
with a specialization of the Tutte polynomial of a specific graph associated with $\rho$,
but we do not need this description of $d_\rho$ here.
For our purpose, the crucial aspect in this boolean-to-classical cumulant formula
is that the sum ranges only over irreducible set partitions.
We can now establish our bound on classical cumulants.

\begin{lemma}
\label{Lem:Bound_Cumulants_Markov}
As above, let $(M_k)_{k \ge 0}$ be an aperiodic irreducible Markov chain with transition matrix $P$,
such that $M_0$ is distributed according to the stationary distribution $\pi$ of the chain.
  Let $r>0$. Then there exists a constant $D_{P,r}$ depending on the transition matrix 
  $P$ and on $r$
  with the following property.
  For any distinct integers $i_1 < i_2 < \cdots < i_r$ and states $s_1,\cdots,s_r$, we have
  \[ |\ka_r(Y^{s_1}_{i_1},\cdots,Y^{s_r}_{i_r}) | \le D_{P,r} \la_2^{i_r-i_1}.\]
\end{lemma}
\begin{proof}
  For any subset $C$ of $[r]$, we know by \cref{Lem:Bound_Boolean} that
  \[ |B_{|C|}(Y^{s_j}_{i_j}; j \in C)| \leq \text{cst} \, \la_2^{i_{\max(C)}-i_{\min(C)}}.\]
  If $\rho$ is an irreducible set partition, one can easily check that
  \[\sum_{C \in \rho} i_{\max(C)}-i_{\min(C)} \ge i_r -i_1.\]
  Therefore each summand in \eqref{Eq:From_boolean_to_classical} is 
  bounded in absolute value by a constant times $\la_2^{i_r-i_1}$,
  which proves the lemma.
\end{proof}

\subsection{A weighted dependency graph for Markov chains}
\label{Subsec:WDG_Markov}
We denote by $\NN_{\ge 0}$ the set of nonnegative integers.
\begin{proposition}
As above, let $(M_k)_{k \ge 0}$ be an aperiodic irreducible Markov chain on a finite state space $S$,
such that $M_0$ is distributed according to the stationary distribution $\pi$ of the chain.
Recall that $Y_i^s$ is the indicator function of the event $M_i=s$.

  We consider the complete graph $\WDep$ on $A:=\NN_{\ge 0} \times S$ with weight $\la_2^{j-i}$
   on the edge $\{(i,s),(j,t)\}$ (for any nonnegative integers $i<j$ and states $s,t$ in $S$).
  Finally, let $\Psi$ be the function on multisets of elements of $A$
  that is identically equal to $1$.

  Then $\WDep$ is a $(\Psi,\bC)$ weighted dependency graph for the family 
  $\{Y_i^s;\, (i,s) \in A\}$ for some sequence $\bC=(C_r)_{r \ge 1}$.
  \label{PropWDGInMarkov}
\end{proposition}
\begin{proof}
Consider a multiset $B=\{(i_1,s_1),\dots, (i_r,s_r)\}$ of elements of $A$
and the induced graph $\WDep[B]$. Assume, without loss of generality that $i_1 < \dots < i_r$.
Then it is easy to observe that the maximum weight of a spanning tree in $\WDep[B]$
is
$\MWST{\WDep[B]} =  \la_2^{i_r-i_1}$.

We use \cref{PropAlternate}.
Vertices $(i,s)$ and $(j,t)$ in $\WDep$ are connected by an edge of weight 1 if and only if $i=j$.
But $(Y_i^s)^2=Y_i^s$ and $Y_i^s\, Y_i^t=0$ if $s \ne t$.
Therefore, it is enough to prove that, for any fixed $r>0$,
 there exists a constant $D_r$ with the following property:
 for any {\em distinct} integers $i_1 < \dots < i_r$ and any states $s_1, \dots, s_r$,
 we have
\[ |\ka_r(Y^{s_1}_{i_1},\cdots,Y^{s_r}_{i_r}) | \le D_r \MWST{\WDep[B]} = D_r \la_2^{i_r-i_1}.\]
The existence of such a constant is given by \cref{Lem:Bound_Cumulants_Markov}.
\end{proof}

\subsection{Subword counts in strings generated by a Markov source}
\label{Subsect:TCL_Markov}
We consider the following pattern matching problem.
Let $u_1,\dots,u_d$ be {\em finite} words on a finite alphabet $S$
of respective lengths $\ell_1,\cdots,\ell_d$.
An occurrence of $\LLL=(u_1,\dots,u_d)$ in $w$ is a factorization
$w=w_0 u_1 w_1 \cdots u_d w_d$, where the $w_i$'s are (possibly
empty) words on the alphabet $S$.
This corresponds to an occurrence of the $u=u_1\cdots u_d$ as a subwords,
where letters from the same $u_i$ are required to be consecutive.

As before, let $(M_k)_{k \ge 0}$ be an aperiodic irreducible Markov chain on $S$,
such that $M_0$ is distributed according to the stationary distribution $\pi$ of the chain.
We are interested in the number $X_N$ of occurrences of $\LLL$
in the random word $W_N=(M_0,\cdots,M_N)$.
\medskip

The position of such an occurrence is a $d$-uple $(i_1,\cdots,i_d)$,
where each $i_j$ is the index of the first letter of $u_j$ in $w$
(in particular, we always have $i_{j+1} \ge i_j +\ell_j$).
Denote $\III$ the set of possible positions of occurrences that is
\[\III=\{(i_1,\cdots,i_d) \in \NN_{\ge 0}^d \text{ such that, for all }
j \le d-1,\, i_{j+1} \ge i_j +\ell_j\}.\]
We also define $\III_N$ as the same set with the additional condition $i_d +\ell_d-1 \le N$.
For $I \in \III$, we denote $Y_I$ the indicator function
of the event ``$W$ has an occurrence of $\LLL$ in position $I$''.
Using the above variables $Y_i^s$, we can write
\[Y_I = \prod_{j=1}^d \left( \prod_{k=1}^{\ell_j} Y_{i_j+k-1}^{(u_j)_k} \right)
\text{ and }
X_N = \sum_{I \in \III_N} Y_I.\]

An estimate for the variance of $X_N$ is given by Bourdon and Vallée 
\cite[Theorem 3]{Bourdon_Vallee:Pattern_Matching}:
\begin{equation}
  \Var(X_N) = \sigma^2(\LLL) N^{2d-1} \left( 1+\O(\tfrac{1}{n}) \right),
  \label{Eq:Var_Estimate_Markov}
\end{equation}
where $\sigma^2(\LLL)$ is an explicit constant
depending on both the pattern $\LLL$ and the transition matrix $P$ of the Markov chain.

Our main result in this section is that the fluctuations 
of order $N^{d-1/2}$ of $X_N$ are Gaussian (possibly degenerate if $\sigma(\LLL)=0$).
\begin{theorem}
  With the above notation, we have the convergence in distribution
  \[ \widetilde{X_N}=\frac{X_N - \esper(X_N)}{N^{d-1/2}} \to \N(0,\sigma(\LLL)).\]
  \label{Thm:TCL_Patterns_Markov_Sources}
\end{theorem}
\begin{proof}
  \cref{PropWDGInMarkov} gives a weighted dependency graph
  for the variables $(Y_i^s)_{i \ge 0, s \in S}$.
  Using \cref{PropProducts}, we get a weighted dependency graph
  for monomials in these variables (with a fixed bound on degrees),
  so in particular for the $(Y_I)_{I \in \III}$.
  The weight of the edge between $Y_I$ and $Y_J$ in this dependency graph
  is $\la_2^{d(I,J)}$ where $d(I,J)$ is the minimal distance between elements
  of the sets
  \[\{i_t+k-1;\, 1 \le t \le d,\, 1\le k \le \ell_t\} \text{ and }
  \{j_t+k-1;\, 1 \le t \le d,\, 1\le k \le \ell_t\}.\]
  It is clear that 
  \[d(I,J)\ge \min_{i \in I,\, j \in J} |j-i| -m,
  \text{ where }m=\max_{1 \le t \le d} \ell_t,\]
  and thus
  \[\la_2^{d(I,J)} \le \la_2^{-m} \max_{i \in I,\, j \in J} \la_2^{|j-i|}.\]
  The corresponding function $\bm{\Psi}$ is simply the constant function equal to $1$.

  Consider the restriction of this weighted dependency graph to $\III_N$.
  Using the notation of \cref{SubsecNormalityCriterion},
  we have $R_N=|\III_N|=\O(N^d)$.
  To find an upper bound for $T_{\ell,N}$, let us fix $I_1$,\ldots, $I_\ell$ 
  and set $I= \bigcup_{j=1}^\ell I_j$.
  Then for $J$ in $\III$, we have 
  \[ W(\{J\},\{I_1,\cdots,I_\ell\}) = \max_{1 \le u \le \ell} \la_2^{d(J,I_u)} 
  \le \la_2^{-m} \, \left( \max_{i \in I \atop j \in J} \la_2^{|j-i|} \right)
  \le \la_2^{-m} \, \sum_{i \in I \atop j \in J} \la_2^{|j-i|}.\]
  Therefore,
  \[\sum_{J \in \III} W(\{J\},\{I_1,\cdots,I_\ell\}) 
  \le \la_2^{-m} \sum_{i \in I} \, \sum_{j=1}^N \, \sum_{J \in \III \atop J \ni j} \la_2^{|j-i|}.\]
  The summand does not depend on $J$, so that the last summation symbol can be replaced with the number
  of sets $J$ in $\III$ containing $j$. This number is smaller than $N^{d-1}$.
  Moreover for a fixed $i$, the sum $\sum_{j=1}^N \la_2^{|j-i|}$ is bounded by
  the constant $\tfrac{2}{1-\la_2}$.
  Finally we have
  \[\sum_{J \in \III} W(\{J\},\{I_1,\cdots,I_\ell\}) \le \la_2^{-m}\, |I|\, \tfrac{2}{1-\la_2}\, N^{d-1}
  \le \frac{2\, \la_2^{-m}\, \ell\, d\, }{1-\la_2} N^{d-1}.\]
  Since this holds for any $I_1$,\ldots, $I_\ell$ in $\III$,
  we have $T_{\ell,N}=\O(N^{d-1})$.

  Using \cref{LemBorneCumulant}, we have
  \begin{multline*}
    \left| \ka_r(\widetilde{X_N}) \right|=\left|\tfrac{1}{N^{r(d-1/2)}} \ka_r(X_N) \right| \\
  \le \tfrac{1}{N^{r(d-1/2)}} R_N T_{1,N} \cdots T_{r-1,N} = 
  \tfrac{1}{N^{r(d-1/2)}} \O(N^{d + (d-1)(r-1)})=\O(N^{-r/2+1}).
\end{multline*}
Therefore cumulants of $\widetilde{X_N}$ of order at least 3 tend to $0$.
On the other hand, its expectation and variance tend to $0$ and $\sigma(\LLL)$
respectively.
This concludes the proof using the method of moments.
\end{proof}

\begin{remark}
  The upper bound in Bourdon and Vallee's estimate \eqref{Eq:Var_Estimate_Markov}
  for the variance of $X_N$ can be obtained from the weighted dependency graph structure
  and \cref{LemBorneCumulant}.
  This upper bound alone implies the concentration result advertised by these authors.
  Note however that their result is proved for more general sources 
  and pattern problems.
\end{remark}

\appendix
\section{Proof of \cref{PropExSCQF}}
\label{Sect:Proof_SCQF_Factorials}
We start by a lemma.
\begin{lemma}\label{LemTechCumulants}
For any nonnegative integers $a_1,\dots,a_{\ell-1}$, the following rational function in $t$
has degree at most $-\ell+1$:
\[R(t) = \prod_{\delta \subseteq [\ell-1]} \left( t - \sum_{j \in \delta} a_j \right)^{(-1)^{|\delta|+1}} - 1. \]
\end{lemma}
\begin{proof}
    This corresponds to \cite[Lemma 2.4]{FerayRandomPermutationsCumulants},
    but we copy the proof for completeness.

Define $R_\ev$ (resp. $R_\odd$) as
\[\prod_\delta \left( t - \sum_{j \in \delta} a_j \right),\]
where the product runs over subsets of $[\ell-1]$ of even (resp. odd) size.
Clearly, $R(t)=\tfrac{R_\odd-R_\ev}{R_\ev}$.
Expanding the product, one gets
\[R_\ev = \sum_{m \geq 0} \ \ \frac{1}{m!} \ 
\sum_{\delta_1,\dots,\delta_m} \ \ \sum_{j_1\in \delta_1,\dots,j_m \in \delta_m}  (-1)^m  
a_{j_1} \dots a_{j_m} t^{2^{\ell-2} - m}.\]
The index set of the second summation symbol is the set of lists of $m$ distinct
(but not necessarily disjoint) subsets of $[\ell-1]$ of even size.
Of course, a similar formula with subsets of odd size holds for $R_\odd$.

Let us fix an integer $m<\ell-1$ and a list $j_1,\dots,j_m$.
Denote $j_0$ the smallest integer in $[\ell-1]$ different from $j_1,\dots,j_m$
(as $m<\ell-1$, such an integer necessarily exists).
Then one has a bijection:
\[ \begin{array}{rcl}
\left\{ \begin{array}{c}
    \text{lists of subsets}\\
\delta_1,\dots,\delta_m \text{ of even size such}\\
\text{that, }\forall\ h\leq m, j_h \in \delta_h
\end{array}\right\}
& \to &
\left\{ \begin{array}{c}
    \text{lists of subsets}\\
\delta_1,\dots,\delta_m \text{ of odd size such}\\
\text{that, }\forall\ h\leq m, j_h \in \delta_h
\end{array}\right\}\\
(\delta_1,\dots,\delta_m)
&\mapsto&
(\delta_1 \nabla \{j_0\},\dots,\delta_m \nabla \{j_0\}),
\end{array}\]
where $\nabla$ is the symmetric difference operator.
Thus the summand 
\hbox{$(-1)^m a_{j_1} \dots a_{j_m} t^{2^{\ell-2} - m}$}
appears as many times
in $R_\ev$ as in $R_\odd$.
Finally, all terms corresponding to values of $m$ smaller than $\ell-1$
cancel in the difference
$R_\odd - R_\ev$ and $R_\odd - R_\ev$ has degree at most $2^{\ell-2} - \ell +1$.
Dividing by $R_\ev$, which has degree $2^{\ell-2}$, this ends the proof.
\end{proof}

We now prove \cref{PropExSCQF}, using the notation defined there.
\begin{proof}[Proof of \cref{PropExSCQF}]
We proceed by induction first on $\ell$, and then on $a_\ell$. 

For $\ell=1$, there is nothing to prove. Consider $\ell>1$ and 
assume that the statement holds for all $\ell'<\ell$.
In particular, for any $\Delta \subsetneq [\ell]$, the subfamily 
\[\big(u^{(n)}_\delta(a_i; i \in \Delta)\big)_{\delta \subseteq \Delta, n \ge n_0}\]
has the $\eps_n$ SC/QF property and
\[ \P_\Delta\big(\uu^{(n)}(a_1,\dots,a_\ell)\big) -1 = \O(X_n^{-|\Delta|+1}).\]
We thus have to prove that 
\begin{equation}
    \P_{[\ell]}\big(\uu^{(n)}(a_1,\dots,a_\ell)\big) -1 = \O(X_n^{-\ell+1}).
    \label{EqToProve2}
\end{equation}

If $a_\ell=0$, then for any $\Delta \subseteq [\ell-1]$,
\[u^{(n)}_\Delta\big(a_1,\dots,a_\ell\big) = u^{(n)}_{\Delta \cup \{\ell\}}\big(a_1,\dots,a_\ell\big),\]
so that $\P_{[\ell]}\big(\uu^{(n)}(a_1,\dots,a_\ell)\big) =1$ and \cref{EqToProve2} trivially holds.

Assume that \cref{EqToProve2} for $a_\ell=k$ and consider the case $a_\ell=k+1$.
Observe that, if $\ell \in \Delta$,
\[ u^{(n)}_\Delta\big( a_1,\dots,a_{\ell-1},k \big)
= \left (X_n - \sum_{i \in \Delta \atop i \ne \ell} a_i -k\right)
u^{(n)}_\Delta\big(a_1,\dots,a_{\ell-1},k+1 \big).\]
On the other hand, if $\ell \notin \Delta$,
\[ u^{(n)}_\Delta\big(a_1,\dots,a_{\ell-1},k \big)  
= u^{(n)}_\Delta\big(a_1,\dots,a_{\ell-1},k+1 \big).\]
Finally,
\[
    \P_{[\ell]}\big(\uu^{(n)}(a_1,\dots,a_{\ell-1},k) \big)
= \prod_{\Delta \subseteq [\ell] \atop \ell \in \Delta} 
\left (X_n - \sum_{i \in \Delta \atop i \ne \ell} a_i -k\right)^{(-1)^{|\Delta|}} \!
\cdot \P_{[\ell]}\big(\uu^{(n)}(a_1,\dots,a_{\ell-1},k+1) \big).
\]
Subsets $\Delta$ of $[\ell]$ that contains $\ell$ are in trivial bijection with
subsets of $[\ell]-1$, so that the product above correspond to the rational function $R(t)+1$
from \cref{LemTechCumulants}, evaluated in $X_n-k$.
Thus it is $1+\O(X_n^{-\ell+1})$.
By induction hypothesis
\[ \P_{[\ell]}\big(\uu^{(n)}(a_1,\dots,a_{\ell-1},k) \big) = 1+\O(X_n^{-\ell+1}) \]
and thus
\[ \P_{[\ell]}\big(\uu^{(n)}(a_1,\dots,a_{\ell-1},k+1) \big) = 1+\O(X_n^{-\ell+1}), \]
which ends the proof.
\end{proof}
\begin{remark}
    The exact same proof works for the family
\[v^{(n)}_\Delta(a_1,\cdots,a_\ell)= \left(X_n+\textstyle \sum_{i \in \Delta} a_i \right)!\] 
\end{remark}
\section{Variance computations}
\subsection{Crossings in random pair partitions}
\label{AppVarPP}
The goal of this section is to compute (asymptotically)
the variance of $\Cr_n$, the number of crossings             
    in a uniform random pair partitions of $[2n]$.
We first establish a polynomiality result for it.

\begin{lemma}
    The quantity
    \[(2n-1)^2 \, (2n-3)^2\, (2n-5) \, (2n-7) \, \Var(\Cr_n) \]
    is a polynomial in $n$ of degree at most $9$.
    \label{LemVarPoly}
\end{lemma}
\begin{proof}
    We use the decomposition
    \[\Var(\Cr_n)=\sum_{i_1<j_1<k_1<l_1 \atop i_2<j_2<k_2<l_2} 
       \Cov(Y'_{i_1,j_1,k_1,l_1}, Y'_{i_2,j_2,k_2,l_2}). \]
    We split the sum depending on which summation indices are equal.
    For a given set of equalities 
    ({\em e.g.} $i_1=j_2$ and $l_1=l_2$, but all other indices are distinct),
    the covariance is always the same and the corresponding number of terms
    is a polynomial in $n$.
    
    Moreover, from the discussion of \cref{SubsectDefPP} on the probability
    that a random pair partitions contains a given set of pairs,
    we see that 
    \[(2n-1)^2 \, (2n-3)^2\, (2n-5) \, (2n-7) \, \Cov(Y'_{i_1,j_1,k_1,l_1}, Y'_{i_2,j_2,k_2,l_2})\]
    is always a polynomial in $n$.
    This proves that
    \[(2n-1)^2 \, (2n-3)^2\, (2n-5) \, (2n-7) \,\Var(\Cr_n) \]
    is a polynomial in $n$, as claimed.

    Besides, from \cref{LemBorneCumulant}, we know that $\Var(\Cr_n)=\O(n^3)$
    (recall from the proof of \cref{ThmAsympNormPP} that,
    in this case, $R_n= \O(n^2)$ and $T_{1,n}=\O(n)$).
    Therefore the degree of the above polynomial is at most $9$.
\end{proof}

A polynomial of degree at most $9$ can be determined by polynomial interpolation
from its values on the set $\{0,\cdots,9\}$.
But $\Var(\Cr_n)$ can be easily computed with the help of a computer algebra software
for small values of $n$.
We performed this computation using sage \cite{sage}.
The code has been embedded in the pdf file for interested readers of the electronic version.
We obtain the following result.
\begin{proposition}
    Let $\Cr_n$ be the number of crossings
    in a uniform random pair partition of $[2n]$.
    We have
    \[\Var(\Cr_n)= \frac{n(n-1)(n-3)}{45}.\]
\end{proposition}
We refer to \cite[Theorem 3]{FlajoletNoy:CLT_Crossings} for another proof of this result,
which also explains the polynomiality in $n$,
but relies on the ``remarkable exact formula'' for the generating series of crossings.

\subsection{Subgraph counts: proof of \cref{EqAympVarianceGnm}}
\label{AppVarGnm}
We first write.
\[\Var(X^H_n) = \sum_{H'_1,H'_2 \in A^H_n}
\Cov(Y_{H'_1},Y_{H'_2})\]
Observe that, if $H'_1 \cap H'_2 \simeq K$, then
\begin{equation}
    \Cov(Y_{H'_1},Y_{H'_2}) = \frac{(m_n)_{2e_H-e_K}}{(E_n)_{2e_H-e_K}}-
\left( \frac{(m_n)_{e_H}}{(E_n)_{e_H}} \right)^2.
\label{EqCovK}
\end{equation}
Unlike in the $G(n,p)$ model, this covariance can be negative.
More precisely, it is negative if and only if the copies $H'_1$ and $H'_2$ of $H$
are edge-disjoint.
The total contribution of such pairs is given
in the following lemma.
\begin{lemma}
    One has
    \[
        \sum_{H'_1,H'_2 \in A^H_n \atop E_{H'_1} \cap E_{H'_2} =\emptyset}
    \Cov(Y_{H'_1},Y_{H'_2})
    = - \frac{2 e_H^2}{\Aut(H)^2} (n)_{v_H} (n-2)_{v_H-2} \, p_n^{2e_H-1} (1-p_n) 
    + O\big[n^{2v_H-4} p_n^{2e_H-2}\big] .
    \]
    \label{LemTCovNoEdges}
\end{lemma}
\begin{proof}
    Consider two edge-disjoint copies $H'_1$ and $H'_2$ of $H$
    and let us look at \cref{EqCovK} in this case.
    We have
    \begin{multline*}
        p_n^{-2e_H} \frac{(m_n)_{2e_H}}{(E_n)_{2e_H}}
        = \frac{\prod_{i=0}^{2e_H-1} \left( 1-\frac{i}{m_n} \right)}
        {\prod_{i=0}^{2e_H-1} \left( 1-\frac{i}{E_n} \right)} 
        =\prod_{i=0}^{2e_H-1} \left( 1-i\left( \frac{1}{m_n} - \frac{1}{E_n} \right)\right) +\O(m_n^{-2})\\
        = 1 - \frac{2e_H(2e_H-1)}{2} \frac{1}{m_n} (1 - p_n) + \O(m_n^{-2}).
    \end{multline*}
    Similarly,
    \[  p_n^{-2e_H} \left( \frac{(m_n)_{e_H}}{(E_n)_{e_H}} \right)^2
    = 1 - 2 \frac{e_H(e_H-1)}{2} \frac{1}{m_n} (1 - p_n) + \O(m_n^{-2}).\]
    Putting both equations together, we get
    \[ \Cov(Y_{H'_1},Y_{H'_2}) = - p_n^{2e_H} e_H^2 \frac{1}{m_n} (1 - p_n) + \O( p_n^{2e_H} m_n^{-2} ).\] 

    On the other hand, we claim that the total number of pairs $(H'_1,H'_2)$ of copies of $H$
    that do not share an edge is asymptotically $(n)_{v_H}^2/\Aut(H)^2 (1+\O(n^{-2}))$.
    Indeed, $(n)_{v_H}^2/\Aut(H)^2$ is the number of pairs $(H'_1,H'_2)$ of copies of $H$.
    If we think at such a pair being taken independently uniformly at random,
    the vertex sets of $H'_1$ and $H'_2$ are independent uniform random $v_H$-element subsets of $[n]$
    and the probability that they have at least two vertices in common in $\O(n^{-2})$.
    This explains the above claim.

    Bringing both estimates together, we get:
    \begin{multline*}
        \sum_{H'_1,H'_2 \in A^H_n \atop E_{H'_1} \cap E_{H'_2} = \emptyset}
        \Cov(Y_{H'_1},Y_{H'_2})
        = - \frac{e_H^2}{\Aut(H)^2} \frac{(n)^2_{v_H} p_n^{2e_H}}{m_n} (1-p_n) \\
        +  O\left(\frac{n^{2v_H-2} p_n^{2e_H} (1-p_n)}{m_n}\right) + 
        O\left(\frac{n^{2v_H} p_n^{2e_H}}{m_n^2} \right).
    \end{multline*}
    Substituting $m_n=\tfrac{p_n \, n(n-1)}{2}$ and observing that the second error term
    is bigger than the first complete the proof.
\end{proof}

Consider now pairs
$(H'_1,H'_2)$ with a non-trivial edge intersection.
Denote by $e_K$ the number of edges in the intersection
of $H'_1$ and $H'_2$.

Consider the expression of $\Cov(Y_{H'_1},Y_{H'_2})$ given in \cref{EqCovK}.
A straightforward computation gives:
    \begin{multline}
        \Cov(Y_{H'_1},Y_{H'_2})=p_n^{2e_H-e_K} \left[ 1- p_n^{e_K} + O\left(\frac{1}{m_n}\right) \right]\\
        = p_n^{2e_H-e_K}(1- p_n^{e_K})\bigg\{ 1+O\big[ m_n^{-1}(1-p_n)^{-1} \big] \bigg\}.
        \label{EqTechCovK}
    \end{multline}
We use the easy inequality (for $0 \le p \le 1$ and $e$ positive integer)
\[\frac{1-p^e}{1-p}=1 + \cdots + p^{e-1} \ge e p^{e-1} + (1-p) \delta_{e>1},\]
so that the estimate \eqref{EqTechCovK} gives
\begin{multline*}
    \Cov(Y_{H'_1},Y_{H'_2}) \ge e_K\, p_n^{2e_H-1}(1-p_n) \bigg\{ 1+O\big[ m_n^{-1}(1-p_n)^{-1} \big] \bigg\}\\
+ \delta_{e_K>1} \, p_n^{2e_H-e_K}(1-p_n)^2 \bigg\{ 1+O\big[ m_n^{-1}(1-p_n)^{-1} \big] \bigg\}.
\end{multline*}
Call $A_{H'_1,H'_2}$ and $B_{H'_1,H'_2}$ the first and second term in the right-hand side.
\begin{lemma}
    One has
    \begin{equation*}
        \sum_{H'_1,H'_2 \in A^H_n \atop E_{H'_1} \cap E_{H'_2} \neq \emptyset}  A_{H'_1,H'_2}
        = \frac{2 e_H^2}{\Aut(H)^2} (n)_{v_H} (n-2)_{v_H-2} p_n^{2e_H-1} (1-p_n) \\
            + O\big[n^{2v_H-4} p_n^{2e_H-2}\big].
    \end{equation*}
    \label{LemSumA}
\end{lemma}
\begin{proof}
    From the definition,
    \begin{equation}
        \sum_{H'_1,H'_2 \in A^H_n \atop E_{H'_1} \cap E_{H'_2} \neq \emptyset}  A_{H'_1,H'_2}
    = \left(\sum_{H'_1,H'_2 \in A^H_n \atop E_{H'_1} \cap E_{H'_2} \neq \emptyset} e_K \right)
    p_n^{2e_H-1}(1-p_n) \bigg\{ 1+O\big[ m_n^{-1}(1-p_n)^{-1} \big] \bigg\}.
    \label{EqSumA}
\end{equation}
    The parenthesis counts the number of pairs $(H'_1,H'_2)$ of copies of $H$
    with a marked common edge.
    Such a pair can be constructed as follows:
    choose successively 
    \begin{enumerate}
        \item a list of vertices $v_1,\cdots,v_k$ for the vertices of $H'_1$
            ($(n)_{v_H}$ choices),
        \item the edge of $H'_1$ that will be the marked common edge ($e_H$ choices),
        \item the edge of $H'_2$  that will be the marked common edge ($e_H$ choices);
            this determines up to a switch (2 choices) two vertices of $H'_2$.
        \item choose a list of vertices $w_1,\cdots,w_{k-2}$ for the other vertices
            of $H'_2$, which does not contain the extremities of the marked common edges
             ($(n-2)_{v_H-2}$ choices).
    \end{enumerate} 
    Doing so, we construct $\Aut(H)^2$ times each pair $(H'_1,H'_2)$.
    Thus we have
    \[\left(\sum_{H'_1,H'_2 \in A^H_n \atop E_{H'_1} \cap E_{H'_2} \neq \emptyset} e_K \right)
    = \frac{2 e_H^2}{\Aut(H)^2} (n)_{v_H} (n-2)_{v_H-2}. \]
    We plug this in \cref{EqSumA} and expand the remainder
    (recall that $m_n \asymp n^2 p_n$) to get the statement in the lemma.
\end{proof}

The last lemma estimates the sum of $B_{H'_1,H'_2}$.
\begin{lemma}
    One has
    \[\sum_{H'_1,H'_2 \in A^H_n \atop \big|E_{H'_1} \cap E_{H'_2}\big|>1}
    B_{H'_1,H'_2} \asymp \frac{(n^{v_H} \, p_n^{e_H})^2}{\tPhi_H}\, (1-p_n)^2 .\]
    \label{LemSumB}
\end{lemma}
\begin{proof}
    From the definition of $B_{H'_1,H'_2}$, we have:
    \[\sum_{H'_1,H'_2 \in A^H_n \atop \big|E_{H'_1} \cap E_{H'_2}\big|>1}                      
        B_{H'_1,H'_2} \asymp
        \sum_{H'_1,H'_2 \in A^H_n \atop \big|E_{H'_1} \cap E_{H'_2}\big|>1}                      
        p_n^{2e_H-e_K}(1-p_n)^2 .\]
    We split the sum depending on isomorphy type $K$ of the intersection $K$
    of $H'_1$ and $H'_2$ and we get
    \[\sum_{H'_1,H'_2 \in A^H_n \atop \big|E_{H'_1} \cap E_{H'_2}\big|>1}                      
        p_n^{2e_H-e_K} (1-p_n)^2
        = \sum_{K \subseteq H \atop e_K>1}
        N_K\, p_n^{2e_H-e_K} (1-p_n)^2,\]
where $N_K$ is the number of pairs $(H'_1,H'_2)$ with intersection isomorphic to $K$.
Note that the summation index does not depend on $n$.
Furthermore, all summands are nonnegative, thus
the order of magnitude of the sum is simply
the maximum of the orders of magnitude of the summands.
It is easy to see that $N_K \asymp n^{2v_H-v_K}$: see, {\em e.g.}, \cite[Proof of Lemma 3.5]{JansonRandomGraphs}.
\[
    \sum_{K \subseteq H \atop e_K>1}             
        N_K\, p_n^{2e_H-e_K} (1-p_n)^2 \asymp
        \max_{K \subseteq H \atop e_K>1} n^{2v_H-v_K} p_n^{2e_H-e_K} (1-p_n)^2 
        =\frac{(n^{v_H} \, p_n^{e_H})^2}{\tPhi_H}\, (1-p_n)^2.
        \]
        This completes the proof of the lemma.
\end{proof}

\begin{proof}
    [Proof of \cref{EqAympVarianceGnm}]
    The variance $\Var(X^H_n)$ is bounded from below by the sum of
    the three terms considered in \cref{LemTCovNoEdges,LemSumA,LemSumB}.
    Note that the main terms in the estimates of \cref{LemTCovNoEdges,LemSumA}
    cancel each other.
    Besides the error term in these lemmas are smaller than 
    the main term in \cref{LemSumB}.
    Indeed, using $\tPhi_H \le n^3 p_n^2$ and $n (1-p_n)^2 \gg 1$, we have:
    \[
        \frac{(n^{v_H} \, p_n^{e_H})^2}{\tPhi_H}\, (1-p_n)^2 
        \ge \frac{(n^{v_H} \, p_n^{e_H})^2}{n^3\, p_n^2} (1-p_n)^2
        \gg \frac{(n^{v_H} \, p_n^{e_H})^2}{n^4\, p_n^2}.
        \]
    Therefore $\Var(X^H_n)$ is asymptotically at least of order 
    $\tfrac{(n^{v_H} \, p_n^{e_H})^2}{\tPhi_H}\, (1-p_n)^2$,
    as claimed.
\end{proof}

\section{\for{toc}{Tightness of piecewise-affine random functions} 
\except{toc}{Moment inequalities and\texorpdfstring{\\}{}
tightness of piecewise-affine random functions}}
\label{AppOnlyOnLattice}
The goal of this last appendix section is to establish the following:
for piecewise-affine random functions,
tightness can be inferred from moment inequalities
{\em for points of the mesh}.
We start by a trivial lemma.
\begin{lemma}
    \label{LemIneqNormEq}
    For any $a>1$, the exists a constant $C_a$ such that,
    \[\text{for all }\ x,y,z\ge 0, \text{ one has } (x+y+z)^a \le C_a (x^{a} + y^{a}+z^{a}).\]
\end{lemma}
\begin{proof}
    This comes from the equivalence of the following norms on $\RR^3$:
    \[(x,y,z) \mapsto \big(|x| + |y| +|z| \big) \text{ and }
    (x,y,z) \mapsto \big(|x|^{a} + |y|^{a} +|z|^{a} \big)^{1/a}.
    \qedhere \]
\end{proof}

\subsection{One-dimensional case}
The following lemma can be found in unpublished lecture notes of Marckert.
\begin{lemma}
    Consider a sequence $(X_n)$ of random elements in $C[0,1]$.
    Assume that for each $n$, almost surely $X_n$ is affine on each segment
    $[j/n, (j+1)/n]$ (for $0 \le j \le n-1$) and
    that there exists positive constants $a$, $b$ and $\la$ with \bm{$a \ge 1+b$} such that
\begin{equation}
    \esper \big[|{X_n}(s)- {X_n}(t)|^a \big] \le \la \, |s-t|^{1+b}, 
    \label{EqTightnessCriterion1DApp}
\end{equation}
as soon as $ns$ and $nt$ are integers ($n \ge 1$ and $s,t \in [0,1]$).

Then \eqref{EqTightnessCriterion1DApp} holds as well for any $s$ and $t$ in $[0,1]$
with the same exponents $a$ and $b$ but a different constant $\la'$ instead of $\la$.

As a consequence, if moreover $X_n(0)$ is tight, then the sequence $X_n$ is also tight.
    \label{LemOnlyOnLattice1D}
\end{lemma}
\begin{proof}
Let $n \ge 1$ and $s$ and $t$ in $[0,1]$ with $t<s$. We distinguish two cases.
\begin{itemize}
    \item If $s$ and $t$ belong to the same segment $[j/n, (j+1)/n]$ (for some $0 \le j \le n-1$),
        then, since $X_n$ is affine on $[j/n, (j+1)/n]$, one has
        \[{X_n}(s)- {X_n}(t) = n(s-t) \, \bigg[{X_n}\big( (j+1)/n \big)- {X_n}(j/n)\bigg].\]
        The $a$-th moment of the right-hand side can be bounded by \eqref{EqTightnessCriterion1DApp}, so that
        \[
            \esper\big[|{X_n}(s)- {X_n}(t)|^a\big] \le
            n^a(s-t)^a \, \la \, (1/n)^{1+b} = \la \, (s-t)^{1+b} \left(n(s-t)\right)^{a-1-b} \le \la \, (s-t)^{1+b}.
        \]
        The last inequality comes from the fact that $n(s-t) \le 1$ and $a-1-b \ge 0$.
    \item Consider now the case where $nt \in [j/n, (j+1)/n]$ and $ns \in [k/n, (k+1)/n]$ with $j<k$.
        Then set $s'=k/n$ and $t'=(j+1)/n$ so that
        \begin{itemize}
            \item $t \le t' \le s' \le s$;
            \item $ns'$ and $nt'$ are integers;
            \item $s$ and $s'$ belong to $[k/n, (k+1)/n]$; and, 
            \item $t$ and $t'$ belong to $[j/n, (j+1)/n]$.
        \end{itemize}
        Then, using \cref{LemIneqNormEq}, \cref{EqTightnessCriterion1DApp} and the first part of the proof,
        we get 
        \begin{multline}
            \label{EqTech5}
            \hspace{-3mm} \esper\big[|{X_n}(s)- {X_n}(t)|^a\big] \le
        C_a \left( \esper\big[|{X_n}(s)- {X_n}(s')|^a\big]
        + \esper\big[|{X_n}(s')- {X_n}(t')|^a\big]
        + \esper\big[|{X_n}(t')- {X_n}(t)|^a\big] \right) \\
        \le C_a \left( \la (s-s')^{1+b} + \la (s'-t')^{1+b} + \la (t'-t)^{1+b} \right)
        \le C_a \la (s-t)^{1+b}.
    \end{multline}
\end{itemize}
This proves that \eqref{EqTightnessCriterion1DApp} holds for any $s$ and $t$ in $[0,1]$
and $\la'=C_a \la$.
The tightness assertion then follows from \cite[Corollary 16.9 for $d=1$]{KallenbergBookProba}.
\end{proof}

\subsection{Two-dimensional case}
We now state and prove a two-dimensional analogue of the previous lemma.
\begin{lemma}
    Consider a sequence $(X_n)$ of random elements in $C[0,1]^2$.
    Assume that for each $n$, almost surely $X_n$ is affine on each square
    $[i/n,(i+1)/n] \times [j/n, (j+1)/n]$ (for $0 \le i,j \le n-1$) and
    that there exists positive constants $a$, $b$ and $\la$ with \bm{$a \ge 2+b$} such that
\begin{equation}
    \esper \big[|{X_n}(s_1,s_2)- {X_n}(t_1,t_2)|^a \big] \le \la \, (|s_1-t_1|+|s_2-t_2|)^{2+b}, 
    \label{EqTightnessCriterion2DApp}
\end{equation}
as soon as $ns_1$, $ns_2$, $nt_1$ and $nt_2$ are integers ($n \ge 1$ and $s_1,s_2,t_1,t_2 \in [0,1]$).

Then \eqref{EqTightnessCriterion2DApp} holds as well for any $s_1,s_2,t_1,t_2$ in $[0,1]$
with the same exponents $a$ and $b$ but a different constant $\la'$ instead of $\la$.

As a consequence, if moreover $X_n(0,0)$ is tight, then the sequence $X_n$ is also tight.
    \label{LemOnlyOnLattice2D}
\end{lemma}
\begin{proof}
Let $n \ge 1$ and $\bs=(s_1,s_2)$ and $\bt=(t_1,t_2)$ in $[0,1]^2$
and let us prove \eqref{EqTightnessCriterion2DApp}. 
We write $|\bs-\bt|:=|s_1-t_1|+|s_2-t_2|$ and distinguish three cases.
\begin{itemize}
    \item If $\bs$ and $\bt$ belong to the same square $[i/n,(i+1)/n] \times [j/n, (j+1)/n]$ (for some $0 \le i,j \le n-1$),
        then, since $X_n$ is affine on this square, we have
        \begin{multline*}
            {X_n}(\bs)- {X_n}(\bt)=n(s_1-t_1)\left[ {X_n}\left( (i+1)/n,j/n \right) - {X_n}\left( i/n,j/n \right)\right] \\
        + n\left( s_2-t_2) \right) \left[ {X_n}\left( i/n,(j+1)/n \right) - {X_n}\left( i/n,j/n \right) \right].
    \end{multline*}
   Therefore we have
        \begin{multline*}
            \esper\big[|{X_n}(\bs)- {X_n}(\bt)|^a\big] \le
            C_a n^a \left( |s_1-t_1|^a + |s_2 -t_2|^a \right) \, \la \, (1/n)^{2+b} \\
            \le C_a \la \, (|\bs-\bt|)^{2+b} \left(n(|\bs-\bt|)\right)^{a-2-b} \le 2^{a-2-b} C_a \la \, (|\bs-\bt|)^{2+b}.
        \end{multline*}
    \item If the segment $[\bs,\bt]$ crosses at most two lines of the grid, call $\bu$ and $\bv$ the intersection points,
        so that $|\bs-\bt|= |\bs-\bu|+|\bu-\bv|+|\bv-\bt|$.
        Since $\bs$ and $\bu$ (respectively, $\bu$ and $\bv$ and $\bv$ and $\bt$) lie in the same square,
        we can apply the first case to bound $\esper\big[|{X_n}(\bs)- {X_n}(\bu)|^a\big]$
        (respectively, $\esper\big[|{X_n}(\bu)- {X_n}(\bv)|^a\big]$ and $\esper\big[|{X_n}(\bv)- {X_n}(\bt)|^a\big]$).
        Then the same computation as in \eqref{EqTech5} shows that \eqref{EqTightnessCriterion2DApp}
        holds in this case.
    \item If the segment $[\bs,\bt]$ crosses more than two lines, it crosses two lines in the same direction,
        which implies that $|\bs-\bt| \ge 1$. Call $\bu$ and $\bv$ the North-East corners of the squares containing
        $\bs$ and $\bt$, respectively.
        Then $|\bs-\bu| \le 2 \le 2 |\bs-\bt|$. Similarly, $|\bv-\bt| \le 2 |\bs-\bt|$.
        By the triangular inequality $|\bu-\bv| \le |\bs-\bu| + |\bs-\bt| + |\bv-\bt| \le 5 |\bs-\bt|$.
        Bringing everything together, one has $9|\bs-\bt| \ge |\bs-\bu|+|\bu-\bv|+|\bv-\bt|$.
        Now, the same computation as in \eqref{EqTech5} proves \eqref{EqTightnessCriterion2DApp}. 
\end{itemize}
We conclude that \eqref{EqTightnessCriterion2DApp} holds for any $\bs$, $\bt$ in $[0,1]^2$.
The tightness assertion then follows from \cite[Corollary 16.9 for $d=2$]{KallenbergBookProba}.
\end{proof}
This lemma is easily generalized to any dimension,
though we do not need such a generalization in this paper.

\section*{Acknowledgements}
This work has benefited at several levels from discussions with many people
and the author would like to thank them warmly:
Jehanne Dousse, Jean-François Marckert, Pierre-Loïc Méliot, 
Ashkan Nighekbali, Piotr Śniady and Marko Thiel.

\bibliographystyle{abbrv}
\bibliography{../../courant.bib}

\end{document}